\documentclass[11pt,a4paper,twoside,reqno,centertags]{article}

 \usepackage{amsmath,amsthm,amsfonts,amssymb}
\usepackage[mathscr]{euscript}
\usepackage[bookmarks=false]{hyperref}
\usepackage{graphicx}

\usepackage{amsfonts,amsmath,amsthm}
\usepackage{amstext}

\graphicspath{{figures/}}

\newtheorem{theorem}{Theorem}[section]
\newtheorem{lemma}[theorem]{Lemma}
\newtheorem{definition}[theorem]{Definition}

\newenvironment{remark}%
  {\par\medbreak\refstepcounter{theorem}%
    \noindent\textbf{Remark~\thetheorem. }}%
  {\par\medskip}

\newcommand{\mr}[1]{\ensuremath{\mathrm{#1}}}

\newcommand{\vz}[1]{\ensuremath{\mathbb{#1}}}
\newcommand{\id}{\ensuremath{\bm{\mr{\mathbb{I}}}}}

\newcommand{\R}{{\vz R}}
\newcommand{\N}{{\vz N}}

\newcommand{\Z}{{\vz Z}}
\newcommand{\C}{{\vz C}}
\newcommand{\T}{{\vz T}}

\newcommand{\dvg}{\text{div}\,}

\newcommand{\TV}{\text{TV}}
\newcommand{\TVa}{\text{TV}_{\text{a}}}
\newcommand{\sgn}{\text{sgn}}
\DeclareMathOperator{\supp}{supp}

\long\def\drop#1{}

\let\e\varepsilon
\let\epsilon\varepsilon

\def\XXint#1#2#3{{\setbox0=\hbox{$#1{#2#3}{\int}$}
     \vcenter{\hbox{$#2#3$}}\kern-.5\wd0}}

\usepackage{geometry}                
\usepackage{amssymb}
\usepackage{latexsym}
\usepackage{amssymb}
\usepackage{amsfonts}
\usepackage{bm}
\usepackage{epsfig}
\usepackage{bbm}
\usepackage{amsmath}

\usepackage{psfrag}
\usepackage{color}
\usepackage{multicol}
\usepackage{subfig}
\usepackage{appendix}
\usepackage{relsize}

\begin{document}
 
\title{ $\Gamma$-convergence of graph Ginzburg-Landau functionals}
\date{\today}
\maketitle     
 
\vspace{ -1\baselineskip}

{\small
\begin{center}
 {\sc Yves van Gennip, Andrea L. Bertozzi}\\
 {\sc Department of Mathematics}\\
 {\sc University of California Los Angeles}\\
 {\sc Los Angeles, CA 90095, USA}\\
 {\sc \texttt{yvgennip@math.ucla.edu} \qquad \texttt{bertozzi@math.ucla.edu}}\\[10pt]
\end{center}
}

\numberwithin{equation}{section}
\allowdisplaybreaks

 \smallskip

 \begin{quote}
\footnotesize
{\bf Abstract.}
We study $\Gamma$-convergence of graph based 
Ginzburg-Landau functionals, both the limit for
 zero diffusive interface parameter $\e \to 0$ 
 and the limit for infinite nodes in the graph 
 $m \to \infty$. For general graphs we prove that
  in the limit $\e \to 0$ the graph cut objective
   function is recovered. We show that the continuum 
 limit of this objective function on 4-regular 
 graphs is related to the total variation seminorm and
  compare it with the limit of the discretized 
  Ginzburg-Landau functional. For both functionals we 
 also study the simultaneous limit $\e \to 0$ 
 and $m \to \infty$, by expressing $\e$ as a power of 
 $m$ and taking $m \to \infty$. Finally we 
 investigate the continuum limit for a nonlocal means type 
 functional on a completely connected graph.
\end{quote}
\hspace{0.94cm}\textbf{AMS Subject Classifications:} 35R02, 35Q56

\section{Introduction}

\subsection{The continuum Ginzburg-Landau
functional}\label{sec:continuumGL}

In this paper we study an adaptation of the classical real
Ginzburg-Landau (also called Allen-Cahn) functional to graphs.  The
Ginzburg-Landau functional is the object to be minimized\footnote{Note
that to avoid trivial minimizers an additional constraint needs to be
added.  In materials science it is common to add a mass constraint of
the form $\int_\Omega u = M$ for a fixed $M>0$.  In image analysis
applications one often adds a fidelity term of the form $\lambda
\|u-f\|_{L^2(\Omega)}^2$ to the functional $F_\e^{GL}$, where
$\lambda>0$ is a parameter and $f\in L^2(\Omega)$ is given data, often
a noisy image which needs to be cleaned up,
\cite{RudinOsherFatemi92}.} in a well known phase field model for
phase separation in materials science, \textit{e.g.}
\cite{Modica87a,Modica87b} and is given by
\begin{equation}\label{eq:GinzburgLandau}
F_\e^{GL}(u) := \e \int_\Omega |\nabla u(x)|^2 \,dx + \frac1\e
\int_\Omega W(u(x))\, dx, \qquad \e>0,
\end{equation}
where $u\in W^{1,2}(\Omega)$ is the phase field describing the
different phases the material can be in and $W$ is a double well
potential with two minima, \textit{e.g.} $W(s) = s^2 (s-1)^2$.
$\Omega$ is a bounded domain in $\R^N$.

Recently \cite{BertozziFlenner12} this functional has been adapted to
weighted graphs in an application to machine learning and data
clustering: An image is interpreted as a weighted graph, with the
vertices corresponding to the pixels and the weights based on the
similarities between the pixels' neighborhoods.  The phase separating
nature of the Ginzburg-Landau functional then drives separation of the
different features in the image.

The continuum functional $F_\e^{GL}$ has been extensively used and
studied, but a theoretical understanding of its equivalent on graphs
is lacking.  In this paper we use $\Gamma$-convergence
\cite{DalMaso93,Braides02} to study the asymptotic behavior of
minimizers of the graph based Ginzburg-Landau functional when either
$\e\to 0$ or the number of nodes in the graph $m \to \infty$.  In
Section~\ref{sec:explainGamma} we discuss $\Gamma$-convergence in more
detail.  Its most important feature is that if a sequence of functions
$\{f_n\}_{n=1}^\infty$ $\Gamma$-converges to a limit function
$f_\infty$ and in addition satisfies a specific compactness condition,
then minimizers of $f_n$ converge to minimizers of $f_\infty$.

It has been proven
\cite{ModicaMortola77,Modica87a,Modica87b}\footnote{As poster child
for $\Gamma$-convergence the proof has been reproduced, clarified, and
extended upon in various ways, see \textit{e.g.}
\cite{Baldo90,Sternberg88,FonsecaTartar89,KohnSternberg89,BarrosoFonseca94,Alberti01a,Alberti01b,ContiFonsecaLeoni02,AlbertiBaldoOrlandi05}},
that $F_\e^{GL}$ $\Gamma$-converges as $\e \to 0$ to the total
variation functional
\begin{equation}\label{eq:surfacetension}
F_0^{GL}(u) := \sigma(W) \int_\Omega |\nabla u|
\end{equation}
where now $u$ is restricted to functions of bounded variation taking
on two values (corresponding to the minima of the potential $W$)
almost everywhere and the surface tension coefficient $\sigma(W)$ is
determined by the potential $W$ (see Section~\ref{sec:proofk} for more
details).  Because the total variation of a binary function is
proportional to the length of the boundary between the regions where
the function takes on different values, from this limit functional the
phase separating behavior can be seen clearly: $u$ takes on one of two
values, corresponding to the two different phases of the material and
by minimizing the BV seminorm of $u$ the interface between the two
phases gets minimized.

One of the results in this paper is a similar $\Gamma$-convergence
statement for the graph Ginzburg-Landau functional:
\begin{equation}\label{eq:graphGLfirstglance}
f_\e(u) := \chi \sum_{i,j=1}^m \omega_{ij} (u_i-u_j)^2 + \frac1\e
\sum_{i=1}^m W(u_i),
\end{equation}
where $u_i$ is the value of $u$ on node $i$, $\omega_{ij}$ the weight
of the edge connecting nodes $i$ and $j$, $m$ is the number of nodes
in the graph, $\e>0$ and $\chi\in (0,\infty)$ is a constant
independent of $\e$ and $m$, usually chosen to be $\chi=\frac12$ so
the first summation is the analogue of $\int |\nabla u|^2$ (see
Section~\ref{sec:LaplDirTV}).  The different terms in this functional
and its scaling will be explained below.

The Euler-Lagrange equations for this functional are a nonlinear
extension of the graph heat equation using the graph Laplacian
\cite{Chung97}.  Nonlinear elliptic equations on graphs were investigated
in \cite{Neuberger06} and recently in \cite{ManfrediObermanSviridov12} 
their well-posedness was studied.

We study not only the limit $\e\to 0$ in analogy with the classical
continuum result, but also investigate the limit $m\to\infty$.  For a
graph embedded in $\R^n$ this can be interpreted as the limit for
finer discretization or sampling scale.  In order to make sense of
this limiting process we need to assume some additional structure on
the graph that tells us how nodes are added along a sequence of
increasing $m$.  In this paper we consider 4-regular graphs
(\textit{i.e.,} each node is connected to exactly 4 edges) with
uniformly weighted edges in Sections~\ref{sec:differentscalingshNe}
and~\ref{sec:differentscalingskNe}, and a completely connected graph
for the nonlocal means functional as studied in Section~\ref{sec:NLM},
but it is an interesting question if and how this can be extended to
different types of graphs.  Adaptation of our results to a 2-regular
graph is fairly direct, but it is not clear at this moment how to
extend our method to other graphs, even regular ones.

\subsection{Different scalings on a 4-regular graph}

The formulation of $f_\e$ in (\ref{eq:graphGLfirstglance}) does not
require the graph to be embedded in a surrounding space, although an
embedding may exist as in the case of the 4-regular graph considered
as an $N\times N$ square grid on the flat torus $\T^2$.

We study two natural scalings for the functional on this 4-regular
graph.  The first is a direct reformulation of the graph functional
$f_\e$ from (\ref{eq:graphGLfirstglance}) with $\chi=\frac12$ and
weights equal to $N^{-1}$ on all existing edges and zero between two
vertices that are not connected by an edge:
\begin{equation}\label{eq:hNe}
h_{N,\e}(u) := N^{-1} \sum_{i, j=1}^{N}
(u_{i+1,j}-u_{i,j})^2+(u_{i,j+1}-u_{i,j})^2 + \e^{-1} \sum_{i,j=1}^N
W(u_{i,j}).
\end{equation}
The second we get from discretizing the functional $F_\e^{GL}$ on the
square grid using a forward finite difference scheme for the gradient
and the trapezoidal rule for the integrals:
\begin{equation}\label{eq:kNe}
k_{N,\e}(u) := \e \sum_{i, j=1}^N
(u_{i+1,j}-u_{i,j})^2+(u_{i,j+1}-u_{i,j})^2 + \e^{-1} N^{-2}
\sum_{i,j=1}^N W(u_{i,j}).
\end{equation}
The subscripts in $u_{i,j}$ denote the horizontal and vertical
coordinates along the square grid.

We will consider $\Gamma$-limits of these functionals when $\e\to 0$
and $N\to \infty$ sequentially.  We also prove results in the case
where we set $\e=N^{-\alpha}$ for $\alpha>0$ in a specified range and
take $N\to\infty$.  Based on the $\Gamma$-convergence result in the
continuum case we expect $h_{N,\e}$ and $k_{N,\e}$ to converge to
total variation functionals.  This intuition turns out to be correct,
but with a twist: $k_{N,\e}$ converges to the total variation
functional $\int_{\T^2} |\nabla u|$, but $h_{N,\e}$ converges to the
\emph{anisotropic} total variation $\int_{\T^2} |u_x|+|u_y|$.  It
picks up the directionality of the grid.  Precise results are stated
and proved in Sections~\ref{sec:differentscalingshNe}
and~\ref{sec:differentscalingskNe}.  These results fit in very well
with the research on $\Gamma$-convergence of discrete functionals to
continuum functionals, as in \textit{e.g.}
\cite{Braides02,BraidesGelli02,AlicandroCicalese04,
BraidesGelli06,AlicandroCicaleseGloria07,
AlicandroBraidesCicalese08,AlicandroCicaleseGloria10, 
ChambolleGiacominiLussardi10}.
In fact, many of the techniques used in
Section~\ref{sec:differentscalingskNe} are inspired by
\cite{AlicandroCicalese04} specifically.

We like to point out that there is also a substantial literature on the 
convergence of graph Laplacians and their eigenvalues and eigenvectors 
to continuum limits. See \textit{e.g.} \cite{HeinAudibertvonLuxburg05,
GineKoltchinskii06,BelkinNiyogi07,HeinAudibertvonLuxburg07,BelkinNiyogi08,
vonLuxburgBelkinBousquet08,MaiervonLuxburgHein11} and references therein. 
The techniques used and the kind of results obtained in those papers are quite different 
from ours, but in a certain sense our results can be seen as nonlinear extensions of 
the graph Laplacian case.

In all cases we have to impose extra constraints on minimizers of the
Ginzburg-Landau functional to avoid trivial minimizers.  We show prove
results showing that in most cases the addition of a mass constraint
or the addition of an $L^p$ fidelity term to the functional is
compatible with the $\Gamma$-convergence results.

\subsection{Asymptotic behavior of nonlocal means}

The functionals of nonlocal means type ---or   (anisotropic)
nonlocal total variation type--- we consider are built on the square
grid in which the graphs are fully connected,
\cite{BuadesCollMorel05,GilboaOsher07,GilboaOsher08,Bresson09}.  Fix
$\Phi\in C^{\infty}(\T^2)$.  We study
\begin{equation}\label{eq:gN}
g_N(u) := N^{-4} \sum_{i,j,k,l=1}^N (\omega_{L,N})_{i,j,k,l}
|u_{i,j}-u_{k,l}|,
\end{equation}
where $\omega_{L,N} := e^{-d_{L,N}^2/\sigma^2}$ with $\sigma, L >0$
constants (possibly depending on $N$) and
\begin{equation}\label{eq:dLN}
(d_{L,N}^2)_{i,j,k,l} := \sum_{r,s=-L}^L  \Big ( \Phi \Big (\frac{i-r}N,
\frac{j-s}N\Big) - \Phi\Big(\frac{k-r}N, \frac{l-s}N\Big) \Big )^2.
\end{equation}
If $\Phi$ is thought of as an image on $\T^2$, as in \textit{e.g.}
\cite{BertozziFlenner12}, then $L$ gives the size of the pixel
neighborhoods whose pairwise comparisons form the graph weights.  As
we will see in Section~\ref{sec:Gammagraph}, $g_N$ arises as the
$\Gamma$-limit of $f_\e$ in (\ref{eq:graphGLfirstglance}) as $\e\to 0$
on this particular fully connected graph.  This is a natural class of
problems for which to study $\Gamma$-convergence as $N\to\infty$.

\subsection{Structure of the paper}

This paper is structured as follows.  Section~\ref{sec:setup} sets up
notation and gives more background information about how to set up a
PDE-to-graph `dictionary' used to find the graph analogue of the
Ginzburg-Landau functional.  It also gives more details about
$\Gamma$-convergence.  In Section~\ref{sec:Gammagraph} the
$\Gamma$-convergence result for $f_\e$ is proved.  This result holds
for general finite undirected weighted graphs.  Next we turn our
attention to the square grid on the torus.  In
Sections~\ref{sec:differentscalingshNe}
and~\ref{sec:differentscalingskNe} the $\Gamma$-convergence results
for $h_{N,\e}$ and $k_{N,\e}$ respectively are stated and proved.
$\Gamma$-convergence for the nonlocal means type functional $g_N$ is
discussed in Section~\ref{sec:NLM}.  We close with a discussion of our
results and open questions for future research in
Section~\ref{sec:discopen}.

\section{Setup}\label{sec:setup}

We will start with introducing some general graph theoretical
notation.

\subsection{Graph notation}\label{sec:graphnotation}

Let $G=(V, E)$ be an undirected graph with vertex (or node) set $V$, $|V|=m\in
\N$, and edge set $E\subset V^2$.  
Consider the space $\mathcal{V}$ of all functions $V\to \R$.  A
function $u\in \mathcal{V}$ can be seen as a labeling of the vertices
of $G$.  It is useful to number the vertices in $V$ from $1$ to $m$
(in arbitrary but fixed order).  We will write $I_m$ for the set of
integers $i$ satisfying $1\leq i\leq m$.  If $u\in \mathcal{V}$ and
$n_i\in V$ is the $i^{\text{th}}$ vertex we will use the shorthand
notation $u_i:= u(n_i)$.  Let $\mathcal{E}$ be the space of all
functions $E\to \R$, which are skew-symmetric with respect to edge
direction, \textit{i.e.,} if $\varphi\in \mathcal{E}$ and
$e_{ij}:=(n_i, n_j) \in E$ is the edge between the $i^{\text{th}}$ and
$j^{\text{th}}$ vertex in $V$ we write $\varphi_{ij} :=
\varphi(e_{ij})$ and demand $\varphi_{ij} = -\varphi_{ji}$\footnote{We
impose skew-symmetry so that $\mathcal{E}$ can be viewed as the space
of \emph{flows} as defined in \textit{e.g.} \cite[Section
2.2]{CandoganMenacheOzdaglarParrilo11}.  An interesting topological
structure arises in this setting \cite[Section
3]{CandoganMenacheOzdaglarParrilo11}, but for our current purposes the
demand of skew-symmetry neither hinders nor helps
us.\label{note:skewsymm}}.  Since the graph is undirected we have
$e_{ij} \in E \Leftrightarrow e_{ji}\in E$.  When no confusion arises
we will abuse notation slightly and consider
$e_{ij}=e_{ji}$\footnote{See note~\ref{note:skewsymm}.}.  In this
paper we consider weighted graphs, which means we assume there is
given a function $\omega: E\to (0,\infty)$, called the weight
function, which assigns a positive weight to each edge.  Because the
graph is undirected the weight function is symmetric: $\omega_{ij}:=
\omega(e_{ij}) = \omega_{ji}$.  It is often useful to extend $\omega$
to a function on $V^2$ instead of on $E\subset V^2$ by identifying the
edge $e_{ij}$ with the pair $(n_i, n_j) \in V^2$ of its end vertices
and setting $\omega_{ij} = 0$ if and only if $e_{ij}\notin E$.  In
particular, if the graph has no self-loops, $\omega_{ii}=0$ for all
$i\in I_m$. In the same way we can extend $\varphi\in \mathcal{E}$ to a function 
$\varphi: V^2 \to \R$ by setting it to zero on node pairs that are not connected 
by and edge. We can incorporate unweighted graphs in this framework by
viewing them as weighted graphs with the range of $\omega$ restricted
to be $\{0, 1\}$.  We define the degree of vertex $n_i$ as $d_i :=
\sum_{j\in I_m} \omega_{ij}$.  If $G$ has no isolated vertices, then 
for every $i\in I_m\,\,$ $d_i >0$.

\subsection{Graph Laplacians, Dirichlet energy, and total
variation}\label{sec:LaplDirTV}

Our first goal is to define operators that serve as the graph gradient
and graph divergence operators.  Using these operators we can then
define a graph Laplacian, a Dirichlet energy, and isotropic and
anisotropic total variations on the graph.  There are many possible
choices to do this.  Ours follow \cite[Section
2]{HeinAudibertvonLuxburg07} and \cite{GilboaOsher08} and are
presented here.  In Appendix~\ref{sec:structureexplained} we give
details and background on the justification of these choices.
$\mathcal{V} \cong \R^m$ and \footnote{The factor $\frac12$ in $m(m-1)/2$ 
comes in because the graph is undirected. Strictly 
speaking $\mathcal{E}$ is not isomorphic to $\R^{m(m-1)/2}$ if we impose 
skew-symmetry, but this distinction is not relevant for our purposes. 
The Hilbert space structure can be defined in any case.} 
$\mathcal{E}\cong \R^{m(m-1)/2}$ are Hilbert spaces defined via the 
following inner products:
\[
\langle u, v \rangle_{\mathcal{V}} := \sum_{i\in I_m} u_i v_i d_i^r,
\quad \langle \varphi, \phi \rangle_{\mathcal{E}} := \frac12
\sum_{i,j\in I_m} \varphi_{ij} \phi_{ij} \omega_{ij}^{2q-1},
\]
for some $r\in [0,1]$ and $q\in [1/2,1]$.  Different choices of $r$
and $q$ are useful in different contexts, as will become clear later
in this section.  We also define the dot product as operator from
$\mathcal{E}\times\mathcal{E}$ to $\mathcal{V}$ for $\varphi, \phi \in
\mathcal{E}$ as
\[
(\varphi\cdot\phi)_i := \frac12 \sum_{j\in I_m} \varphi_{ij} \phi_{ij}
\omega_{ij}^{2q-1}.
\]
With the Hilbert space structure in place, if we define a difference
operator, then all the other operators and functionals will follow
naturally.  We define the difference operator or gradient $\nabla:
\mathcal{V} \to \mathcal{E}$ as
\[
(\nabla u)_{ij} := \omega_{ij}^{1-q} (u_j-u_i).
\]
Notice that the choice $q=1$ makes the gradient operator nonlocal on
the graph, because its dependence on $\omega_{ij}$ disappears.  The
locality reappears in the $\mathcal{E}$-`inner product' (or strictly
speaking sesquilinear form), which is thus turned semi-definite.  The
opposite is the case for $q=\frac12$.

The other graph objects of interest for this paper now follow:
 
$\bullet$
Norms:
    \begin{itemize}
    \item[-] 
 $\|u\|_{\mathcal{V}} := \sqrt{\langle
    u,u\rangle_{\mathcal{V}}} = \sqrt{\sum_{i\in I_m} u_i^2 d_i^r}$,
    
    \item[-] 
    $\|\varphi\|_{\mathcal{E}} := \sqrt{\langle \varphi,
    \varphi\rangle_{\mathcal{E}}} = \sqrt{\frac12 \sum_{i,j\in I_m}
    \varphi_{ij}^2 \omega_{ij}^{2q-1}}$,
    
      \item[-] 
     $\|\varphi\|_i := \! 
    \sqrt{( \varphi \cdot \varphi)_i}  \! =  \! \sqrt{\frac12 \sum_{j\in I_m}
    \varphi_{ij}^2 \omega_{ij}^{2q-1}}$.  
    Note that
    $\|\cdot\|_{\mathcal{E},\text{dot}} \in \mathcal{V}$, 
    
      \item[-] 
    $\|u\|_{\mathcal{V},\infty} \!  :=  \! \max\{|u_i|  \!  : \!  i \! \in \!  I_m\}$ and
    $\|\varphi\|_{\mathcal{E},\infty} := \!  \max\{|\varphi_{ij}|: i,j  \! \in \! 
    I_m\}$.
    \end{itemize}
 
$\bullet$
The Dirichlet energy does not depend on $r$ or $q$:
    \[
    \frac12 \|\nabla u\|_{\mathcal{E}}^2 = \frac14 \sum_{i,j\in I_m}
    \omega_{ij} (u_i-u_j)^2.  \]

$\bullet$
 The divergence $\dvg: \mathcal{E}\to \mathcal{V}$ defined as the
adjoint of the gradient\footnote{If the graph has an isolated node $i$ for which $\omega_{ij}=$ for all $j$ and $d_i=0$ we interpret this definition as $(\dvg \varphi)_i=0$.}
    \[
    (\dvg \varphi)_i := \frac1{2d_i^r} \sum_{j\in I_m} \omega_{ij}^q
    (\varphi_{ji}-\varphi_{ij}).  \]

$\bullet$
 A family of graph Laplacians $\Delta_r := \dvg\circ\nabla:
\mathcal{V} \to \mathcal{V}$ (not to be confused with the
$p$-Laplacians from the literature).  Writing out this definition
gives
    \[
    (\Delta_r u)_i := d_i^{1-r} u_i - \sum_{j\in I_m}
    \frac{\omega_{ij}}{d_i^r} u_j = \sum_{j\in I_m}
    \frac{\omega_{ij}}{d_i^r} (u_i-u_j).  \]
    If we view $u$ as a vector in $\R^m$ we can also write
    \[
    \Delta_r u = (D^{1-r} - D^{-r} W) u, \]
    where $D^r$ is the diagonal matrix with diagonal elements $D_{ii}
    = d_i^r$ and $W$ is the weight matrix with elements $W_{ij} =
    \omega_{ij}$.  We can recover two of the most frequently used
    graph Laplacians from the literature (\textit{cf.}
    \cite{Chung97,vonLuxburg07,HeinAudibertvonLuxburg07}) by choosing
    either $r=0$ or $r=1$.  For $r=0$ we get the unnormalized graph
    Laplacian, for $r=1$ we have the random walk Laplacian, which also
    goes by the name of (asymmetric) normalized Laplacian.  For the
    latter case, the connection with random walks comes from the fact
    that $D^{-1}W$ is a stochastic matrix, \textit{i.e.,} the sum of
    the elements in each of its rows equals $1$.  Note that $\Delta_r$
    is only symmetric if $r=0$.  A third graph Laplacian which is
    often encountered in the literature is the symmetric normalized
    Laplacian $\id - D^{-\frac12} W D^{-\frac12}$, where $\id$ is the
    $m$ by $m$ identity matrix.  However, this one does not fit well
    into the current framework and we will not consider it here.
 
$\bullet$
Total variations\footnote{An interesting question which
    falls outside the scope of this paper is in which respects, if
    any, the curvatures derived as `derivatives' from these total
    variations resemble the continuum case curvature.}:

    \begin{itemize}
    \item[-]
 The isotropic total variation $TV: \mathcal{V}\to \R$
    defined by
    \begin{align*}
    \TV(u) :=& \max\{ \langle \dvg \varphi, u\rangle_{\mathcal{V}} :
    \varphi\in \mathcal{E},\,\, \underset{i\in
    I_m}\max\,\,\|\varphi\|_i\leq 1\}\\
    =& \frac{\sqrt2}{2} \sum_{i\in I_m} \sqrt{\sum_{j\in I_m}
    \omega_{ij} (u_i-u_j)^2}.
    \end{align*}
    \item[-]
 A family of anisotropic total variations ${\TVa}_q:
    \mathcal{V}\to \R$ defined by
    \begin{align*}
    {\TVa}_q(u) :=  \max\{ \langle \dvg \varphi,
    u\rangle_{\mathcal{V}} : \varphi\in \mathcal{E},\,\,
    \|\varphi\|_{\mathcal{E},\infty}\leq 1\} 
    =  \frac12 \sum_{i,j\in I_m} \omega_{ij}^q |u_i-u_j|.
    \end{align*}
    $\TV$ and ${\TVa}_\frac12$ appear in \cite{GilboaOsher08} as
    isotropic and anisotropic total variation respectively.  In this
    paper we show that ${\TVa}_1$ is the $\Gamma$-limit of a sequence
    of Ginzburg-Landau type functionals (Theorem~\ref{thm:Gammaf}).
    \end{itemize}

Because we can identify $u\in \mathcal{V}$ with a vector $\hat u\in \R^m$
and all norms on $\R^m$ are equivalent, we can express convergence in any
of these norms. For definiteness we choose a simple norm, not dependent
on the degree function $d$: For a sequence 
$\{u_n\}_{n=1}^\infty \subset \mathcal{V}$
and $u_\infty \in \mathcal{V}$ and 
corresponding vectors $\hat u_n, \hat u_\infty \in \R^m$
we define
\[
u_n \to u_\infty \text{ as } n\to\infty \quad
 \text{iff} \quad |\hat u_n - \hat u_\infty|_2 \to 0 \text{ as } n\to\infty,
\]
where $|\cdot|_p$, with $p\in \N$, is defined for $\hat u \in \R^m$ as
\[
|\hat u|_p := \Big (\sum_{i\in I_m} \hat u_i^p \Big )^{\frac1p},
\]
subscript $i$ labeling the elements of the vector.

Where this does not lead to confusion, we will use the same notation
$u$ for both the function $u\in \mathcal{V}$ and the corresponding
vector $\hat u \in \R^m$.

\subsection{The functionals}\label{sec:functionals}

A standard choice of double well potential is $W(s) = s^2 (s-1)^2$.
This is a representative example in the class of potentials for which
our results hold.  We always assume that $W\in C^2(\R)$, $W\geq 0$,
and $W(s)=0$ iff $s\in\{0, 1\}$.  Different lemmas and theorems in
this paper require different additional assumptions:
\begin{enumerate}
\item[$(W_1)$] There exists two disjunct open intervals $\hat I_0$ and
$\hat I_1$ containing $0$ and $1$ respectively, constants $c_0$,
$c_1>0$, and a $\beta>0$, such that
\begin{align}\label{eq:Wawayzero}
&0 \leq \max\{W(s): s\in \hat I\} < \min\{W(s): s\in \hat I^c\} \quad
\text{where } \hat I:=\hat I_0\cup \hat I_1 \text{ and}\notag\\
&\forall s\in \hat I_0\,\,\, W(s) \geq c_0 |s|^\beta \quad \text{and}
\quad \forall s\in \hat I_1\,\,\, W(s) \geq c_1 |s-1|^\beta.
\end{align}
\item[$(W_2)$] There exists a $c>0$ such that for large $|s|\,\,$
$W(s) \geq c (s^2-1)$.  \item[$(W_3)$] There exist $c_1, c_2 >0$ and
$p>0$ such that for large $|s|\,\,$ $c_1 |s|^p \leq W(s) \leq c_2
|s|^p$.  \item[$(W_4)$] There exist $c_3, c_4 >0$ and $q>0$ such that
for large $|s|\,\,$ $c_3 |s|^q \leq W'(s) \leq c_4 |s|^q$.
\end{enumerate}
 $(W_1)$ describes the behavior near the wells.  It says that $W$ is
 strictly bounded away from zero outside of neighborhoods of its wells
 and inside these neighborhoods $W$ has a polynomial lower bound.  We
 need it when we study the simultaneous scaling $\Gamma$-limit for
 $h_N^\alpha$ and gives us explicit estimates of how quickly sequences
 of functions with bounded Ginzburg-Landau `energy' approach the wells
 of the potential.  Assumption $(W_2)$ is a coercivity condition that
 will help establish compactness in some situations (it could be
 replaced by any assumption that allows the conclusion that
 $\int_{\T^2} W(u)$ is bounded from below by a function which is
 coercive in $\|u\|_{L^2(\T^2)}$).  $(W_3)$ with $p\geq 2$ is a
 condition needed to prove compactness in the classical Modica-Mortola
 $\Gamma$-convergence result for $F_\e^{GL}$ (see \textit{e.g.}
 \cite[Proposition 3]{Sternberg88}).  In addition we will use its
 lower bound to prove equi-coerciveness
 (Definition~\ref{def:equicoercive}) of the functional $k_N^\alpha$, which is defined below in (\ref{eq:kN1}).  Finally we will
 need $(W_4)$ to control the behavior of $W$ in between grid points,
 when studying the simultaneous scaling $\Gamma$-limit for
 $k_N^\alpha$.  As is easily checked the standard example $W(s) = s^2
 (s-1)^2$ satisfies all the above assumptions for correctly chosen
 constants.

We frequently encounter binary functions in $\mathcal{V}$ and write
\[
\mathcal{V}^b := \left\{ u\in \mathcal{V}: \forall i\in I_m\,\, u_i\in
\{0, 1\}\right\}.
\]
The graph Ginzburg-Landau functional $f_\e: \mathcal{V} \to \R$ from
(\ref{eq:graphGLfirstglance}) can be defined in terms of the Dirichlet
energy:
\[
f_\e(u) = 2 \chi \|\nabla u\|_{\mathcal{E}}^2 + \frac1\e \sum_{i=1}^m
W(u_i), \qquad \text{with } \chi\in (0, \infty).
\]
Let $\T^2$ be the two-dimensional flat unit torus.  We construct a
square grid with $m=N^2$ nodes $G_N := N^{-1} \Z^2 \cap \T^2$.
Interpreting $G_N$ as a graph we can use the notation from
Section~\ref{sec:graphnotation} with subscript $N$, \textit{e.g.}
$V_N$ are the vertices of $G_N$, $\mathcal{V}_N$ are the real-valued
functions on $V_N$, $(\mathcal{V}^b)_N$ the binary ($\{0,1\}$-valued)
functions on $V_N$, etc.  We understand the vertices $V_N$ to be
embedded in $\T^2$.  To distinguish the horizontal and vertical
directions in our graph when working on $G_N$, instead of $u_i$ we
will write $u_{i,j}:= u(n_{i,j})$ where $n_{i,j}:= (i/N, j/N) \in
V_N\subset\T^2$.  We refer to a single square in the grid by
\begin{equation}\label{eq:SNij}
S_N^{i,j}:= [i/N, (i+\nobreak 1)/N)\times [j/N, (j+1)/N).
\end{equation}
We remind the reader that we introduced three different functionals on
the square grid: the graph theoretical Ginzburg-Landau functional
$h_{N,\e}: \mathcal{V}_N \to \R$ in (\ref{eq:hNe}), the discretized
Ginzburg-Landau functional $k_{N,\e}: \mathcal{V}_N \to \R$ in
(\ref{eq:kNe}), and the `sharp interface' (\textit{i.e.,} $\e\to 0$)
nonlocal means functional $g_N: \mathcal{V}_N^b \to \R$
in (\ref{eq:gN}).

We call $h_{N,\e}$ the graph theoretical Ginzburg-Landau functional
because it is equal to $f_\e$ from (\ref{eq:graphGLfirstglance}) if we
choose the weight $\omega$ in $f_\e$ as
\[
\omega(n_{i,j}, n_{k,l}) := \left\{\begin{array}{ll} N^{-1} & \text{if
} (|i-k|=1 \wedge j=l) \vee (i=k \wedge |j-l|=1),\\
0 & \text{otherwise}, \end{array}\right.
\]
and $\chi=\frac12$.  $k_{N,\e}$ we get by using the trapezoidal rule
and a standard finite difference scheme to discretize the
Ginzburg-Landau functional $F_\e^{GL}$.  We have used the periodicity
to relate the terms of the form $(u_{i,j}-u_{i-1,j})^2$ and
$(u_{i,j}-u_{i,j-1})^2$ to $(u_{i+1,j}-u_{i,j})^2$ and
$(u_{i,j+1}-u_{i,j})^2$ in the sum, respectively.

To study $\Gamma$-convergence for the `simultaneous' limits $\e\to 0$
and $N\to\infty$ of $h_{N,\e}$ and $k_{N,\e}$ we set $\e=N^{-\alpha}$,
for $\alpha>0$, and let $N\to\infty$ in the functionals
\begin{align}
h_N^\alpha(u) &:= N^{-1} \sum_{i, j=1}^{N}
(u_{i+1,j}-u_{i,j})^2+(u_{i,j+1}-u_{i,j})^2 + N^\alpha \sum_{i,j=1}^N
W(u_{i,j}),\label{eq:hN}\\
k_N^\alpha(u) &:= N^{-\alpha} \sum_{i, j=1}^N
(u_{i+1,j}-u_{i,j})^2+(u_{i,j+1}-u_{i,j})^2 + N^{\alpha-2}
\sum_{i,j=1}^N W(u_{i,j})\label{eq:kN1}.
\end{align}
We prove $\Gamma$-convergence results for $h_N^\alpha$ in
Section~\ref{sec:simulscalehNalpha} and for $k_N^\alpha$ in
Section~\ref{sec:simulscalekNalpha}.

Note that $h_N^\alpha = N^{\gamma-1} k_N^\gamma$ if $\gamma =
\frac{\alpha+3}2$.  This shows that we do not expect $h_N^\alpha$ and
$k_N^\gamma$ to have the same limit, unless possibly if
$\alpha=\gamma=1$.  This value falls outside the regimes for $\alpha$
we consider and hence we do find different limits.

\subsection{$\Gamma$-convergence}\label{sec:explainGamma}

$\Gamma$-convergence was introduced by De Giorgi and 
Franzoni in \cite{DeGiorgiFranzoni75}. It is a type 
of convergence for function(al)s that is tailored to 
the needs of minimization problems as we will see below. 
A good introduction to the subject is \cite{Braides02},
 the standard reference work is \cite{DalMaso93}.

\begin{definition}\label{def:LBUB}
Let $X$ be a metric space and let $\{F_n\}_{n=1}^\infty$
 be a sequence of functionals $F_n: X \to \R \cup\{\pm \infty\}$.
  We say that $F_n$ $\Gamma$-converges to the functional 
  $F: X\to \R\cup\{\pm \infty\}$, denoted by
   $F_j \overset{\Gamma}{\to} F$ if, for all $u\in X$ we have that
\begin{enumerate}
\item[(LB)] for every sequence $\{u_n\}_{n=1}^\infty$ 
such that $u_n\to u$ it holds that $F(u) \leq
 \underset{n\to\infty}\liminf\, F_n(u_n)$ and
\item[(UB)] there exists a sequence $\{u_n\}_{n=1}^\infty$
 such that $F(u) \geq \underset{n\to\infty}\limsup\, F_n(u_n)$.
\end{enumerate}
\end{definition}

The lower bound condition (LB) tells us that the values along 
the sequence $F_n(u_n)$ are bounded from below by $F(u)$, the
upper bound (UB) shows that the value $F(u)$ is actually achieved.
Combined with a compactness or equi-coerciveness condition
(Definition~\ref{def:equicoercive} below), this allows for
conclusions on the minimizers of $F_n$ and $F$.

It is useful to note that to prove (LB) for a given sequence
$\{u_n\}_{n=1}^\infty$ we only need to prove it for a 
subsequence $\{u_{n'}\}_{n'=1}^\infty \subset \{u_n\}_{n=1}^\infty$ 
such that
$\underset{n'\to\infty}\lim\, F_{n'}(u_{n'}) =
\underset{n\to\infty}\liminf\, F_n(u_n)$. If (LB) is satisfied
for such a sequence (which always exists), then
$F(u) \leq \underset{n'\to\infty}\liminf\, F_{n'}(u_{n'}) =
\underset{n'\to\infty}\lim\, F_{n'}(u_n') =
\underset{n\to\infty}\liminf\, F(u_n)$.  Hence, when proving (LB) we
will assume without loss of generality that $\{u_n\}_{n=1}^\infty$ is
such a sequence.  The uniqueness of the limit then implies that it
suffices to prove (LB) for any subsequence.  Clearly we can also
assume that $\liminf\, F_n(u_n) < \infty$ and hence it suffices to
prove (LB) for a specific subsequence $\{u_{n''}\}_{n''=1}^\infty$ for
which there is a $C>0$ such that $F_{n''}(u_{n''}) \leq C$.  In
practice this means that to prove (LB) we can assume a uniform bound
on $F_n(u_n)$ and freely pass to subsequences when needed.

If we are working with functionals that depend on a continuous 
parameter, \textit{e.g.} $\e\to 0$ or $N\to\infty$, we have to 
prove (LB) and (UB) for an arbitrary sequence $\{\e_n\}_{n=1}^\infty$
 with $\e_n\to 0$ as $n\to\infty$ (or $\{N_n\}_{n=1}^\infty$ 
 with $N_n\to \infty$ as $n\to\infty$).

\begin{definition}\label{def:equicoercive}
Let $X$ be a metric space and let $\{F_n\}_{n=1}^\infty$ be a sequence
 of functionals $F_n: X \to \R \cup\{\pm \infty\}$. We say the
  sequence is \emph{equi-coercive} if for every $t\in \R$ there exists 
a compact set $K_t\subset X$ such that for every $n\in \N\,\,$ 
$\{u\in X: F_n(u)\leq t\} \subset K_t$.
\end{definition}
In practice equi-coerciveness is proved by showing that any sequence
 $\{u_n\}_{n=1}^\infty$ for which $F_n(u_n)$ is uniformly bounded
 has a convergent subsequence.

$\Gamma$-convergence combined with equi-coerciveness allows us to
 conclude the following result.

\begin{theorem}[Chapter 7 in \cite{DalMaso93} and Theorem 1.21 in \cite{Braides02}]
Let $X$ be a metric space, $\{F_n\}_{n=1}^\infty$ be a sequence 
of equi-coercive functionals $F_n: X \to \R \cup\{\pm \infty\}$,
 and let $F$ be the $\Gamma$-limit of $F_n$ for $n\to \infty$. 
 Then there exists a minimizer of $F$ in $X$ and
$
\min\{F(u): u\in X\} = \underset{n\to\infty}\lim\, \inf\{F_n(u): u\in X\}.
$
Furthermore, if $\{u_n\}_{n=1}^\infty\subset X$ is a precompact sequence such that
\[
\underset{n\to\infty}\lim\, F_n(u_n) = \underset{n\to\infty}\lim\, \inf\{F_n(u): u\in X\},
\]
then every cluster point of this sequence is a minimizer of $F$.
\end{theorem}

Following \cite{AlicandroCicalese04} for our purposes it turns out it is 
often more useful to reformulate $\Gamma$-convergence in terms of the $\Gamma$-lower limit
\[
F'(u) := \inf\{\underset{n\to\infty}\liminf\, F_n(u_n): u_n\to u\}
\]
and the $\Gamma$-upper limit
\[
F''(u) := \inf\{\underset{n\to\infty}\limsup\, F_n(u_n): u_n\to u\},
\]
\cite[Definition 4.1]{DalMaso93}. It can be shown,
 \cite[Remark 4.2, Proposition 8.1]{DalMaso93}, \cite{Braides06},
  that our definition of $\Gamma$-convergence above is equivalent 
  to the following two conditions. For each $u\in X$ we have that
\begin{enumerate}
\item[(LB')] $F(u) \leq F'(u)$ and
\item[(UB')] $F(u) \geq F''(u)$.
\end{enumerate}
The benefit of this reformulation is that the functions $F'$ and $F''$ 
are lower semicontinuous \cite[Proposition 6.8]{DalMaso93}, which 
comes in handy in Section~\ref{sec:differentscalingskNe}. In fact, 
since conditions (LB) and (LB') are equivalent, we will sometimes 
use the combination (LB)+(UB') to prove $\Gamma$-convergence 
in this paper. Note that (UB) implies (UB').

\subsection{Constraints}\label{sec:explainconstraints}

It is common in semi-supervised learning applications to have a mass
constraint or an additional term in the functional corresponding to a
fit to the known data.
Moreover such constraints are typically necessary to obtain nontrivial
minimizers.
We need to check that these constraints are compatible with the
convergence.

First consider the case of an adding a fidelity term of the form
$\lambda |u-f|_p^p$ to the functional, where $f\in \mathcal{V}$ is a
given function (usually representing some known data to which the
minimizer should be similar) defined on some or all of the vertices in
$V$ and $\lambda>0$ is a parameter.
If $p$ agrees with the topology of the $\Gamma$-convergence we can use
the property that $\Gamma$-limits are stable under addition of a
continuous term or a sequence of continuously convergent terms
\cite[Definition 4.7, Propositions 6.20--21]{DalMaso93} to conclude
that the Ginzburg-Landau functionals plus fidelity term again
$\Gamma$-converge.
We will summarize the results that are relevant for us in the
following definition and lemma, based on the cited definition and
propositions in \cite{DalMaso93}.

\begin{definition}
Let $X$ be a metric space and let $\{F_n\}_{n=1}^\infty$ be a sequence
of functionals $F_n: X \to \R \cup\{\pm \infty\}$.  We say the
sequence is \emph{continuously convergent} to a function $F: X\to \R
\cup\{\pm \infty\}$ if for every $u\in X$ and for every $\eta>0$ there
is an $\bar N \in \N$ and a $\delta>0$ such that for all $n\geq \bar
N$ and $v\in X$ with $\|u-v\|<\delta$ we have $|F_n(v)-F(u)| < \eta$.
\end{definition}

\begin{lemma}\label{lem:addperturbations}
Let $X$ be a metric space and let $\{F_n\}_{n=1}^\infty$ be a sequence
of functionals $F_n: X \to \R \cup\{\pm \infty\}$ which
$\Gamma$-converges to $F: X \to \R \cup\{\pm \infty\}$ and
$\{H_n\}_{n=1}^\infty$ a sequence of functionals $H_n: X \to \R
\cup\{\pm \infty\}$ which is continously convergent to $H: X \to \R
\cup\{\pm \infty\}$, then $F_n+H_n$ $\Gamma$-converges to $F+H$. 
If $\hat G: X \to \R \cup \{\pm \infty\}$ is continuous functional,
then $F_n + \hat G$ $\Gamma$-converges to $F+\hat G$.
\end{lemma}

If instead a mass constraint is imposed on the minimizers we have to
check that each convergent sequence preserves the constraint (to make
it compatible with the lower bound and compactness conditions) and
that the recovery sequence from the upper bound condition does satisfy
the constraint (or can be adapted to satisfy it, without violating the
upper bound condition).  Since we are dealing with $L^p$ convergence,
the former is usually trivially satisfied, but the latter does demand
some more attention.  Details for each of the functionals are provided
in the relevant sections.

For functions $u\in \mathcal{V}^b$ we carefully need to determine the
form of our mass constraint.  The constraint $\sum_{i=1}^m u_i = M$
leads to shrinking support when $m\to\infty$ which is unwanted.  An
alternative condition is an average mass constraint of the form
$\frac1m \sum_{i=1}^m u_i = M$.  Note that $\sum_{i=1}^m u_i$ can take
on only integer values between 0 and $m$ and hence $m M$ should be of
that form as well, in order for the average mass constraint not to
lead to an empty set of admissible minimizers.  If we impose this for
all $m$ this is only possible if $M=0$ or $M=1$, however specific
subsequences of $m$ can be able to satisfy this condition for
different values of $M$ (\textit{e.g.} if $M=\frac12$ and we consider
even $m$).  Hence the choice of $M$ can constrain the subsequences of
$m$ which are admissible.  In order to avoid the possible difficulties
with the average mass \emph{equality}, one can also impose an average
mass \emph{inequality}.  Since the arguments for the mass equality
easily generalize to an inequality, we will not discuss this situation
further.

\section{$\Gamma$-convergence of $f_\e$}\label{sec:Gammagraph}

\subsection{$\Gamma$-convergence and compactness}

In this section we prove $\Gamma$-conver\-gence and compactness for the
functional $f_\e: \mathcal{V} \to \R$ from
(\ref{eq:graphGLfirstglance}).

\begin{theorem}[$\Gamma$-convergence]\label{thm:Gammaf}
$\displaystyle f_\e \overset{\Gamma}{\to} f_0$ as $\e\to 0$, where
\begin{align*}
f_0 (u) :=& \left\{ \begin{array}{ll} \chi \sum_{i,j\in I_m}
\omega_{ij} |u_i-u_j| & \text{if } u\in\mathcal{V}^b,\\ +\infty &
\text{otherwise}\end{array}\right.    = \left\{ \begin{array}{ll}
2 \chi {\TVa}_1(u) & \text{if } u\in\mathcal{V}^b,\\ +\infty &
\text{otherwise}.\end{array}\right.
\end{align*}
\end{theorem}

\begin{theorem}[Compactness]\label{thm:compactnessongraph}
Let $W$ satisfy the coercivity condition $(W_2)$, let
$\{\e_n\}_{n=1}^\infty \subset \R_+$ be a sequence such that $\e_n\to
0$ as $n\to \infty$, and let $\{u_n\}_{n=1}^\infty \subset
\mathcal{V}$ be a sequence such that there exists a $C>0$ such that
for all $n\in\N\,\,$ $f_{\e_n}(u_n)<C$.  Then there exists a
subsequence $\{u_{n'}\}_{n'=1}^\infty \subset \{u_n\}_{n=1}^\infty$
and a $u_\infty \in \mathcal{V}^b$ such that $u_{n'} \to u_\infty$ as
$n\to\infty$.
\end{theorem}

Although in $F_\e^{GL}$ the first term is scaled by $\e$, the first
term of $f_\e$ contains no $\e$.  The reason for this is that the
Dirichlet energy in $F_\e^{GL}$ is unbounded for the binary functions
$u$ that form the domain of the limit functional $F_0^{GL}$.  However,
the difference terms in $f_\e$ are finite even for the binary
functions and thus need no rescaling.  The proof of
$\Gamma$-convergence uses this fact, to view the difference terms as a
continuous perturbation of the functional
\[
w_\e(u):= \frac1\e \sum_{i\in I_m} W(u_i).
\]

\begin{lemma}\label{lem:Gammaw}
The sequence of functionals $w_\e$ $\Gamma$-converges:
$
w_\e \underset{\e\to 0}{\overset{\Gamma}{\to}} w_0
$
where $w_0: \mathcal{V} \to \{0,\infty\}$ is defined via
\[
w_0(u):= \left\{\begin{array}{ll} 0 &\text{if } u\in \mathcal{V}^b,\\
+\infty & \text{otherwise}.  \end{array}\right.
\]
\end{lemma}

\begin{proof}[\bf Proof.]
To prove the required lower bound (LB) let $u\in \mathcal{V}$ and
consider sequences $\{\e_n\}_{n=1}^\infty$ and $\{u_n\}_{n=1}^\infty$
such that $\e_n \to 0$ and $u_n \to u$ as $n\to \infty$.  If $u \in
\mathcal{V}^b$ from $w_\e \geq 0$ it follows that
$
w_0(u) = 0 \leq \underset{n\to\infty}{\liminf}\, w_{\e_n}(u_n).
$
If on the other hand $u\in \mathcal{V}\setminus\mathcal{V}^b$, then $\exists \bar v\in V$ such that for
large enough $n \,\,$ $u_n(\bar
v)\notin \{0,1\}$ and hence
\[
\underset{n\to\infty}{\liminf}\, w_{\e_n}(u_n) \geq
\underset{n\to\infty}{\liminf}\, \frac1{\e_n} W(\bar v) = \infty =
w_0(u).
\]
For the upper bound (UB) we can assume without loss of generality that
$u \in \mathcal{V}^b$.  Define for every $n\in\N\, $ $u_n := u$, then
trivially $u_n \to u$ if $n\to \infty$ and furthermore
$
\underset{n\to\infty}{\limsup}\, w_{\e_n}(u_n) = 0 = w_0(u).
$
\end{proof}

\noindent
{\bf Proof of Theorem~\ref{thm:Gammaf}.}
Define
\[
\hat w(u):= \chi \sum_{i, j\in I_m} \omega_{ij} (u_i-u_j)^2.
\]
$\hat w$ is a polynomial on $\R^m$ and hence continuous.
$\Gamma$-convergence is stable under continuous perturbations
(\textit{e.g.} \cite[Proposition 6.21]{DalMaso93}).  Since $f_\e$ is a
continuous perturbation of $w_\e$ we have by Lemma~\ref{lem:Gammaw}
$
f_\e \overset{\Gamma}\to \hat w + w_0 $ as $ \e \to 0.
$
We complete the proof by noting that if $u\in\mathcal{V}^b$, then
$$
\hat w(u) = \chi \sum_{i, j\in I_m} \omega_{ij} |u_i-u_j|.
\eqno{\qed}
$$

\begin{proof}[\bf Proof of Theorem~\ref{thm:compactnessongraph}]
By the uniform bound on $f_{\e_n}(u_n)$ we have
\[
\sum_{i\in I_m} W((u_n)_i) \leq C \e_n.
\]
Combined with the coercivity condition $(W_2)$ on $W$ we conclude that
for $n$ large enough the sequence $\{u_n\}_{n=1}^\infty$ is bounded
and hence by the Bolzano-Weierstrass theorem there exists a converging
subsequence with limit $u_\infty$.  Since $W \in C^2(\R)$ and $\e_n\to
0$ as $n\to\infty$ we conclude that $W(u_\infty(v_i))=0$ for all $i\in
I_m$ and hence $u_\infty \in \mathcal{V}^b$.
\end{proof}

\begin{remark}
Since $f_0$ is defined on binary functions $u$, if we write
\[
S_k := \{i\in I_m: u_i=k\} \quad \text{for } k\in \{0, 1\},
\]
we can rewrite $f_0$ as
\[
f_0(u) = \left\{ \begin{array}{ll} \chi \sum_{i\in S_0, j\in S_1}
\omega_{i,j} & \text{if } u\in\mathcal{V}^b,\\ +\infty &
\text{otherwise}\end{array}\right.
\]
Minimizing $f_0$ thus corresponds to finding a minimal graph cut of
$G$, \textit{i.e.,} dividing the graph into clusters with minimal edge
weight between them.  Such a minimization requires an extra constraint
to avoid trivial minimizers.  A common choice is to prescribe the
number of clusters (in this case two) one wants, or to introduce a
normalization into the sum of the weights based on the cluster sizes
(\textit{e.g.} normalized cut, normalized association
\cite{ShiMalik00}, and Cheeger cut \cite{SzlamBresson10})).  One could
also minimize $f_0$ under a fixed mass constraint (see
Section~\ref{sec:constraintsfe} below).  One often relaxes the problem
of normalized graph cut minimization by losing the binarity
constraint, \textit{e.g.} in spectral clustering
\cite{NgJordanWeiss02,vonLuxburg07}.

If the weight function $\omega$ is such that there is a nontrivial
partition $A \cup B = I_m$ such that $\omega_{ij} = 0$ for $i\in A$
and $j \in B$ then a nontrivial minimizer of $f_0$ (without the mass
constraint) is clearly given via $S_0 = A$ and $S_1 = B$.  In this
case we can write
\[
u_i = \frac{\sum_{j\in I_m} \omega_{ij} u_j}{d_i}
\]
which is of a form related to the image denoising method known as
nonlocal means, \textit{cf.} \cite{BuadesCollMorel05}.

\end{remark}

\begin{remark}
In this paper we assume that $W$ has wells at $0$ and $1$.  The proofs
in this section make no use of this fact and can easily be extended to
potentials $W$ with wells at values $s_1$ and $s_2$.  In this case the
set $\mathcal{V}^b$ needs to be redefined as the set of functions
taking values in $\{s_1, s_2\}$ and the limit functional $f_0$ from
Theorem~\ref{thm:Gammaf} is multiplied by a factor $|s_1-s_2|$, since
$(u_i-u_j)^2 = |s_1-s_2| |u_i-u_j|$ for $u\in \mathcal{V}^b$.
\end{remark}

\subsection{Constraints}\label{sec:constraintsfe}

Next we show that the addition of a fidelity term $\lambda |u-f|_p^p$
or a mass constraint is compatible with the $\Gamma$-convergence.

\begin{theorem}[Constraints]\label{thm:constraintsforfe}
$\displaystyle f_\e + \lambda |\cdot-f|_p^p \overset{\Gamma}{\to} f_0
+ \lambda |\cdot-f|_p^p$ as $\e\to 0$, where $p\in \R_+$, $\lambda>0$,
and a given function $f\in \mathcal{V}$ (or possibly a given function
$f: U\to \R$ where $U$ is a strict subset of the vertex set $V$ and
the sum in $|\cdot-f|_p^p$ is restricted to vertices in $U$).
Compactness for $f_\e + \lambda |\cdot -f|_p^p$ as in
Theorem~$\ref{thm:compactnessongraph}$ holds. 
If instead, for fixed $M>0$, the domain of definition of $f_\e$ is
restricted to
\[
\mathcal{V}^M:=\{u\in\mathcal{V}: \sum_{i\in I_m} u_i = mM\},
\]
where $M$ is such that $mM$ is an integer between 0 and $m$, then the
results of Theorems~$\ref{thm:Gammaf}$ and~$\ref{thm:compactnessongraph}$
remain valid, with the domain of $f_0$ restricted to $\mathcal{V}^M$.
\end{theorem}

\begin{proof}[\bf Proof.]
The fidelity term $\lambda |u-f|_p^p$ is a polynomial, hence a
continuous perturbation independent of $\e$ to $f_\e$ and thus
$\Gamma$-convergence follows by Theorem~\ref{thm:Gammaf} and
Lemma~\ref{lem:addperturbations}.  The addition of this term does not
affect the compactness property at all.

The mass constraint is compatible with the limit functional being
defined on binary functions.  The constraint is preserved under
convergence in $\mathcal{V} \cong \R^m$, so it is compatible with (LB)
from Definition~\ref{def:LBUB} and compactness.  If $u\in
\mathcal{V}^b$ satisfies the mass constraint, then trivially so does
the recovery sequence $\{u_n\}_{n=1}^\infty$ used to prove (UB) from
Definition~\ref{def:LBUB} in Lemma~\ref{lem:Gammaw}.
\end{proof}

\section{$\Gamma$-limits for the graph based functional
$h_{N,\e}$}\label{sec:differentscalingshNe}

In this section we will study the convergence properties of $h_{N,\e}$
from (\ref{eq:hNe}).  We consider two different cases.  In the first
we first take the limit $\e\to 0$ and then $N\to\infty$, in the second
we take both limits at once by substituting $\e=N^{-\alpha}$ for well
chosen $\alpha>0$ and then considering the limit $N\to\infty$ for
$h_N^\alpha$ in (\ref{eq:hN}).  These results have a similar feel as
numerical convergence results, however the lack of regularity of the
binary limit functions complicates the results and proofs.

\smallskip

\begin{remark}
We note that, while we give the proofs for $\T^2$ in
Sections~\ref{sec:proofh} and~\ref{sec:simulscalehNalpha}, they can
easily be generalized to $\T^d$ for any $d\in \N$, if we let the
scaling factor in the first term of $h_{N,\e}$ in (\ref{eq:hNe}) be
$N^{1-d}$ instead of $N^{-1}$ and we change $h_N^\alpha$ in
(\ref{eq:hN}) accordingly:
\begin{align*}
h_{N,\e}(u) &:= N^{1-d} \sum_{i, j=1}^{N}
(u_{i+1,j}-u_{i,j})^2+(u_{i,j+1}-u_{i,j})^2 + \e^{-1} \sum_{i,j=1}^N
W(u_{i,j}),\\
h_N^\alpha(u) &:= N^{1-d} \sum_{i, j=1}^{N}
(u_{i+1,j}-u_{i,j})^2+(u_{i,j+1}-u_{i,j})^2 + N^\alpha \sum_{i,j=1}^N
W(u_{i,j}).
\end{align*}
For the extra fidelity term in Theorem~\ref{thm:constraintsforh} the
scaling factor then needs to be $N^{-d}$ instead of $N^{-2}$.
\end{remark}

\subsection{Compactness and sequential $\Gamma$-limits: first $\e\to
0$, then $N\to\infty$}\label{sec:proofh}

By Theorem~\ref{thm:Gammaf} we immediately have $\displaystyle h_{N,
\e} \overset{\Gamma}{\to} h_{N,0}$ as $\e\to 0$, where $h_{N,0}$ is
defined for $u\in \mathcal{V}$ as
\[
h_{N,0}(u) := \left\{\begin{array}{ll} N^{-1} \sum_{i, j=1}^N \big(
|u_{i+1,j}-u_{i,j}| + |u_{i,j+1}-u_{i,j}|\big) & \text{if } u \in
\mathcal{V}_N^b,\\
+\infty & \text{otherwise}.  \end{array}\right.
\]

In this section we prove that $\displaystyle h_{N, 0}
\overset{\Gamma}{\to} h_{\infty,0}$ as $N\to \infty$ where
$h_{\infty,0}$ is defined for $u\in L^1(\T^2)$ as
\[
h_{\infty, 0}(u) := \left\{\begin{array}{ll} \int_{\T^2} |u_x| + |u_y|
&\text{if } u \in BV(\T^2; \{0, 1\}),\\
+\infty &\text{otherwise}, \end{array}\right.
\]
The anisotropic total variation in $h_{\infty, 0}$ is defined as
\[
\int_{\T^2} |u_x| + |u_y| := \sup \left\{ \int_{\T^2} u \,\text{div} v
: v\in C_c^1(\T^2; \R^2), \forall x \, \, |v(x)|_\infty\leq 1\right\},
\]
where for a vector $v(x)= (v_1(x), v_2(x)) \in \R^2$, the norm
$|\cdot|_\infty$ is defined by
\[ |v(x)|_\infty \, : =\, \max \{ |v_1(x)|, |v_2(x)| \}.  \]
For functions $u\in BV(\T^2; \{0, 1\})$ this anisotropic total
variation gives the length of the (reduced\footnote{For the definition of reduced boundary see \cite[Definition
3.54]{AmbrosioFuscoPallara00}.}) boundary of the set
$\{u=1\}$ projected onto the horizontal and vertical axes (counting
multiplicities).

We prove a compactness and $\Gamma$-convergence result.

\begin{theorem}[Compactness]\label{thm:compactnessforh}
Let $\{N_n\}_{n=1}^\infty\subset\N$ satisfy $N_n\to \infty$ as
$n\to\infty$ and let $\{u_n\}_{n=1}^\infty \subset L^1(\T^2)$ be a
sequence for which there is a constant $C>0$ such that for all
$n\in\N$
$
h_{N_n,0}(u_n) \leq C.
$
Then there exists a subsequence $\{u_{n'}\}_{n'=1}^\infty \subset
\{u_n\}_{n=1}^\infty$ and a $u\in BV(\T^2; \{0,1\})$ such that
$u_{n'}\to u$ in $L^1(\T^2)$ as $n'\to\infty$.
\end{theorem}

\begin{theorem}[$\Gamma$-convergence]\label{thm:GammaconvergenceforhN0}
$\displaystyle h_{N, 0} \overset{\Gamma}{\to} h_{\infty, 0}$ as $N\to
\infty$ in the $L^1(\T^2)$ topology.
\end{theorem}

The convergence of $h_{N,0}$ to an anisotropic total variation is reminiscent of 
the related, but different, results in \cite{ChambolleGiacominiLussardi10}.

Because the limit function $h_{\infty, 0}$ is defined on $L^1(\T^2)$
functions, it will be useful to identify the binary functions on the
square graph, \textit{i.e.,} the functions in $\mathcal{V}_N^b$, with a
subset of $L^1(\T^2)$, namely binary functions on $\T^2$ that are
piecewise constant on the squares of the grid.  Using the notation
$S_N^{i,j}$ from (\ref{eq:SNij}) we define
\begin{align}
\mathcal{A}_N &:= \left\{ u \in L^1(\T^2): \forall (i,j)\in
I_{N}^2\,\, u \text{ is constant a.e. on }
S_N^{i,j}\right\},\label{eq:AN}\\
\mathcal{A}_N^b &:= \left\{ u \in \mathcal{A}_N: u \in L^1(\T^2;
\{0,1\}) \right\}.\notag
\end{align}
We construct a bijection between the normed spaces $\mathcal{V}_N$ and
$\mathcal{A}_N$ by identifying $u \in \mathcal{V}_N$ with the unique
$\tilde u \in \mathcal{A}_N$ which satisfies $\tilde u \equiv u_{i,j}$
a.e. on $S_N^{i,j}$ for all $(i,j)\in I_{N}^2$.  It is easy to check
that convergence in $\mathcal{V}_N$ corresponds to $L^p$ convergence
in $\mathcal{A}_N$ ($1\leq p < \infty$) and that the bijection maps
the subset $\mathcal{V}_N^b$ to $\mathcal{A}_N^b$ and vice versa (for
its inverse).  In what follows we will drop the tilde if this does not
lead to confusion.

With this identification we write for $u\in \mathcal{A}_N$
\begin{equation}\label{eq:hN0reformulation}
h_{N,0}(u) = \left\{\begin{array}{ll} 
  \int_{\T^2} |u_x| + |u_y| &
\text{if } u \in \mathcal{A}_N^b,\\
+\infty & \text{otherwise}.  \end{array}\right.
\end{equation}
In fact for our $\Gamma$-convergence purposes without loss of
generality we extend the functional to all $u\in L^1(\T^2)$, such that
$h_{N,0}(u) = + \infty$ if $u\in L^1(\T^2)\setminus \mathcal{A}_N^b$.

First we prove the compactness result.

\begin{proof}[\bf Proof of Theorem~\ref{thm:compactnessforh}]
By the definition of the isotropic and anisotropic total variation and
(\ref{eq:hN0reformulation}) we have for all $n\in\N$
\[
\int_{\T^2} |\nabla u_n| \leq \int_{\T^2} |(u_n)_x| + |(u_n)_y| \leq
C.
\]
In addition for each $n\in \N\,\,$ $u_n\in \mathcal{A}_{N_n}^b$, hence
$\|u_n\|_{L^1(\T^2)} \leq 1$.  We deduce that the sequence
$\{u_n\}_{n=1}^\infty$ is uniformly bounded in the BV norm and thus by
compactness (\cite[Theorem 1.19]{Giusti84} or \cite[5.2.3 Theorem
4]{EvansGariepy92}) there exists a subsequence
$\{u_{n'}\}_{n'=1}^\infty \subset \{u_n\}_{n=1}^\infty$ and a $u\in
BV(\T^2)$ such that $u_{n'}\to u$ in $L^1(\T^2)$ as $n'\to\infty$.
Since all $u_n$ take the values 0 and 1 a.e. by pointwise a.e.
convergence (after possibly going to another subsequence) so does $u$.
\end{proof}

In the next lemma (LB) from Definition~\ref{def:LBUB} is proved.

\begin{lemma}[Lower bound]\label{lem:hlowerbound}
Let $u\in L^1(\T^2)$ and let $\{u_n\}_{n=1}^\infty \subset L^1(\T^2)$
and $\{N_n\}_{n=1}^\infty\subset\N$ be such that $u_n \to u$ in
$L^1(\T^2)$ and $N_n \to \infty$ as $n\to \infty$.  Then
\[
h_{\infty,0}(u) \leq \underset{n\to\infty}\liminf\, h_{N_n,0}(u_n).
\]
\end{lemma}

\begin{proof}[\bf Proof.]
First consider the case where $u\in BV(\T^2; \{0,1\})$, then without
loss of generality we can assume that $u_n \in \mathcal{A}_{N_n}^b$.
Analogous to the isotropic total variation also this anisotropic total
variation is lower semicontinuous with respect to $L^1$ convergence
\cite[Lemma A.5]{ChoksivanGennipOberman11} hence we find
\[
h_{\infty,0}(u) = \int_{\T^2} |u_x|+|u_y| \leq
\underset{n\to\infty}\liminf\, \int_{\T^2} |(u_n)_x|+|(u_n)_y| =
\underset{n\to\infty}\liminf\, h_{N_n,0}(u_n).
\] 
If $u\in L^1(\T^2) \setminus BV(\T^2; \{0,1\})$ and $u_n\to u$ in
$L^1(\T^2),$ then
\[
h_{\infty,0}(u) = \infty = \underset{n\to\infty}\liminf\,
h_{N_n,0}(u_n),
\]
which we prove via contradiction: Assume
$\underset{n\to\infty}\liminf\, h_{N_n,0}(u_n) < \infty$, then there
is a subsequence $\{u_{n'}\}_{n'=1}^\infty \subset
\{u_n\}_{n=1}^\infty$ for which $h_{N_{n'},0}(u_{n'})$ is uniformly
bounded and hence by Theorem~\ref{thm:compactnessforh} $u_{n'} \to
\hat u$ in $L^1(\T^2)$ as $n'\to\infty$, where $\hat u \in
BV(\T^2;\{0,1\})$.  By the uniqueness of limit $\hat u = u$ which is a
contradiction.
\end{proof}

To prove the $\limsup$ inequality we use the following results from
\cite{ChoksivanGennipOberman11} (which we give here in a form adapted
to our situation which follows easily from the results in
\cite{ChoksivanGennipOberman11}).

\begin{lemma}[Corollary A.4 and Theorem 4.1 from
\cite{ChoksivanGennipOberman11}]\label{lem:existencevectorfield}
For $u\in BV(\T^2)$ we have
\begin{align*}
&
\int_{\T^2} |u_x| + |u_y| 
\\
&
= \sup \Big \{ \int_{\T^2} u \,\text{div} \,
v : v\in L^\infty(\T^2; \R^2), \dvg v \in L^\infty(\T^2), \,
|v(x)|_\infty\leq 1 \, {\rm a.e.}  \Big \}.
\end{align*}
Furthermore, for each $u \in \mathcal{A}_N^b$ there exists a $v\in
L^\infty(\T^2; \R^2)$ such that $|v(x)|_\infty\leq 1$ a.e.,
\[
-\int_{\T^2} u\, \dvg v = \int_{\T^2} |u_x| + |u_y|,
\]
and $\|\dvg v\|_{L^\infty(\T^2)} = 4 N$.
\end{lemma}
The first result in the above lemma says we can relax the condition on
the admissible vector fields in the definition of the anisotropic
total variation to $L^\infty$ vector fields with essentially bounded
divergence.  The second result shows that the supremum is achieved by
a specific vector field if $u\in \mathcal{A}_N^b$.

\begin{lemma}[Upper bound]\label{lem:hupperbound}
Let $u\in L^1(\T^2)$ and $\{N_n\}_{n=1}^\infty\subset\N$ be such that
$N_n \to \infty$ as $n\to \infty$.  Then there exists a sequence
$\{u_n\}_{n=1}^\infty \subset L^1(\T^2)$ such that $u_n \to u$ in
$L^1(\T^2)$ as $n\to \infty$ and
$
h_{\infty,0}(u) \geq \underset{n\to\infty}\limsup\, h_{N_n,0}(u_n).
$
\end{lemma}

\begin{proof}[\bf Proof.]
Without loss of generality we can assume that $u\in BV(\T^2; \{0,
1\})$.  Construct $u_n$ as follows.  For $x\in S_{N_n}^{i,j}$ define
\[
u_n(x):= \left\{\begin{array}{ll} 1 &\text{if }
 (S_{N_n}^{i,j} )^\circ \subset \supp u,\\
0 &\text{otherwise}.  \end{array}\right.
\]
In words, $u_n$ takes the value one on those squares of the grid whose
interior lies completely in the support of $u$ and zero on the other
squares.  By $\partial_* \supp u$ denote the reduced boundary of the
set $\supp u$, \textit{i.e.,} all points in $\partial \supp u $ for
which there is a well defined normal vector (see \cite[Definition
3.54]{AmbrosioFuscoPallara00}).  Since $\supp u_n$ is the maximal set
which is both a union of squares on the grid and is contained in
$\supp u$, the difference in area between $\supp u$ and $\supp u_n$ is
bounded by the length of the reduced boundary of $\supp u$ times the
area of a square,  i.e.,
\[
\int_{\T^2} |u_n-u| \leq \mathcal{H}^1(\partial^* \supp u) N_n^{-2}.
\]
Because $u\in BV(\T^2; \{0,1\})$ the set $\supp u$ has finite
perimeter and hence $u_n\to u$ in $L^1(\T^2)$.

For each $u_n$ let $v_n\in L^\infty(\T^2;\R^2)$ be the vector field
whose existence is guaranteed by Lemma~\ref{lem:existencevectorfield},
then
\begin{equation}\label{eq:usebarcodevectorfield}
\int_{\T^2} |(u_n)_x|+|(u_n)_y| = - \int_{\T^2} u_n \, \dvg v_n =
-\int_{\T^2} u\, \dvg v_n + \int_{\T^2} (u-u_n)\, \dvg v_n.
\end{equation}
For the last term we have
\begin{align*}
 \Big  | \int_{\T^2} (u-u_n)\, \dvg v_n \Big |
&
 \leq
\|u-u_n\|_{L^1(\T^2)} \|\dvg v_n\|_{L^\infty(\T^2)}  \\
&
\leq 4
\mathcal{H}^1(\partial^* \supp u) N_n^{-1} \to 0, \text{as } n\to
\infty.
\end{align*}
Hence, by the first statement of Lemma~\ref{lem:existencevectorfield}
we have
\begin{align*}
\underset{n\to\infty}\limsup\, \int_{\T^2} |(u_n)_x|+|(u_n)_y|  &
=
\underset{n\to\infty}\limsup \Big (-\int_{\T^2} u\, \dvg v_n \Big )
\\
&
\leq \underset{n\to\infty}\limsup\, \int_{\T^2} |u_x| + |u_y| =
\int_{\T^2} |u_x| + |u_y|,
\end{align*}
which proves the result.
\end{proof}

\noindent
{\bf Proof of Theorem~\ref{thm:GammaconvergenceforhN0}.}
Combining Lemmas~\ref{lem:hlowerbound} and~\ref{lem:hupperbound}
proves the $\Gamma$-convergence result in
Theorem~\ref{thm:GammaconvergenceforhN0}.
\qed

\smallskip

\begin{remark}\label{rem:LpinsteadofL1}
Note that we could have used any $L^p$ space instead of $L^1$ in the
results above.  Because $\T^2$ is bounded, convergence in $L^p$
implies convergence in $L^1$ and so the result of
Lemma~\ref{lem:hlowerbound} is easily recovered.  For
Theorem~\ref{thm:compactnessforh} and Lemma~\ref{lem:hupperbound} we
note that because $u_n$ and $u$ are binary functions taking values 0
and 1 a.e., the bound on their $L^p$ difference is the same as that on
their $L^1$ difference and the results follow again.
\end{remark}

We end this section with an illustration of the preference for squares
and rectangles of $h_{N,0}$.

\begin{lemma}[Minimizers of $h_{N,0}$]\label{lem:minimizehN0}
Let $M\in [0,1]$ be such that $N^2 M = K^2$ for some $K\in \N$.

If $M\in [0,\frac14)$, then $u_0$ is a minimizer of $h_{N,0}$ over all
$u\in\mathcal{A}_N^b$ that satisfy $\int_{\T^2} u = M$ if and only if
$u_0$ is the characteristic function of a square.

If $M\in (\frac14,1)$ then $u_0$ is a minimizer if and only if it is
the characteristic function of a rectangle of the form
$R=[a,b]\times[0,1]\subset \T^2$ or $R=[0,1] \times [a,b]\subset\T^2$
for $a, b$ that satisfy the mass constraint.

If $M=\frac14$, $u_0$ is a minimizer if and only if it is the
characteristic function of a square or a rectangle $R$ as above.
\end{lemma}

\begin{proof}[\bf Proof.]
First consider the square grid $G_N$ to be embedded in $\R^2$ instead
of in $\T^2$ so that we do not have periodic boundary conditions on
$[0,1]^2$.  Let $u$ be the characteristic function of a set $\Omega$.
We can assume $\Omega$ is connected, because else $h_{N,0}(u)$ can be
lowered by rearranging $\Omega$ to be connected without changing the
mass.  Let $u_0$ have the square with sides of length $L_0=K N^{-1} =
\sqrt{M}$ as support and let $\Omega$ be contained in a rectangle with
sides of lengths $L$ and $B$.  Then $\int_{[0,1]^2} u \leq LB$ and
hence
\[
h_{N,0}(u_0) = 4 L_0 = 4  \Big ({\int_{\T^2} u}
\Big ) ^{1/2} \leq 4 \sqrt{LB} \leq 2
(L+B) \leq h_{N,0}(u),
\]
with equality if and only if $L=B=L_0$.  Hence characteristic
functions of squares are the minimizers of $h_{N,0}$ if we ignore
periodic boundary conditions.

However, on the periodic torus if $\Omega$ is a rectangle we can use
periodicity to eliminate two sides of the rectangle.  This can be done
only if the other two sides have length 1 and hence $h_{N,0}(u)=2$.
This beats the square if $4 L_0 = 4 \sqrt{M} > 2$.
\end{proof}

It is worth noting here that this asymptotic behavior of $h_{N,\e}$ is
not accessible via numerical simulations of a gradient flow, since it
is dependent on $\e$ being small enough for the minimizers to be
essentially binary and hence not differentiable.

\subsection{Simultaneous scaling $\Gamma$-limit for
$h_N^\alpha$}\label{sec:simulscalehNalpha}

In Section~\ref{sec:proofh} we first took the limit $\e \to 0$ for
$h_{N,\e}$ before letting $N\to\infty$.  In this section we will
consider the limit if we let both parameters go to their limit
simultaneously.  To this end we choose $\e=N^{-\alpha}$ for $\alpha>0$
and consider the limit $N\to\infty$ of $h_N^\alpha$ in (\ref{eq:hN}).
We identify $\mathcal{V}_N$ with $\mathcal{A}_N$ in the sense of
Section~\ref{sec:proofk} and extend $h_N^\alpha$ to all of $L^1(\T^2)$
by setting $h_N^\alpha(u) := +\infty$ for $u\in
L^1(\T^2)\setminus\mathcal{A}_N$.

We show that the $\Gamma$-limit is again given by $h_{\infty, 0}$ if
$\alpha$ is large enough, depending on the growth rate $\beta$ of $W$
around its wells given in assumption $(W_1)$.

\begin{theorem}[$\Gamma$-convergence]\label{thm:GammaconvergenceforhNalpha}
Assume that $W$ satisfies condition $(W_1)$ for some $\beta>0$ in
$(\ref{eq:Wawayzero})$ and let $\alpha > \beta$, then
$
h_N^\alpha \overset{\Gamma}{\to} h_{\infty, 0} $
as $ N\to
\infty$
in the $L^1(\T^2)$ topology.
\end{theorem}

\begin{theorem}[Compactness]\label{thm:compactnessforHN}
Assume that $W$ satisfies condition $(W_1)$ for some $\beta>0$ in
$(\ref{eq:Wawayzero})$ and let $\alpha > \beta$.  Let
$\{N_n\}_{n=1}^\infty\subset\N$ satisfy $N_n\to \infty$ as
$n\to\infty$ and let $\{u_n\}_{n=1}^\infty \subset L^1(\T^2)$ be a
sequence for which there is a constant $C>0$ such that for all
$n\in\N$
$
h_{N_n}^\alpha(u_n) \leq C.
$
Then there exists a subsequence $\{u_{n'}\}_{n'=1}^\infty \subset
\{u_n\}_{n=1}^\infty$ and a $u\in BV(\T^2; \{0,1\})$ such that
$u_{n'}\to u$ in $L^1(\T^2)$ as $n'\to\infty$.
\end{theorem}

We first prove the lower bound for the $\Gamma$-limit.

\begin{lemma}[Lower bound]\label{lem:hlowerboundallNs}
Let $u\in L^1(\T^2)$ and let $\{u_n\}_{n=1}^\infty \subset L^1(\T^2)$
and $\{N_n\}_{n=1}^\infty\subset\N$ be such that $u_n \to u$ in
$L^1(\T^2)$ and $N_n \to \infty$ as $n\to \infty$.  Assume that $W$
satisfies condition $(W_1)$ for some $\beta>0$ in
(\ref{eq:Wawayzero}), and let $\alpha > \beta$.  Then
\[
h_{\infty,0}(u) \leq \underset{n\to\infty}\liminf\,
h_{N_n}^\alpha(u_n).
\]
\end{lemma}

\begin{proof}[\bf Proof.]
First consider the case where $u\in BV(\T^2; \{0,1\})$.  Without loss
of generality we may assume that $h_{N_n}^\alpha$ is uniformly
bounded.  Since $W$ is nonnegative we deduce there is a $C>0$ such
that for all $(i,j) \in I_{N_n}^2,\,\,$ $W((u_n)_{i,j}) \leq C
N_n^{-\alpha}$.  Together with (\ref{eq:Wawayzero}) in assumption
($W_2)$ this implies that for $n$ large enough and all $(i,j) \in
I_{N_n}^2$ we have $(u_n)_{i,j} \in \hat I$.  In addition, by the
growth condition in (\ref{eq:Wawayzero}) we get for $n$ large enough
and $s\in\{0,1\}$
\[
c_s |(u_n)_{i,j}-s|^\beta \leq W((u_n)_{i,j}) \leq C N_n^{-\alpha}.
\]
Hence, if we define $\delta_n := (C/\min(\{c_0, c_1\}))^{\frac1\beta}
N_n^{-\frac\alpha\beta}$ we deduce that for $n$ large enough
$\delta_n\in(0,\frac12)$ and for all $(i,j) \in I_{N_n}^2,\,\,$
$(u_n)_{i,j} \in (-\delta_n, \delta_n) \cup (1-\delta_n, 1+\delta_n)$.
Define
\begin{align}
X_n &:= \big\{ (i,j)\in I_{N_n}^2: \big((u_n)_{i,j} \in (-\delta_n,
\delta_n) \wedge (u_n)_{i+1,j} \in (1-\delta_n,
1+\delta_n)\big)\notag\\
 &\hspace{1cm}\vee \big((u_n)_{i+1,j} \in (-\delta_n, \delta_n) \wedge
 (u_n)_{i,j} \in (1-\delta_n,
 1+\delta_n)\big)\big\},\label{eq:definitionXn}
\end{align}
then for $(i,j)\in X_n$ we have $|(u_n)_{i+1,j}-(u_n)_{i,j}| \geq
1-2\delta_n$, hence
\begin{align}\label{eq:splitupsums}
& N_n^{-1} \sum_{i, j=1}^{N_n}((u_n)_{i+1, j}-(u_n)_{i,j})^2
 \\
\geq & N_n^{-1} (1-2 \delta_n)  \! \! \! \sum_{(i,j)\in X_n}
|(u_n)_{i+1,j}-(u_n)_{i,j}| + N_n^{-1}
\! \! \!   \sum_{(i,j)\in X_n^c}
\left((u_n)_{i+1,j}-(u_n)_{i,j}\right)^2\notag\\
= & N_n^{-1} \sum_{i,j=1}^{N_n} |(u_n)_{i+1,j}-(u_n)_{i,j}| - 2
\delta_n N_n^{-1} \sum_{(i,j)\in X_n} |(u_n)_{i+1,j}-(u_n)_{i,j}|
\notag\\
&\hspace{0.2cm}+ N_n^{-1} \sum_{(i,j)\in X_n^c}
\left((u_n)_{i+1,j}-(u_n)_{i,j}\right)^2 - N_n^{-1} \sum_{(i,j)\in
X_n^c} |(u_n)_{i+1,j}-(u_n)_{i,j}|.\notag
\end{align}
For the second summation in (\ref{eq:splitupsums}) we note that there
are $c>0$, $\tilde C>0$, such that
\begin{align*}
0 &\leq 2 \delta_n N_n^{-1} \sum_{(i,j)\in X_n}
|(u_n)_{i+1,j}-(u_n)_{i,j}| \leq c \delta_n N_n^{-1} N_n^2 (1+2
\delta_n)\\
&\leq \tilde C (N_n^{1-\alpha/\beta}+N_n^{1-2 \alpha/\beta}) \to 0
\quad \text{as } n\to\infty.
\end{align*}
The third and fourth summation in (\ref{eq:splitupsums}) we combine
into
\begin{align*}
&\hspace{0.5cm}\Big| N_n^{-1} \sum_{(i,j)\in X_n^c}
|(u_n)_{i+1,j}-(u_n)_{i,j}|
\big(|(u_n)_{i+1,j}-(u_n)_{i,j}|-1\big)\Big| \\
&\leq N_n^{-1} \sum_{(i,j)\in X_n^c} |(u_n)_{i+1,j}-(u_n)_{i,j}|\,\,
\big||(u_n)_{i+1,j}-(u_n)_{i,j}|-1\big|\\
&\leq N_n^{-1} N_n^2 2 \delta_n (2 \delta_n - 1) \leq \tilde C
(N_n^{1-2\alpha/\beta}-N_n^{1-\alpha/\beta}) \to 0 \quad \text{as }
n\to\infty,
\end{align*}
for some $\tilde C>0$.  As in Section~\ref{sec:proofk} we now identify
$\mathcal{V}_{N_n}$ with $\mathcal{A}_{N_n}$ to write for the first
summation in (\ref{eq:splitupsums}) 
$$
\displaystyle N_n^{-1}
\sum_{i,j=1}^{N_n} |(u_n)_{i+1,j}-(u_n)_{i,j}| = \int_{\T^2}
|(u_n)_x|, 
$$
 hence from (\ref{eq:splitupsums}) and the computations
that followed we deduce
\begin{equation}\label{eq:sumtoanisotvx}
N_n^{-1} \sum_{i, j=1}^{N_n}((u_n)_{i+1,j}-(u_n)_{i,j})^2 \geq
\int_{\T^2} |(u_n)_x| + \mathcal{O}(N_n^{1-\alpha/\beta}).
\end{equation}
Analogously we have 
$$
\displaystyle N_n^{-1} \sum_{i,
j=1}^{N_n}((u_n)_{i, j+1}-(u_n)_{i,j})^2 \geq \int_{\T^2} |(u_n)_y| +
\mathcal{O}(N_n^{1-\alpha/\beta}).  
$$
 By the lower semicontinuity of
the anisotropic total variation with respect to $L^1$ convergence
\cite[Lemma A.5]{ChoksivanGennipOberman11} we have 
$$
\int_{\T^2} |u_x|+|u_y| \leq \underset{n\to\infty}\liminf\,
\int_{\T^2} |(u_n)_x|+|(u_n)_y|, 
$$
 hence
\begin{align*}
\int_{\T^2}  \! |u_x|+|u_y| 
&
\leq  \!  \underset{n\to\infty}\liminf  N_n^{-1}
\sum_{i, j=1}^{N_n}((u_n)_{i+1,j}-(u_n)_{i,j})^2 +
((u_n)_{i,j+1}-(u_n)_{i,j})^2  \\
&
\leq \underset{n\to\infty}\liminf\,
h_{N_n}^\alpha(u_n),
\end{align*}
which proves the lower bound for the case where $u\in BV(\T^2;
\{0,1\})$.

Now consider $u\in L^1(\T^2) \setminus BV(\T^2; \{0,1\})$.  Let
$u_n\to u$ in $L^1(\T^2)$, then an argument from contradiction
(similar to that at the end of the proof of
Lemma~\ref{lem:hlowerbound}) and the compactness proved below in
Theorem~\ref{thm:compactnessforHN} prove that $ h_{\infty,0}(u) =
\infty = \underset{n\to\infty}\liminf\, h_{N_n}^\alpha(u_n).  $
\end{proof}

Next we prove the upper bound.

\begin{lemma}[Upper bound]\label{lem:hupperboundallNs}
Let $u\in L^1(\T^2)$ and $\{N_n\}_{n=1}^\infty\subset\N$ be such that
$N_n \to \infty$ as $n\to \infty$.  Then there exists a sequence
$\{u_n\}_{n=1}^\infty \subset L^1(\T^2)$ such that $u_n \to u$ in
$L^1(\T^2)$ as $n\to \infty$ and
\[
h_{\infty,0}(u) \geq \underset{n\to\infty}\limsup\,
h_{N_n}^\alpha(u_n).
\]
\end{lemma}

\begin{proof}[\bf Proof.]
Without loss of generality we can assume that $u\in BV(\T^2; \{0,
1\})$.  Construct $u_n\in\mathcal{A}_{N_n}$ as follows.  For $x\in
S_{N_n}^{i,j}$ define, as in the proof of Lemma~\ref{lem:hupperbound},
\[
u_n(x):= \left\{\begin{array}{ll} 1 &\text{if }
 (S_{N_n}^{i,j} )^\circ \subset \supp u,\\
0 &\text{otherwise}.  \end{array}\right.
\]
Then $\big((u_n)_{i+1,j}-(u_n)_{i,j}\big)^2 =
|(u_n)_{i+1,j}-(u_n)_{i,j}|$ and hence identifying $\mathcal{V}_{N_n}$
with $\mathcal{A}_{N_n}$
\[
N_n^{-1} \sum_{i,j=1}^{N_n} \big((u_n)_{i+1,j}-(u_n)_{i,j}\big)^2 =
N_n^{-1} \sum_{i,j=1}^{N_n} |(u_n)_{i+1,j}-(u_n)_{i,j}| = \int_{\T^2}
|(u_n)_x|.
\]
Similarly 
$$
 N_n^{-1} \sum_{i,j=1}^{N_n}
\big((u_n)_{i,j+1}-(u_n)_{i,j}\big)^2 = \int_{\T^2} |(u_n)_y|.  
$$
Since every $u_n$ takes values in $\{0,1\}$ we can repeat the argument
from the proof of Lemma~\ref{lem:hupperbound} in and following
(\ref{eq:usebarcodevectorfield}) to prove that
\[
\underset{n\to\infty}\limsup\, \int_{\T^2} |(u_n)_x| + |(u_n)_y| \leq
\int_{\T^2} |u_x| + |u_y|.
\]
Since for every $n\in\N$ and $(i,j)\in I_{N_n}^2\,,$
$W((u_n)_{i,j})=0$ we get the desired result.
\end{proof}

\noindent
{\bf Proof of Theorem~\ref{thm:GammaconvergenceforhNalpha}.}
Combining Lemmas~\ref{lem:hlowerboundallNs}
and~\ref{lem:hupperboundallNs} we get the $\Gamma$-convergence result
in Theorem~\ref{thm:GammaconvergenceforhNalpha}.
\qed

Next we prove compactness.

\smallskip

\noindent
{\bf Proof of Theorem~\ref{thm:compactnessforHN}.}
By the first part of the proof of Lemma~\ref{lem:hlowerboundallNs} we
have, after possibly going to a subsequence, that for all $n\in\N$ and
all $(i,j)\in I_{N_n}^2\,\,$ $(u_n)_{i,j} \in \hat I$, hence
$\|u_n\|_{L^1(\T^2)}$ is uniformly bounded.  By the same proof, in
particular (\ref{eq:sumtoanisotvx}) and the uniform bound on
$h_{N_n}^\alpha(u_n)$, we have for all $n\in\N$ that $ 
\int_{\T^2} |(u_n)_x| + |(u_n)_y|$ is uniformly bounded.  We deduce as
in the proof of Theorem~\ref{thm:compactnessforh} a uniform bound on
the BV norms of $u_n$ from which it follows by the compactness theorem
in BV (\cite[Theorem 1.19]{Giusti84} or \cite[5.2.3 Theorem
4]{EvansGariepy92}) that there exists a subsequence
$\{u_{n'}\}_{n'=1}^\infty \subset \{u_n\}_{n=1}^\infty$ and a $u\in
BV(\T^2)$ such that $u_{n'}\to u$ in $L^1(\T^2)$ as $n'\to\infty$.  By
the arguments in the proof of Lemma~\ref{lem:hlowerboundallNs} each
$u_{n'}$ takes values in $(-\delta_{n'},
\delta_{n'})\cup(1-\delta_{n'}, 1+\delta_{n'})$ where $\delta_{n'}\to
0$ as $n'\to\infty$.  After possibly going to another subsequence
$u_{n'}$ converges pointwise a.e. to $u$, hence $u$ takes values in
$\{0,1\}$ almost everywhere.
\qed

\subsection{Constraints}

In this section we show that addition of a fidelity term to the
functional or imposing a mass constraint are compatible with the three
$\Gamma$-limits we discussed, \textit{i.e.,} $\e\to 0$ for $h_{N,\e}$,
$N\to \infty$ for $h_{N,0}$, and $N\to\infty$ for $h_N^\alpha$.

\begin{theorem}[Constraints]\label{thm:constraintsforh} 
\label{item:hconstraintsa} 
$(1)$  $\displaystyle h_{N,\e} + \lambda
N^{-2} |\cdot-f|_p^p \overset{\Gamma}{\to} h_{N,0} + \lambda N^{-2}
|\cdot-f|_p^p$ for $\e\to 0$, where $p\in \N$, $\lambda>0$, and a
given function $f\in \mathcal{V}_N$ $($or possibly a given function $f:
U\to \R$ where $U$ is a strict subset of the vertex set $V$ and the
sum in $|\cdot-f|_p^p$ is restricted to vertices in $U ).$    A
compactness result for $h_{N,\e} + \lambda N^{-2} |\cdot-f|_p^p$ as in
Theorem~$\ref{thm:compactnessongraph}$ holds.

    If instead, for fixed $M>0$, the domain of definition of
    $h_{N,\e}$ is restricted to $\mathcal{V}_N^M$ $($\textit{i.e.,}
    $\mathcal{V}^M$ from Theorem~$\ref{thm:constraintsforfe}$ on the
    grid $G_N )$ where $M$ is such that $N^2M$ is an integer between 0
    and $N^2$, then the $\Gamma$-convergence and compactness results
    for $\e\to 0$ remain valid, with the domain of $h_{N,0}$
    restricted to $\mathcal{V}^M$ as well.

\label{item:hconstraintsb}
$(2)$ 
Let $p\in\N$, $\lambda>0$, $f\in
C^1(\T^2)$ and $f_N\in \mathcal{A}_N$ the sampling of $f$ on the grid
$G_N$ $( f$, $f_N$ and their norms can also be defined on subsets of
$\T^2$ and $G_N$ as in part~$\ref{item:hconstraintsa}),$ then
$\displaystyle h_{N,0} + \lambda N^{-2} |\cdot-f_N|_p^p
\overset{\Gamma}{\to} h_{\infty,0} + \lambda
\|\cdot-f\|_{L^p(\T^2)}^p$ as $N\to\infty$ in the $L^p(\T^2)$
topology.  A compactness result as in
Theorems~$\ref{thm:compactnessforh}$ and Remark~$\ref{rem:LpinsteadofL1}$
holds for $h_{N,0} + \lambda N^{-2} |\cdot-f_N|_p^p$.

    If instead the domain of $h_{N,0}$ is restricted to
    $\mathcal{V}_N^M$, for a fixed $M$ such that $N^2 M$ is an integer
    between $0$ and $N^2$, then the compactness and lower bound
    results from Lemma~$\ref{lem:hlowerbound},$
    Theorem~$\ref{thm:compactnessforh},$ and
    Remark~$\ref{rem:LpinsteadofL1}$ remain valid, with the domain of
    $h_{\infty,0}$ restricted to 
    $$
    BV_M(\T^2;\{0,1\}) := \Big \{u\in
    BV(\T^2; \{0,1\}) : \int_{\T^2} u = M \Big \} .
    $$
      With the same
    restriction, the upper bound result from
    Lemma~$\ref{lem:hupperbound}$ is still valid if we restrict it to
    sequences $\{N_n\}_{n=1}^\infty\subset\N$ such that $N_n^2 M$ is
    an integer for each $n\in\N$.

$(3)$ 
If $\lambda$, $f$, and $f_N$ are as in
part~$\ref{item:hconstraintsa}$ and $\alpha>\beta$ as in
Theorem~$\ref{thm:GammaconvergenceforhNalpha},$ then $\displaystyle
h_N^\alpha + \lambda N^{-2} |\cdot -f_N|_1 \overset{\Gamma}\to
h_N^\alpha + \lambda \|\cdot - f\|_{L^1(\T^2)}$ as $N\to\infty$ in the
$L^1(\T^2)$ topology.  A compactness result as in
Theorem~$\ref{thm:compactnessforHN}$ holds for $h_N^\alpha + \lambda
N^{-2} |\cdot -f_N|_1$.

    If instead the domain of $h_N^\alpha$ is restricted to
    $\mathcal{V}_N^M$ for a fixed $M\in [0,1]$, then the
    $\Gamma$-convergence and compactness results from
    Theorems~$\ref{thm:GammaconvergenceforhNalpha}$
    and~$\ref{thm:compactnessforHN}$ remain valid, with the domain of
    $h_{\infty,0}$ restricted to $BV_M(\T^2;\{0,1\})$.

\end{theorem}

We give a sketch of the proofs.

(1)
This follows directly from Theorem~\ref{thm:constraintsforfe}.

(2)
For the fidelity term first we note that for $v, f_N\in
\mathcal{A}_N$, 
$$
N^{-2} |v-f_N|_p^p = \int_{\T^2} |v-f_N|^p .
$$  Hence,
for $v \in \mathcal{A}_N$, $u\in L^p(\T^2)$, $f \in C(\T^2)$ and
$f_N\in \mathcal{A}_N$ the discretization of $f$ on $G_N$ we have
\begin{align*}
\Big |\int_{\T^2} |v-f_N|^p  & - \int_{\T^2} |u-f|^p \Big |  \leq
\int_{\T^2} \big| |v-f_N| - |u-f| \big| 
\\
&
\leq
\|v-f_N-u+f\|_{L^p(\T^2)}^p  
 \leq \|v-u\|_{L^p(\T^2)}^p + \|f-f_N\|_{L^p(\T^2)}^p.
\end{align*}
Since (by a Taylor series argument) $f_N \to f$ in $L^p(\T^2)$ as
$N\to \infty$ the sequence of fidelity terms is continuously
convergent and thus by Lemma~\ref{lem:addperturbations} the
$\Gamma$-limit follows.  Clearly the compactness isn't harmed (even
helped) by adding an extra term to the functional.

For the mass constraint, since $u\in\mathcal{A}_N^b$ the conditions on
$M$ are necessary as explained in
Section~\ref{sec:explainconstraints}.  $L^p(\T^2)$ convergence
preserves average mass and hence the constraint is compatible with
(LB) and compactness.  For (UB) the recovery sequence
$\{u_n\}_{n=1}^\infty$ used in the proof of
Lemma~\ref{lem:hupperbound} has support contained in the support of
$u$ and hence if $\int_{\T^2} u = M$, then $\int_{\T^2} u_n \leq M$.
We can construct a similar recovery sequence $\{\hat
u_n\}_{n=1}^\infty$, where $\hat u_n$ has support on all grid squares
which intersect $\supp u$.  This sequence satisfies all the required
properties of a recovery sequence and has $\int_{\T^2} u_n \geq M$.
Hence, under the assumption on $\{N_n\}_{n=1}^\infty$ which assures
that the mass condition can be satisfied for each $N_n$, there exists
a recovery sequence $\{\overline u_n\}_{n=1}^\infty$ where $\overline
u_n$ takes the value $1$ on $\supp u_n$ as well as on a select chosen
number of squares which lie in $\supp \hat u_n \setminus \supp u_n$.
For these combinations of $\{N_n\}_{n=1}^\infty$ and $M$ (UB) is
compatible with the mass constraint as well.

(3)
For the fidelity term we can use the same arguments as above.

For the mass constraint we note that now $u_n\in \mathcal{A}_{N_n}$
and the limit function $u\in BV(\T^2; \{0,1\})$.  This means each
choice $M\in [0,1]$ is allowed in the mass constraint.  As above,
because of the $L^1$ convergence this mass constraint is compatible
with (LB) and compactness.  For (UB) we note that the proof of
Lemma~\ref{lem:hupperboundallNs} followed Lemma~\ref{lem:hupperbound},
so our argument here is very similar to that for $h_{N,0}$ above, with
the added bonus that we do not need to restrict ourselves to specific
combinations of $\{N_n\}_{n=1}^\infty$ and $M$.  Let the recovery
sequences $\{u_n\}_{n=1}^\infty$ and $\{\hat u_n\}_{n=1}^\infty$ be as
above.  Now construct another recovery sequence $\{\overline
u_n\}_{n=1}^\infty$ by setting $\overline u_n = u_n$ on $\supp u_n$
and $\overline u_n = c_n \hat u_n$ on $(\supp u_n)^c$, where $c_n\in
[0,1]$ is chosen such that for each $n$ the average mass constraint is
satisfied.

\section{$\Gamma$-limits for the the discretized  \\
Ginzburg-Landau
functional $k_{N,\e}$}\label{sec:differentscalingskNe}

In this section we will study the convergence properties of $k_{N,\e}$
from (\ref{eq:kNe}).  We first take the limit $N\to\infty$ and then
$\e\to 0$.  The resulting $\Gamma$-limits are given in
Section~\ref{sec:proofk}.  The simultaneous limit, obtained by
substituting $\e=N^{-\alpha}$ for well chosen $\alpha>0$ and then
considering the limit $N\to\infty$ for $k_N^\alpha$ in (\ref{eq:hN})
is studied in Section~\ref{sec:simulscalekNalpha}.

There has been a series of (recent) papers dealing with convergence of
discrete energies to integral energies among which
\cite{BraidesGelli02,AlicandroCicalese04,BraidesGelli06,AlicandroCicaleseGloria07,AlicandroBraidesCicalese08,AlicandroCicaleseGloria10},
all expanding on the ideas in \cite[Chapter 4]{Braides02}.  For many
of the proofs we will use ideas from \cite{AlicandroCicalese04} in
which the authors prove a $\Gamma$-limit of integral form exists for a
general class of grid based functionals.  Here we study a specific
functional and hence can prove more explicit results.

\begin{remark}\label{rem:kgendim}
In Sections~\ref{sec:proofk} and~\ref{sec:simulscalekNalpha} we prove
the results for $\T^2$, but they can easily be generalized to $\T^d$
for any $d\in \N$, if we let the scaling factor in the first term of
$k_{N,\e}$ in (\ref{eq:kNe}) be $\e N^{2-d}$ instead of $\e N^{0}$ and
the factor in the second term $\e^{-1} N^{-d}$ instead of $\e^{-1}
N^{-2}$ and we change $k_N^\alpha$ in (\ref{eq:kN}) accordingly:
\begin{align}
k_{N,\e}(u) &:= \! \e N^{2-d} \! \sum_{i, j=1}^N\! 
(u_{i+1,j}-u_{i,j})^2 \!   
+(u_{i,j+1}-u_{i,j})^2 \! + \e^{-1} N^{-d}
\sum_{i,j=1}^N W(u_{i,j}),  \\
k_N^\alpha(u) &:= \! N^{2-d-\alpha}\!  \sum_{i, j=1}^N\! 
(u_{i+1,j}-u_{i,j})^2   
+(u_{i,j+1}-u_{i,j})^2 + N^{\alpha-d}
\sum_{i,j=1}^N W(u_{i,j})\label{eq:kN}.
\end{align}
The admissible range of $\alpha$ in Theorems~\ref{thm:kNcompactness}
and~\ref{thm:GammaconvergenceforkNalpha} is dependent on the dimension
$d$ in a way which will be made precise in Remark~\ref{rem:alphadim}.
For the extra fidelity term in Theorem~\ref{thm:constraintsfork} the
scaling factor then needs to be $N^{-d}$ instead of $N^{-2}$.
\end{remark}

\subsection{Sequential $\Gamma$-convergence and compactness:
first $N\to\infty$, then $\e\to 0$}\label{sec:proofk}

We will prove $\displaystyle k_{N, \e} \overset{\Gamma}{\to}
k_{\infty,\e}$ as $N\to \infty$, where $k_{\infty, \e}$ is defined for
$u\in L^1(\T^2)$ as
\[
k_{\infty,\e}(u) := \left\{\begin{array}{ll} 
 \e \int_{\T^2} |\nabla
u|^2 + \e^{-1} \int_{\T^2} W(u) &\text{if } u \in W^{1,2}(\T^2),\\
+\infty &\text{otherwise}.  \end{array}\right.
\]
We see that $k_{\infty, \e}$ is the Ginzburg-Landau functional from
(\ref{eq:GinzburgLandau}).  As explained in
Section~\ref{sec:continuumGL} it is known that this functional
$\Gamma$-converges in either the $L^1(\T^2)$ or $L^2(\T^2)$
topology\footnote{Results are usually stated in the $L^1$ topology,
but for example \cite{RenWei00,Braides02,ChoksiRen05} note that the
results can be stated in the $L^2$ topology as well.} as $\e\to 0$ to
the total variation (\ref{eq:surfacetension}).  To be precise its
$\Gamma$-limit is
\[
k_{\infty,0}(u): = \left\{\begin{array}{ll} \sigma(W) \int_{\T^2}
|\nabla u| &\text{if } u \in BV(\T^2; \{0,1\}),\\
+\infty &\text{otherwise}, \end{array}\right.
\]
where $  \sigma(W):= 2 \int_0^1 \sqrt{W(s)}\,ds>0$ is a
constant depending on the specific form of $W$, in particular on the
transition between its wells, \cite{Modica87a}.  The sequence of
functionals is equi-coercive as well.

For fixed $u\in C^3(\T^2)$ we have pointwise convergence (in
$C^3(\T^2)$) $k_{N,\e}(u) \to k_{\infty,\e}(u)$ as $N\to\infty$ (more
details below), but the dependence of the discretization errors on
derivatives of $u$ prevents us from concluding uniform convergence.
$\Gamma$-convergence offers a useful middle ground between pointwise
and uniform convergence and can thus be seen as an extension of
classical numerical analysis results.  The pointwise convergence
follows from combining the boundedness of $\T^2$ with the trapezoidal
rule (for fixed $i$)
\[
\int_\T f(x,y)\,dy = N^{-1}\sum_{j=1}^N f(x, j/N) + \frac1{12} N^{-2}
\frac{\partial^2 f}{\partial y^2}(x, \xi), \quad \text{for } f\in
C^2(\T^2),
\]
for some $f$-dependent $\xi\in (0,1)$, and the finite difference
approximation of the derivative
\begin{align} 
\label{eq:Taylorfinitediff} 
\frac{\partial u}{\partial x}(i/N, j/N) 
&
= N\big[u((i+1)/N, j/N) -
u(i/N, j/N)\big] \\
&
\notag
- \frac12 N^{-1} \frac{\partial^2 u}{\partial
x^2}((i+r_i)/N, j/N)
\end{align}
for some $r_i\in [0,1]$.  We see the dependence of the errors on
derivatives of the functions $f$ and $u$.

We prove the following $\Gamma$-convergence and compactness results.

\begin{theorem}[$\Gamma$-convergence]\label{thm:GammaconvergenceforkNe}
$\displaystyle k_{N, \e} \overset{\Gamma}{\to} k_{\infty, \e}$ as
$N\to \infty$ in the $L^1(\T^2)$ or $L^2(\T^2)$ topology.
\end{theorem}

\begin{theorem}[Compactness]\label{thm:compactnessforkNe}
Assume $W$ satisfies $(W_2)$.  Let $\{N_n\}_{n=1}^\infty\subset\N$
satisfy $N_n \to \nobreak \infty$ as $n\to\infty$ and let
$\{u_n\}_{n=1}^\infty \subset L^1(\T^2)$ be a sequence for which there
is a constant $C>0$ such that for all $n\in\N$
$
k_{N_n,\e}(u_n) \leq C.
$
Then there exists a subsequence $\{u_{n'}\}_{n'=1}^\infty \subset
\{u_n\}_{n=1}^\infty$ and a $u\in W^{1,2}(\T^2)$ such that $u_{n'}\to
u$ in $L^2(\T^2)$ as $n'\to\infty$.
\end{theorem}

The proof of $\Gamma$-convergence adapts the ideas that are developed
in an abstract general framework in \cite{AlicandroCicalese04} to our
situation.  In Section~\ref{sec:proofh} we constructed a bijection
between $\mathcal{V}_N$ and $\mathcal{A}_N$ from (\ref{eq:AN}).  In
what follows we will identify $u\in \mathcal{V}_N$ with its
counterpart $\tilde u\in \mathcal{A}_N$ and drop the tilde if no
confusion arises.

For $u\in \mathcal{A}_N \subset L^1(\T^2)$ we define pointwise
evaluation by identifying each $u$ with its representative which is
piecewise constant on the squares $S_N^{i,j}$ and for which pointwise
evaluation is well-defined.  For $z\in \T^2$ we define the difference
quotients as
\begin{equation}\label{eq:differencequotient}
D_N^k u(z) := N \big[u(z+e_k/N)-u(z)\big], \quad k\in\{1,2\},
\end{equation}
where $e_k$ denotes the $k^{\text{th}}$ standard basis vector of
$\R^2$.

With the identification between $\mathcal{V}_N$ and $\mathcal{A}_N$ we
extend the functional to all $u\in L^1(\T^2)$ as follows
\begin{equation}\label{eq:kNereformulation}
k_{N,\e}(u) = \left\{\begin{array}{ll} 
 \e \int_{\T^2} \left[(D_N^1u)^2
+ (D_N^2u)^2\right] + \e^{-1} \int_{\T^2} W(u) & \text{if } u \in
\mathcal{A}_N,\\
+\infty & \text{otherwise}.  \end{array}\right.
\end{equation}
In the proof of the $\liminf$ inequality we will use the slicing
method, \cite[Chapter 15]{Braides02}, \cite{Braides06}, which uses the
following notation.  Remembering $\T^2 \cong [0,1)^2$ we define
$
\T^2_1 := \{(0,y)\in \T^2: \exists t\in [0,1): (t,y) \in \T^2\} $
 and $ \T^2_2 := \{(x,0)\in \T^2: \exists t\in [0,1): (x,t)
\in \T^2\}
$
and for $(0,y)\in \T^2_1$, $(x,0)\in \T^2_2$ we define the sets
$
\T^2_{1,y} := \{t\in [0,1): (t,y) \in \T^2\} $ and $
\T^2_{2,x} := \{t\in [0,1): (x,t) \in \T^2\}
$
and the functions $u_{1,y}(t) := u(t,y)$ and $u_{2,x}(t) := u(x,t)$ on
$\T^2_{1,y}$ and $\T^2_{2,x}$ respectively.

In what follows $\e>0$ is fixed.

\begin{lemma}[Lower bound]\label{lem:klowerbound}
Let $u\in L^1(\T^2)$ and let $\{u_n\}_{n=1}^\infty \subset L^1(\T^2)$
and $\{N_n\}_{n=1}^\infty\subset\N$ be such that $u_n \to u$ in
$L^1(\T^2)$ and $N_n \to \infty$ as $n\to \infty$.  Then
\[
k_{\infty,\e}(u) \leq \underset{n\to\infty}\liminf\, k_{N_n,\e}(u_n).
\]
\end{lemma}

\begin{proof}[\bf Proof.]
This proof is an application of arguments in \cite[Proposition
3.4]{AlicandroCicalese04}.

First consider the case where $u\in W^{1,2}(\T^2)$, then we can assume
without loss of generality that $\{k_{N_n,\e}(u_n)\}_{n=1}^\infty$ is
uniformly bounded and thus $u_n\in \mathcal{A}_{N_n}$.  For $(x,y)\in
S_{N_n}^{i,j}$ we define
\begin{align*}
v_n^1(x,y)&:= u_n(i/N_n,j/N_n) + D_{N_n}^1 u_n(i/N_n,
j/N_n)\cdot(x-i/N_n),\\
v_n^2(x,y)&:= u_n(i/N_n,j/N_n) + D_{N_n}^2 u_n(i/N_n,
j/N_n)\cdot(y-j/N_n).
\end{align*}
In what follows $k\in\{1,2\}$.  Note $v_n^k \in BV(\T^2)$.  We denote
the densities of the absolutely continuous (with respect to the
Lebesgue measure) part of the measures $\nabla_x v_n^1$ and $\nabla_y
v_n^2$ by $  \frac{\partial v_n^1}{\partial x}$ and
$ \frac{\partial v_n^2}{\partial y}$ respectively.  Then
for $(x,y) \in  (S_N^{i,j} )^\circ$ we have
\[
\frac{\partial v_n^1}{\partial x}(x,y) = D_{N_n}^1 u_n(i/N_n, j/N_n)
\quad\text{and}\quad \frac{\partial v_n^2}{\partial y}(x,y) =
D_{N_n}^2 u_n(i/N_n, j/N_n).
\]
We deduce $v_n^k \to u$ in $L^1(\T^2)$ as $n\to\infty$ from
$$
\|v_n^k-u\|_{L^1(\T^2)} \leq \|v_n^k-u_n\|_{L^1(\T^2)} +
\|u_n-\nobreak u\|_{L^1(\T^2)}.
$$ The latter term converges to zero by
assumption.  The former we bound by $ \|v_n^k-u_n\|_{L^2(\T^2)}$ using
H\"older's inequality.  We then note that from the uniform bound on
$\{k_{N_n,\e}(u_n)\}_{n=1}^\infty$ we have 
$$
\sum_{i,j=1}^{N_n}
\left[D_{N_n}^k u_n (i/N_n, j/N_n)\right]^2 \leq C N_n^2
$$
 for some
$C>0$ and hence (if $k=1$; similarly for $k=2$)
\begin{align*}
&
\int_{\T^2} (v_n^1-u_n)^2  = \sum_{i,j=1}^{N_n} \int_{S_{N_n}^{i,j}}
\left[ D_{N_n}^1 u_n (i/N_n, j/N_n)\,\,(x-i/N_n)\right]^2 \,dy\,dx\\
&
\quad
= N_n^{-1} \sum_{i,j=1}^{N_n} \left[D_{N_n}^1 u_n (i/N_n,
j/N_n)\right]^2 \int_{i/N_n}^{(i+1)/N_n} (x-i/N_n)^2 \,dx \leq
\frac{C}3 N_n^{-2}.
\end{align*}
For $\mathcal{H}^1$-a.e. $y\in \T^2_1$ the slice $(v_n^1)_{1,y} \in
W^{1,2}(\T^2_{1,y})$.  By Fubini's theorem and Fatou's lemma
\begin{align*}
\underset{n\to\infty}\liminf  \int_{\T^2} \Big (\frac{\partial
v_n^1}{\partial x} \Big )^2 
&
= \underset{n\to\infty}\liminf\,
\int_{\T^2_1} \int_{\T^2_{1,y}} |(v_n^1)'_{1,y}|^2 \,dt \,dy 
\\
&\geq
\int_{\T^2_1} \underset{n\to\infty}\liminf\, \int_{\T^2_{1,y}}
|(v_n^1)'_{1,y}|^2 \,dt \,dy,
\end{align*}
Because
\begin{equation}\label{eq:partvnsquared}
\underset{n\to\infty}\liminf\, \int_{\T^2}  \Big (\frac{\partial
v_n^1}{\partial x} \Big )^2\leq
\underset{n\to\infty}\liminf\,k_{N_n,\e}(u_{N_n})<\infty
\end{equation}
we have that, after possibly going to a subsequence, for
$\mathcal{H}^1$-a.e. $y\in \T^2_1$ the sequence
$\{\|(v_n^1)'_{1,y}\|_{L^2(\T^2_{1,y})}\}_{n=1}^\infty$ is bounded and
hence weakly convergent in $L^2(\T^2_{1,y})$.  Since for
$\mathcal{H}^1$-a.e. $y\in \T^2_1$ we have $(v_n^1)_{1,y} \to u_{1,y}$
in $L^1(\T^2_{1,y})$ we identify the limit as $(v_n^1)'_{1,y}
\rightharpoonup u_{1,y}'$ in $L^2(\T^2_{1,y})$, for
$\mathcal{H}^1$-a.e. $y\in \T^2_1$ (see
Lemma~\ref{lem:lemmaforklowerbound} in
Appendix~\ref{sec:deferredproofs} for details).  Hence, by the weak
lower semicontinuity of the $L^2$ norm, Fatou's lemma, and the
completely analogous results for $u_{2,x}$ and $\frac{\partial
v_n^2}{\partial y}$ we get
\[
\underset{n\to\infty}\liminf  \int_{\T^2}  \Big [ \Big (\frac{\partial
v_n^1}{\partial x} \Big )^2 +  \Big (\frac{\partial v_n^2}{\partial
y} \Big )^2 \Big ] \geq \int_{\T^2_1} \int_{\T^2_{1,y}} |u'_{1,y}|^2
\,dt \,dy + \int_{\T^2_2} \int_{\T^2_{2,x}} |u'_{2,x}|^2 \,dt \,dy.
\]
Putting the slices back together using \cite[4.9.2 Theorem
2]{EvansGariepy92}) we deduce that
\begin{equation}\label{eq:nablaubound}
\underset{n\to\infty}\liminf\, \int_{\T^2}  \Big [ \Big (\frac{\partial
v_n^1}{\partial x} \Big )^2 +  \Big (\frac{\partial v_n^2}{\partial
y} \Big )^2 \Big ] \geq \int_{\T^2} |\nabla u|^2.
\end{equation}
For the other term in the functional we use Fatou's lemma, the
continuity of $W$ and the a.e. pointwise convergence of $u_n$ to $u$
to find
\[
\underset{n\to\infty}\liminf\, \int_{\T^2} W(u_n) \geq \int_{\T^2}
\underset{n\to\infty}\liminf\, W(u_n) = \int_{\T^2} W(u).
\]
This proves the result for $u\in W^{1,2}(\T^2)$.

Now consider the case where $u\in L^1(\T^2)\setminus W^{1,2}(\T^2)$.
Assume that 
$$
\underset{n\to\infty}\liminf\, k_{N_n,\e}(u_n) < \infty,
$$
 then after possibly going to a subsequence for each $n\in \N\,\,$
$u_n\in \mathcal{A}_{N_n}$ and the sequences $\left\{\int_{\T^2}
(D_{N_n}^k u_n)^2\right\}_{n=1}^\infty$ are bounded.  We can follow
the same slicing method as applied above up to equation
(\ref{eq:partvnsquared}).  Again we find that for $\mathcal{H}^1$-a.e.
$y\in \T^2_1$ the sequence
$\{\|(v_n^1)'_{1,y}\|_{L^2(\T^2_{1,y})}\}_{n=1}^\infty$ is bounded and
combined with $(v_n^1)_{1,y} \to u_{1,y}$ in $L^1(\T^2_{1,y})$ for
$\mathcal{H}^1$-a.e. $y\in \T^2_1$ we deduce that $u_{1,y}\in
W^{1,2}(\T^2_{1,y})$ for $\mathcal{H}^1$-a.e. $y\in \T^2_1$ (see
Lemma~\ref{lem:lemmaforklowerbound} in
Appendix~\ref{sec:deferredproofs} for details).  We then continue as
above to arrive at (\ref{eq:nablaubound}) and conclude that $u\in
W^{1,2}(\T^2)$, which contradicts the assumption that $u\not\in
W^{1,2}(\T^2)$.  Hence $ \underset{n\to\infty}\liminf\,
k_{N_n,\e}(u_n) = \infty, $ which concludes the proof.
\end{proof}

Next we prove (UB') (see Section~\ref{sec:explainGamma}).

\begin{lemma}[Upper bound]\label{lem:kupperbound}
Let $u\in L^1(\T^2)$ and $\{N_n\}_{n=1}^\infty\subset\N$ be such that
$N_n \to \infty$ as $n\to \infty$ and let $k_\e''$ be the
$\Gamma$-upper limit of $k_{N_n,\e}$ as $n\to\infty$ with respect to
the $L^2(\T^2)$ topology, then
$
k_\e''(u) \leq k_{\infty, \e}(u).
$
\end{lemma}

\noindent
{\bf Proof.}
This proof is an adaptation of the ideas in \cite[Proposition
3.5]{AlicandroCicalese04}.

The case where $u\in L^1(\T^2)\setminus W^{1,2}(\T^2)$ is trivial.
For the case where $u\in W^{1,2}(\T^2)$ we first assume that $u\in
C^\infty(\T^2)$.  Define a sequence $\{u_n\}_{n=1}^\infty$ such that
for each $n\in\N\,\,$ $u_n\in \mathcal{A}_{N_n}$ in the following way.
If $(x,y)\in S_{N_n}^{i,j}$ then $ u_n(x) := u(i/N_n, j/N_n).  $ Then
(by a Taylor series argument) $u_n \to u$ in $L^2(\T^2)$ as
$n\to\infty$.  Let $\nu_{N_n}^{i,j} := (i/N_n, j/N_n) \in G_{N_n}$,
then
\[
D_{N_n}^1 u_n(\nu_{N_n}^{i,j}) = \int_0^1 \frac{\partial u}{\partial
x}(\nu_{N_n}^{i,j}+s e_1/N_n) \,ds
\]
Jensen's inequality then gives
\[
(D_{N_n}^1 u_n(\nu_{N_n}^{i,j}))^2 = \! \Big (\int_0^1 \frac{\partial
u}{\partial x}(\nu_{N_n}^{i,j}+s e_1/N_n) \,ds \Big )^2 
\! \leq \! \int_0^1
 \Big (\frac{\partial u}{\partial x}(\nu_{N_n}^{i,j}+s
e_1/N_n) \Big )^2 \! ds.
\]
Because $u$ is smooth the Taylor series with remainder gives for $z\in
\T^2$
\[
\frac{\partial u}{\partial x}(\nu_{N_n}^{i,j}+s e_1/N_n) =
\frac{\partial u}{\partial x}(z+s e_1/N_n) + \nabla \frac{\partial
u}{\partial x}((1-c) x + c \nu_{N_n}^{i,j}) \cdot (\nu_{N_n}^{i,j}-z),
\]
for some $c\in [0,1]$.  Using the fact that $u$ and all its
derivatives are bounded we find for some constants $C_u>0$ and $\tilde
C_u$ depending only on $u$
\begin{align*}
&
N_n^{-2} \int_0^1  \Big (\frac{\partial u}{\partial
x}(\nu_{N_n}^{i,j}+s e_1/N_n) \Big )^2\,ds  = \int_{S_{N_n}^{i,j}}
\int_0^1  \Big (\frac{\partial u}{\partial x}(\nu_{N_n}^{i,j}+s
e_1/N_n) \Big )^2\,ds \,dx\\
& \leq \int_{S_{N_n}^{i,j}} \int_0^1  \Big [
\left(\frac{\partial u}{\partial x}(z+s e_1/N_n)\right)^2 + C_u
|\nu_{N_n}^{i,j}-z| + \tilde C_u |\nu_{N_n}^{i,j}-z|^2 \Big ]\,ds
\,dz\\
& \leq \int_0^1 \int_{S_{N_n}^{i,j}+s e_1/N_n}
 \Big (\frac{\partial u}{\partial x} \Big )^2(z) \,dz \,ds + C_u
\left(N_n\right)^{-3} + \tilde C_u \left(N_n\right)^{-4}.
\end{align*}
Analogous estimates hold for $D_{N_n}^2 u_n(\nu_{N_n}^{i,j})$.

Note that the sets $S_{N_n}^{i,j}+s e_k/N_n$, $k\in \{1,2\}$, are just
the squares $S_{N_n}^{i,j}$ shifted a distance $s/N_n$ over the
coordinate axes.  Because we are working on the torus we get for some
$\overline C_u>0$ depending only on $u$
\begin{align*}
&N_n^{-2} \sum_{i,j\in I_{N_n}}  \Big [\left(D_{N_n}^1 u_n\right)^2
(\nu_{N_n}^{i,j}) + \left(D_{N_n}^2 u_n\right)^2
(\nu_{N_n}^{i,j}) \Big ]\\
& \leq  \int_0^1 \sum_{i,j\in I_{N_n}}  \Big [\int_{S_{N_n}^{i,j}+s
e_1/N_n}  \Big (\frac{\partial u}{\partial x} \Big )^2(z) \,dz +
\int_{S_{N_n}^{i,j}+s e_2/N_n}  \Big (\frac{\partial u}{\partial
y} \Big )^2(z) \,dz  \Big ]\,ds  \\
&
\qquad
+ \overline C_u
\big(\left(N_n\right)^{-1}+\left(N_n\right)^{-2}\big) 
  = \int_{\T^2} |\nabla u|^2 + \overline C_u
\big(\left(N_n\right)^{-1}+\left(N_n\right)^{-2}\big).
\end{align*}
We deduce that
\[
\underset{n\to\infty}\limsup\, \int_{\T^2} \left(D_{N_n}^1
u_n\right)^2 + \left(D_{N_n}^2 u_n\right)^2 \leq \int_{\T^2} |\nabla
u|^2.
\]
Because $u\in C^\infty(\T^2)$ $u$ is bounded on $\T^2$ and hence
$\left\{|u_n|\right\}_{n=1}^\infty$ is uniformly bounded on $\T^2$.
Therefore there exists $\hat C>0$ such that for all $n\in\N\,\,$
$W(u_n)\leq \hat C$.  Because the constant $\hat C$ is integrable on
$\T^2$ we can use the dominated convergence theorem (or the reverse
Fatou's lemma) and the continuity of $W$ to deduce
\[
\underset{n\to\infty}\limsup\, \int_{\T^2} W(u_n) \leq \int_{\T^2}
\underset{n\to\infty}\limsup\, W(u_n) = \int_{\T^2} W(u)
\]
(or use $\int_{\T^2} \big(W(u_n)-W(u)\big)\,dx = \int_{\T^2}
\int_{u_n}^u W'(s)\,ds \,dx \leq C \int_{\T^2} |u_n-u|\,dx$).

Combining the two inequalities above leads to
\begin{align}
\notag
&
\underset{n\to\infty}\limsup\, k_{N_n, \e}(u_n)  \\
&
\leq \e\,
\underset{n\to\infty}\limsup\, \int_{\T^2} \left(D_{N_n}^1
u_n\right)^2 + \left(D_{N_n}^2 u_n\right)^2 + \e^{-1}\,
\underset{n\to\infty}\limsup\, \int_{\T^2} W(u_n) \notag\\
&\leq \e \int_{\T^2} |\nabla u|^2 + \e^{-1} \int_{\T^2} W(u) =
k_{\infty, \e}(u).\label{eq:resultforcontu}
\end{align}
In the terminology of $\Gamma$-upper limit of
Section~\ref{sec:explainGamma} we have proven that $k_{\infty,\e}''(u)
\leq k_{\infty,\e}(u)$ for $u\in C^\infty(\T^2)$.  Since
$C^\infty(\T^2)$ is dense in $W^{1,2}(\T^2)$ (using $W^{1,2}(\T^2)$
convergence) we use the lower semicontinuity of the $\Gamma$-upper
limit to conclude (UB') for all $u \in\nobreak W^{1,2}(\T^2)$ as
follows.  Let $\{u_n\}_{n=1}^\infty\subset C^\infty(\T^2)$ be a
sequence such that $u_n\to u$ in $W^{1,2}(\T^2)$ as $n\to\infty$, then
it also converges in $L^2(\T^2)$, hence
\[
k_\e''(u) \leq \underset{n\to\infty}\liminf\, k_\e''(u_n) \leq
\underset{n\to\infty}\liminf\, k_{\infty,\e}(u_n).
\]
Up to taking a subsequence $u_n\to u$ pointwise a.e., hence $W(u_n)\to
W(u)$ pointwise a.e. Thus by possibly redefining $u$ on a set of
measure zero for $n$ large enough we have that $W(u_n) \leq W(u) +
\tilde C$ for some $\tilde C>0$.  We can assume $k_{\infty,\e}(u) <
\infty$, hence $W(u)$ is integrable on $\T^2$.  Now we can again use
the dominated convergence theorem or the reverse Fatou lemma and the
continuity of $W$ to find
\[
\underset{n\to\infty}\limsup\, \int_{T^2} W(u_n) \leq \int_{\T^2}
\underset{n\to\infty}\limsup\, W(u_n) = \int_{\T^2} W(u).
\]
Since $u_n\to u$ in $W^{1,2}(\T^2)$ we have  
$$
\underset{n\to\infty}\limsup\, \int_{\T^2} |\nabla u_n|^2 =
\int_{\T^2} |\nabla u|^2, 
$$
 hence,
\[
k_\e''(u) \leq \underset{n\to\infty}\liminf\, k_{\infty,\e}(u_n) \leq
\underset{n\to\infty}\limsup\, k_{\infty,\e}(u_n) \leq
k_{\infty,\e}(u).
\eqno{\qed}
\]

\noindent
{\bf Proof of Theorem~\ref{thm:GammaconvergenceforkNe}.}
Combining Lemmas~\ref{lem:klowerbound} and~\ref{lem:kupperbound}
proves the $\Gamma$-convergence result.  Note in particular that we
have proven the lower bound for sequences converging in $L^1(\T^2)$
and the recovery sequence for the upper bound converges in
$L^2(\T^2)$, hence we can conclude $\Gamma$-convergence in both
topologies.
\qed

\smallskip

Using a technique from \cite[5.8.2 Theorem 3]{Evans02} we also get
compactness for $k_{N,\e}$.

\smallskip

\noindent
{\bf Proof of Theorem~\ref{thm:compactnessforkNe}.}
In what follows $k\in\{1,2\}$.  By (\ref{eq:kNereformulation}) we have
for all $n\in\N\,\,$ $u_n\in \mathcal{A}_{N_n}$ and
\[
\e \int_{\T^2} \left[(D_{N_n}^1 u_n)^2 + (D_{N_n}^2 u_n)^2\right] +
\e^{-1} \int_{\T^2} W(u_n) \leq C.
\]
By assumption $(W_2)$ on $W$ we find that
$\{\|u_n\|_{L^2(\T^2)}\}_{n=1}^\infty$ is uniformly boun\-ded, hence
there is a subsequence of $\{u_n\}_{n=1}^\infty$ (again labelled by
$n$) and a $u\in L^2(\T^2)$ such that $u_n \rightharpoonup u$ in
$L^2(\T^2)$ as $n\to\infty$.  Moreover we see that $\{\|D_{N_n}^k
u_n\|_{L^2(\T^2)}\}_{n=1}^\infty$ is uniformly bounded and hence there
is a further subsequence $\{u_{n'}\}_{n'=1}^\infty \subset
\{u_n\}_{n=1}^\infty$ and a $w\in L^2(\T^2;\R^2)$ such that
$D_{N_{n'}}^k u_n \rightharpoonup w_k$ in $L^2(\T^2)$ as
$n'\to\infty$. 
Let $\phi\in C_c^\infty(\T^2)$, then
\begin{align*}
\int_{\T^2} u_{n'} D_{N_{n'}}^k\phi &= N_{n'} \int_{\T^2} u_{n'}(x)
[\phi(x+e_k/N_{n'})-\phi(x)]\,dx\\
 &= N_{n'} \int_{\T^2} [u_{n'}(x-e_k/N_{n'})-u_{n'}(x)] \phi(x) dx =
 -\int_{\T^2} \phi  D_{N_{n'}}^k u_{n'}.
\end{align*}
By (\ref{eq:Taylorfinitediff}) we have $ \int_{\T^2} \big(D_{N_{n'}}^k
\phi - \nabla \phi\cdot e_k\big)^2 = K N_{n'}^{-2}, $ for some $K>0$,
hence $D_{N_{n'}}^k \phi \to \nabla \phi \cdot e_k$ in $L^2(\T^2)$ as
$n'\to\infty$.  Combining this strong convergence for the difference
quotient of $\phi$ with the weak convergence for $u_{n'}$ and its
difference quotient we deduce
\[
\int_{\T^2} u \nabla \phi \cdot e_k = \underset{n'\to\infty}\lim\,
\int_{\T^2} u_{n'} D_{N_{n'}}^k\phi = -\underset{n'\to\infty}\lim\,
\int_{\T^2} \phi\, D_{N_{n'}}^k u_{n'} = -\int_{\T^2} \phi\, w_k.
\]
Hence $|\nabla u| = |w| \in L^2(\T^2)$.  We conclude that $u\in
W^{1,2}(\T^2)$.

Finally, the strong convergence $u_{n'} \to u$ in $L^2(\T^2)$ follows
from the bound on $D_{N_{n'}}^k u_n$ by a discrete version of the
Rellich-Kondrachov compactness theorem (see Lemma~\ref{lem:discreteRK}
in Appendix~\ref{sec:deferredproofs} for details).
\qed

\subsection{Simultaneous scaling $\Gamma$-limit for
$k_N^\alpha$}\label{sec:simulscalekNalpha}

In Section~\ref{sec:proofk} we studied the $\Gamma$-limits of
$k_{N,\e}$ by first taking $N\to\infty$ and then $\e\to 0$.  We will
now show we can take both limits at once if we scale $\e$ correctly in
terms of $N$.  This is particularly relevant for numerical
applications.  We set $\e=N^{-\alpha}$ for some $\alpha\in
(0,\frac2{q+3})$ where $q$ is the degree of polynomial growth of $W'$
in condition $(W_4)$ and take the limit $N\to\infty$ of $k_N^\alpha$
in (\ref{eq:kN}).

Note that in contrast to the case for $h_N^\alpha$ the order of the
limits $\e \to 0$ and $N\to\infty$ are reversed and thus we have an
upper bound on $\alpha$ instead of a lower bound.

 We prove a compactness and $\Gamma$-convergence result.

\begin{theorem}[Compactness]\label{thm:kNcompactness}
Assume $W$ satisfies $(W_3)$ and $(W_4)$ for given $p\geq 2$ and $q>0$
and $\alpha\in (0,\frac2{q+3})$.  Let $\{N_n\}_{n=1}^\infty\subset\N$
satisfy $N_n\to \infty$ as $n\to\infty$ and let $\{u_n\}_{n=1}^\infty
\subset L^1(\T^2)$ be a sequence for which there is a constant $C>0$
such that for all $n\in\N$
$
k_{N_n}^\alpha(u_n) \leq C.
$
Then there exists a subsequence $\{u_{n'}\}_{n'=1}^\infty \subset
\{u_n\}_{n=1}^\infty$ and a $u\in BV(\T^2,\{0,1\})$ such that
$u_{n'}\to u$ in $L^2(\T^2)$ as $n'\to\infty$.
\end{theorem}

\begin{theorem}[$\Gamma$-convergence]\label{thm:GammaconvergenceforkNalpha}
Let $W$ satisfy $(W_3)$ and $(W_4)$ for given $p,q>0$ and assume
$\alpha\in (0,\frac2{q+3})$.  Then $\displaystyle k_N^\alpha
\overset{\Gamma}{\to} k_{\infty, 0}$ as $N\to \infty$ in either the
$L^1(\T^2)$ or $L^2(\T^2)$ topology.
\end{theorem}

Using the difference quotient notation from
(\ref{eq:differencequotient}) we can mimic (\ref{eq:kNereformulation})
and write
\[
k_{N_n}^\alpha(u) = \left\{\begin{array}{ll} N^{-\alpha} \int_{\T^2}
\left[(D_N^1u)^2 + (D_N^2u)^2\right] + N^{\alpha} \int_{\T^2} W(u) &
\text{if } u \in \mathcal{A}_N,\\
+\infty & \text{otherwise}.  \end{array}\right.
\]

The proofs of this section make repeated use of the Modica-Mortola
results \cite{ModicaMortola77,Modica87a,Modica87b,Sternberg88} which
show ($L^1$ and $L^2$) compactness and convergence\footnote{As
remarked in an earlier footnote, results are usually stated in the
$L^1$ topology, but for example \cite{RenWei00,Braides02,ChoksiRen05}
note that the $\Gamma$-convergence can be stated in the $L^2$ topology
as well.  Compactness in $L^2$ follows from compactness in $L^1$
combined with the binary nature of the limit function.  Finally note
that the original results on bounded domains are easily adapted for
the torus.  Compactness is not hindered by the periodicity because no
regularity beyond $BV(\T^2; \{0,1\})$ is needed for the limit.  The
lower bound generalizes immediately by restriction to sequences of
periodic functions.  The important properties of the recovery sequence
for the upper bound are local properties near the boundary of $\supp
u$ and so are also satisfied on a periodic domain.} for
$F_{N^{-\alpha}}^{GL} \overset{\Gamma}{\to} F_0^{GL}$ as $N\to
\infty$.  Note that condition $(W_3)$ with $p\geq 2$ on the double
well potential $W$ is needed for the compactness result to hold (see
\textit{e.g.} \cite[Proposition 3]{Sternberg88}).

\smallskip

\noindent
{\bf Proof of Theorem~\ref{thm:kNcompactness}.}
Let $\{u_n\}_{n=1}^\infty \subset L^1(\T^2)$ be a sequence such that
$k_{N_n}^\alpha(u_n) \leq C$.  Below we will prove the claim that
there is a sequence $\{v_n\}_{n=1}^\infty \subset W^{1,2}(\T^2)$ such
that $\|v_n-u_n\|_{L^2(\T^2)}\to 0$ as $n \to \infty$ and
\begin{equation}\label{eq:comparekunkvn}
k_{N_n}^\alpha(u_n) \geq F_{N_n^{-\alpha}}^{GL}(v_n) + R_n, \quad
\underset{n\to\infty}\lim\, R_n = 0.
\end{equation}
Given the veracity of this claim, it follows by the Modica-Mortola
compactness result for $F_{N_n^{-\alpha}}^{GL}$
\cite{ModicaMortola77,Modica87a,Modica87b,Sternberg88} that there is a
subsequence of $\{v_n\}_{n=1}^\infty$ (again labeled by $n$) such that
$v_n \to u$ in $L^2(\T^2)$ for a $u\in BV(\T^2,\{0,1\})$ (here we need
condition $(W_3)$ on $W$ with $p\geq 2$).  Using the triangle
inequality and $\|v_n-u_n\|_{L^2(\T^2)}\to 0$ as $n \to \infty$ we
then conclude that there is a subsequence of $\{u_n\}_{n=1}^\infty$
converging in $L^2(\T^2)$ to $u$.

To prove the claim, first we show that $\|u_n\|_{L^\infty(\T^2)} =
\mathcal{O}(N_n^{\alpha/2})$.  Assume not and let
$\gamma>\frac\alpha2$, then there is a subsequence (labeled again by
$n$) such that for each $n$ there is a square $S_{N_n}^{i,j}$ on which
$|(u_n)_{i,j}|\geq N_n^\gamma$.  For definiteness assume $(u_n)_{i,j}
= N_n^\gamma$.  By the uniform bound on $k_{N_n}^\alpha(u_n)$ we have
\begin{equation}\label{eq:determineungrowth}
N_n^{-\alpha} (N_n^\gamma-(u_n)_{i+1,j})^2 \leq C,
\end{equation}
hence $u_{i+1,j} = \Theta(N_n^\gamma)$.  By induction over all the
squares $S_{N_n}^{i,j}$ we find that $u_n = \Theta(N_n^\gamma)$.
Therefore $\|u_n\|_{L^p(\T^2)}^p = \Theta(N_n^{p\gamma})$, but by the
coercivity condition $(W_3)$ on $W$ (for any $p>0$) the uniform bound
on $k_{N_n}^\alpha(u_n)$ demands
\[
c_1 \|u_n\|_{L^p(\T^2)}^p \leq \int_{\T^2} W(u_n) \leq C
N_n^{-\alpha},
\]
which is a contradiction.

Now for any $u_n$ let $v_n$ be its bilinear interpolation: For $(x,y)
\in S_{N_n}^{i,j}$ define
\begin{align*}
v_n(x, & y)  := N_n^2  \Big [ (u_n)_{i,j} \Big (\frac{i+1}{N_n}-x \Big )
 \Big (\frac{j+1}{N_n} - y \Big )  \\
&
 + (u_n)_{i+1,j}
 \Big (x-\frac{i}{N_n} \Big )  \Big (\frac{j+1}{N_n} - y \Big )\\
&  + (u_n)_{i,j+1}  \Big (\frac{i+1}{N_n}-x \Big )
 \Big (y-\frac{j}{N_n} \Big ) + (u_n)_{i+1,j+1}
 \Big (x-\frac{i}{N_n} \Big )  \Big (y-\frac{j}{N_n} \Big )  \Big ],
\end{align*}
where for notational convenience we have identified $u_n\in
\mathcal{A}_{N_n}$ with its counterpart in $\mathcal{V}_{N_n}$.  Thus
defined $v_n$ is continuous and $\|v_n\|_{L^\infty(\T^2)} =
\|u_n\|_{L^\infty(\T^2)}$.  A straightforward computation shows
\begin{align*}
&  \int_{\T^2} (v_n-u_n)^2 
 = \frac1{18} N_n^{-2} \sum_{i,j=1}^{N_n}  \Big [ \frac72
\big((u_n)_{i,j}-(u_n)_{i+1,j}\big)^2  \\
&
\quad
+ \frac72
\big((u_n)_{i,j}-(u_n)_{i,j+1)}\big)^2 + 4
\big((u_n)_{i,j}-(u_n)_{i+1,j+1}\big)^2 
 \\
&
\quad
- \frac12 \big((u_n)_{i+1,j}-(u_n)_{i,j+1}\big)^2 -
\big((u_n)_{i+1,j}-(u_n)_{i+1,j+1}\big)^2 
\\
&
\qquad
-
\big((u_n)_{i,j+1}-(u_n)_{i+1,j+1}\big)^2 \Big ] .
\end{align*}
First we note that there is a $C>0$ such that
\begin{align*}
&
\big((u_n)_{i,j}-(u_n)_{i+1,j+1}\big)^2 
\\
&\leq C
 \Big [\big((u_n)_{i,j}-(u_n)_{i+1,j}\big)^2+\big((u_n)_{i+1,j}-(u_n)_{i+1,j+1}\big)^2 \Big ],\\
&
\big((u_n)_{i+1,j}-(u_n)_{i,j+1}\big)^2 
\\
 &\leq C
 \Big [\big((u_n)_{i+1,j}-(u_n)_{i,j}\big)^2+\big((u_n)_{i,j}-(u_n)_{i,j+1}\big)^2 \Big ].
\end{align*}
Next we use periodicity to deduce
\[
\sum_{i,j=1}^{N_n} \big((u_n)_{i+1,j}-(u_n)_{i+1,j+1}\big)^2 =
\sum_{i,j=1}^{N_n} \big((u_n)_{i,j}-(u_n)_{i,j+1}\big)^2,
\]
and analogously for similar terms.  Using the uniform bound on
$k_{N_n}^\alpha(u_n)$ we find
\begin{align*}
\int_{\T^2} (v_n-u_n)^2 
&
\leq C N_n^{-2} \sum_{i,j=1}^{N_n}
\Big [\big((u_n)_{i+1,j}-(u_n)_{i,j}\big)^2+\big((u_n)_{i,j+1}-(u_n)_{i,j}\big)^2\Big ]
\\
& \leq C N_n^{-2+\alpha}.
\end{align*}
Thus, $\|v_n-u_n\|_{L^2(\T^2)}\to 0$ as $n \to \infty$.
Another computation gives
\begin{align*}
\int_{\T^2}  \Big (\frac{\partial v_n}{\partial x} \Big )^2 &= \frac13
\sum_{i,j=1}^{N_n}
 \Big [\big((u_n)_{i+1,j}-(u_n)_{i,j}\big)^2+\big((u_n)_{i+1,j+1}-(u_n)_{i,j+1}\big)^2\\
&\hspace{1.8cm}
+\big((u_n)_{i+1,j}-(u_n)_{i,j}\big)\big((u_n)_{i+1,j+1}-(u_n)_{i,j+1}\big) \Big ]
.
\end{align*}
Using periodicity as above in combination with the inequality
\begin{align*}
&
\big((u_n)_{i+1,j}-(u_n)_{i,j}\big)\big((u_n)_{i+1,j+1}-(u_n)_{i,j+1}\big)
\\
&
\leq \frac12
 \Big (\big((u_n)_{i+1,j}-(u_n)_{i,j}\big)^2+\big((u_n)_{i+1,j+1}-(u_n)_{i,j+1}\big)^2 \Big ) 
\end{align*}
we deduce
\[
\int_{\T^2}  \Big (\frac{\partial v_n}{\partial x} \Big )^2 \leq
\sum_{i,j=1}^{N_n} \big((u_n)_{i+1,j}-(u_n)_{i,j}\big)^2
\]
and the analogous result for $\int_{\T^2} \big (\frac{\partial
v_n}{\partial y} \big )^2$.

Finally we note that 
$$
N_n^\alpha \int_{\T^2} W(u_n) = N_n^\alpha
\int_{\T^2} W(v_n) + R_n ,
$$
 where
\begin{align}
\label{eq:resttermRn}
R_n  & = N_n^\alpha \int_{\T^2} \big (W(u_n)-W(v_n)\big ) \leq N_n^\alpha
M_n^{W'} \|u_n-v_n\|_{L^1(\T^2)} \\
&
\notag
\leq N_n^\alpha M_n^{W'}
\|u_n-v_n\|_{L^2(\T^2)} \leq M_n^{W'} N_n^{\frac32\alpha-1}.
\end{align}
Here 
$$
M_n^{W'} := \underset{x\in\T^2}\max\, \underset{s\in
[v_n(x),u_n(x)]}\max\, |W'(s)| . 
$$
 By construction, since
$\|u_n\|_{L^\infty(\T^2)} = \mathcal{O}(N_n^{\alpha/2})$, we have
$$
\|v_n\|_{L^\infty(\T^2)} = \mathcal{O}(N_n^{\alpha/2}) . 
$$
 Hence, the
maximum over $s$ is achieved for some $s=\mathcal{O}(N_n^{\alpha/2})$.
By the regularity of $W'$ and its polynomial growth condition $(W_4)$
we then have $M_n^{W'} = \mathcal{O}(N_n^{\alpha q/2})$, hence $R_n
\to 0$ as $n\to\infty$ by the choice of $\alpha$.

The inequalities above prove the claim and hence finish the proof.
\qed

\begin{remark}\label{rem:kNcompactnessL2}
Clearly, the $L^2(\T^2)$ convergence in Theorem~\ref{thm:kNcompactness}
can be replaced by $L^1(\T^2)$ convergence if desired.
\end{remark}

\begin{lemma}[Lower bound]\label{lem:klowerboundallNs}
Assume $W$ satisfies $(W_3)$ and $(W_4)$ for given $p>0$ and $q>0$ and
$\alpha\in (0,\frac2{q+3})$.  Let $u\in L^1(\T^2)$ and let
$\{u_n\}_{n=1}^\infty \subset L^1(\T^2)$ and
$\{N_n\}_{n=1}^\infty\subset\N$ be such that $u_n \to u$ in
$L^1(\T^2)$ and $N_n \to \infty$ as $n\to \infty$, then
\[
k_{\infty,0}(u) \leq \underset{n\to\infty}\liminf\,
k_{N_n}^\alpha(u_n).
\]
\end{lemma}

\noindent
{\bf Proof.}
Without loss of generality we can assume that $k_{N_n}^\alpha(u_n)$ is
uniformly bounded.  In the proof of Theorem~\ref{thm:kNcompactness} we
established that then estimate (\ref{eq:comparekunkvn}) follows, where
the $v_n$ are the bilinear interpolations of $u_n$ that converge to
$u$ in $L^1(\T^2)$.  Using the $\Gamma$-convergence result of Modica
and Mortola \cite{ModicaMortola77,Modica87a,Modica87b,Sternberg88},
specifically their lower bound, we find
\[
\underset{n\to\infty}\liminf\, k_{N_n}^\alpha(u_n) \geq
\underset{n\to\infty}\liminf\, \left(F_{N_n^{-\alpha}}^{GL}(v_n) +
R_n\right) \geq k_{\infty,0}(u).
\eqno{\qed}
\]

\begin{remark}\label{rem:alphadim}
As noted earlier in Remark~\ref{rem:kgendim} our results (and proofs)
generalize to $\T^d$ if the terms in $k_N^\alpha$ are rescaled
properly depending on the dimension $d$.  In this case we carefully
need to reexamine the admissible range for $\alpha$ in
Theorem~\ref{thm:kNcompactness} and Lemma~\ref{lem:klowerboundallNs}
(and hence by extension Theorem~\ref{thm:GammaconvergenceforkNalpha}).
In particular (\ref{eq:determineungrowth}) becomes
\[
N_n^{2-d-\alpha} (N_n^\gamma-(u_n)_{i+1,j})^2 \leq C
\]
and hence we need to choose $\gamma > \frac{\alpha+d-2}2$ and deduce
that $\|u_n\|_{L^\infty(\T^2)} =
\mathcal{O}(N_n^{\frac{\alpha+d-2}2})$.  Generalizing
(\ref{eq:resttermRn}) and the discussion that follows then leads to
the conclusion that the admissible range of $\alpha$ is $\alpha\in
 \big (0, \frac{2q - 2(d-2)}{q (q+3)} \big )$.  In particular $q$ in
condition $(W_4)$ should be chosen larger than $d-2$.
\end{remark}

\begin{lemma}[Upper bound]\label{lem:kupperboundallNs}
Let $\alpha \in (0,1)$, $u\in L^1(\T^2)$, and
$\{N_n\}_{n=1}^\infty\subset\N$ be such that $N_n \to \infty$ as $n\to
\infty$ and let $k''$ be the $\Gamma$-upper limit of $k_{N_n}^\alpha$
as $n\to\infty$ with respect to the $L^2(\T^2)$ topology, then
$
k''(u) \leq k_{\infty,0}(u).
$
\end{lemma}

\noindent
{\bf Proof.}
The case where $u\in L^1(\T^2)\setminus BV(\T^2;\{0,1\})$ is trivial,
so assume $u\in BV(\T^2;\{0,1\})$.  First assume that $\supp u$ has
smooth boundary $\partial \supp u$.

By the classical Modica-Mortola results used before $k_{\infty,\e}$
$\Gamma$-converges in both the $L^1(\T^2)$ and $L^2(\T^2)$ topology to
$k_{\infty,0}$.  Let $\{v_n\}_{n=1}^\infty$ be the recovery sequence
for this convergence with $\e=N_n^{-\alpha}$, see \textit{e.g.}
\cite[Proposition 2]{Modica87a}, \cite[\S 7.2.1]{Braides06}, then each
$v_n \in W^{1,2}(\T^2)$ is a Lipschitz continuous function, $v_n\to u$
in $L^2(\T^2)$ as $n\to\infty$, and $\underset{n\to\infty}\limsup\,
k_{\infty,N_n^{-\alpha}}(v_n) \leq k_{\infty,0}(u)$.  We denote the
Lipschitz constant of $v_n$ by $K_n$ and note that $K_n =
\mathcal{O}(N_n^\alpha)$ as $n\to \infty$.

We now follow a similar line of reasoning as in the proof of
Lemma~\ref{lem:kupperbound}, but need to be more careful to deal with
the lesser regularity of $v_n$ (only Lipschitz continuous instead of
$C^\infty$).  For each $i,j \in I_N$ and each $n\in \N$, define
\[
\nu_{N_n}^{i,j} := \underset{z\in
\overline{S_{N_n}^{i,j}}}{\text{argmin}}\, \Big[ (D_{N_n}^1 v_n(z))^2
+ (D_{N_n}^2 v_n(z))^2 \Big],
\]
where we used the difference quotient notation from
(\ref{eq:differencequotient}).  Note that since $v_{N_n}$ is
continuous and $\overline{S_{N_n}^{i,j}}$ is compact, the minimum is
attained.  Note that it may be that $\nu_{N_n}^{i,j}\in
\overline{S_{N_n}^{i,j}}\setminus S_{N_n}^{i,j}$.  Now define a
sequence $\{u_n\}_{n=1}^\infty$ such that for each $n\in\N\,\,$
$u_n\in \mathcal{A}_{N_n}$ in the following way.  If $x\in
S_{N_n}^{i,j}$, then $u_n(x) := v_n(\nu_{N_n}^{i,j})$.  Note that by
construction
\[
(D_{N_n}^1 u_n(x))^2 + (D_{N_n}^2 u_n(x))^2 \leq (D_{N_n}^1 v_n(x))^2
+ (D_{N_n}^2 v_n(x))^2
\]
First we check that $u_n\to u$ in $L^2(\T^2)$.  We estimate
$\|u_n-u\|_{L^2(\T^2)} \leq \|u_n-v_n\|_{L^2(\T^2)} +
\|v_n-u\|_{L^2(\T^2)}$.  We know that $v_n \to u$ in $L^2(\T^2)$, for
the first term on the right we use the Lipschitz continuity of $v_n$
to compute
\begin{align*}
\|u_n-v_n\|_{L^2(\T^2)}^2 &= \int_{\T^2} |u_n(x)-v_n(x)|^2 dx = \!\! \!
\sum_{i,j\in I_{N_n}} \int_{S_{N_n}^{i,j}}
|v_n(\nu_{N_n}^{i,j})-v_n(x)|^2\,dx\\
&\leq K_n^2 \sum_{i,j\in I_{N_n}} \int_{S_{N_n}^{i,j}}
|\nu_{N_n}^{i,j} - x|^2\, dx \leq \sqrt2 K_n^2 N_n^2 N_n^{-2}
N_n^{-2}.
\end{align*}
Here we have used that there are $N_n^2$ nodes in the grid, the area
of each square $S_{N_n}^{i,j}$ is $N_n^{-2}$, and for $x\in
S_{N_n}^{i,j}$ we have $|\nu_{N_n}^{i,j}-x| \leq \sqrt2 N_n^{-1}$.
Since $K_n=\mathcal{O}(N_n^\alpha)$ as $n\to\infty$ and $\alpha<1$ we
have $\|u_n-v_n\|_{L^2(\T^2)}^2 \to 0$ as $n\to\infty$ and hence $u_n
\to u$ in $L^2(\T^2)$.

Since $v_{N_n}: \T^2 \to \R$ is Lipschitz continuous, if we fix either
$y\in\T$ or $x\in \T$ so are $v_{N_n}(\cdot, y): \T\to\R$ and
$v_{N_n}(x,\cdot): \T\to\R$.  Therefore by Rademacher's theorem the
partial derivatives of $v_{N_n}$ exist a.e. on horizontal and vertical
lines, hence we have for a.e. $x\in\T^2$,
\[
D_{N_n}^1 v_n(x) = \int_0^1 \frac{\partial v_n}{\partial x}(x+s
e_1/N_n) \,ds,  \quad D_{N_n}^2 v_n(x) = \int_0^1
\frac{\partial v_n}{\partial y}(x+s e_2/N_n) \,ds.
\]
By Jensen's inequality
\[
(D_{N_n}^1 v_n(x))^2 = \Big (\int_0^1 \frac{\partial v_n}{\partial
x}(x+s e_1/N_n) \,ds \Big )^2 \leq \int_0^1 \Big (\frac{\partial
v_n}{\partial x}(x+s e_1/N_n) \Big )^2\,ds,
\]
and similarly for $D_{N_n}^2 v_n(x)$.
For the finite difference terms in $k_{N_n}^\alpha(u_n)$ we now find,
by construction of $u_n$,
\begin{align*}
&\hspace{0.8cm} \int_{\T^2} \left[(D_{N_n}^1u_n(x))^2 +
(D_{N_n}^2u_n(x))^2\right]\,dx\\
&= \sum_{i,j\in I_{N_n}} \int_{S_{N_n}^{i,j}}
\left[(D_{N_n}^1u_n(x))^2 + (D_{N_n}^2u_n(x))^2\right]\,dx \\
&
\leq
\sum_{i,j\in I_{N_n}} \int_{S_{N_n}^{i,j}} \left[(D_{N_n}^1v_n(x))^2 +
(D_{N_n}^2v_n(x))^2\right]\,dx\\
&\leq \sum_{i,j\in I_{N_n}} \int_{S_{N_n}^{i,j}} \int_0^1
 \Big [ \Big (\frac{\partial v_n}{\partial x}(x+s e_1/N_n) \Big )^2 +
 \Big (\frac{\partial v_n}{\partial y}(x+s e_2/N_n) \Big )^2 \Big ]
\,ds \,dx\\
&= \int_0^1 \sum_{i,j\in I_{N_n}}  \Big [\int_{S_{N_n}^{i,j}+s e_1/N_n}
 \Big (\frac{\partial v_n}{\partial x}(x) \Big )^2\,dx +
\int_{S_{N_n}^{i,j}+s e_2/N_n}  \Big (\frac{\partial v_n}{\partial
y}(x) \Big )^2\,dx  \Big ] \,ds  \\
&= \int_{\T^2} |\nabla v_n|^2,
\end{align*}
where we have used Fubini's theorem and the fact that we are working
on a torus, so the union of all sets of the form $S_{N_n}^{i,j}+s
e_k/N_n$ (for either $k=1$ or $k=2$) is the same as the union of all
$S_{N_n}^{i,j}$, \textit{i.e.,} $\T^2$.

To deal with the double well potential term in $k_{N_n}^\alpha$ we
first note that for each $x\in\T^2$ and each $n\in\N\,\,$ $v_n(x)\in
[0,1]$, hence $W'$ is bounded on intervals of the form $[v_n(x),
v_n(\nu_{N_n}^{i,j})]$ by some $C>0$.  Therefore, for $x\in
S_{N_n}^{i,j}$,
\begin{align*}
W(u_n(x))-W(v_n(x)) &= W(v_n(\nu_{N_n}^{i,j})) - W(v_n(x)) =
\int_{v_n(x))}^{v_n(\nu_{N_n}^{i,j})} W'(s)\,ds\\
&\leq C |v_n(\nu_{N_n}^{i,j})-v_n(x)| \leq C K_n |\nu_{N_n}^{i,j}-x|.
\end{align*}
We thus find
\begin{align*}
\int_{\T^2} W(u_n)(x)\,dx &= \int_{\T^2}W(v_n(x))\,dx + C K_n
\sum_{i,j\in I_{N_n}} \int_{S_{N_n}^{i,j}} |\nu_{N_n}^{i,j}-x|\,dx\\
&\leq \int_{\T^2}W(v_n(x))\,dx + C K_n N_n^2 N_n^{-2} N_n^{-1}.
\end{align*}
Hence $\int_{\T^2} W(u_n)(x)\,dx \leq \int_{\T^2}W(v_n(x))\,dx + R_n$,
where $R_n = \mathcal{O}(N_n^{\alpha-1})$.  Since $\alpha<1$, $R_n\to
0$ as $n\to\infty$.

Combining both terms in $k_{N_n}^\alpha$ we find $k_{N_n}^\alpha(u_n)
\leq k_{\infty,N_n^{-\alpha}}(v_n) + R_n$.  We already know that
$\underset{n\to\infty}\limsup\, k_{\infty,N_n^{-\alpha}}(v_n) \leq
k_{\infty,0}(u)$, so we have proved that for $u\in BV(\T^2;\{0,1\})$
with smooth $\partial \supp u$ we have
$
\underset{n\to\infty}\limsup\, k_{N_n}^\alpha(u_n) \leq
k_{\infty,0}(u)
$
or in terms of the $\Gamma$-upper limit: $k''(u) \leq
k_{\infty,0}(u)$.

Now let $u\in BV(\T^2;\{0,1\})$, not necessarily with the smoothness
condition on the boundary.  Then by \cite[Theorem
3.42]{AmbrosioFuscoPallara00} there is a sequence $\{\hat
u_n\}_{n=1}^\infty \subset BV(\T^2;\{0,1\})$ such that each $\supp
u_n$ has smooth boundary and $\underset{n\to\infty}\lim\,
k_{\infty,0}(\hat u_n) = k_{\infty,0}(u)$.  Hence by lower
semicontinuity of the $\Gamma$-upper limit $k''$ we conclude
\[
k''(u) \leq \underset{n\to\infty}\liminf\, k''(\hat u_n) \leq
\underset{n\to\infty}\liminf\, k_{\infty,0}(\hat u_n) =
k_{\infty,0}(u).
\eqno{\qed}
\]

\begin{proof}[\bf Proof of Theorem~\ref{thm:GammaconvergenceforkNalpha}]
Since we have used $L^1(\T^2)$ convergence in (LB) in
Lemma~\ref{lem:klowerboundallNs} and $L^2(\T^2)$ convergence in
Lemma~\ref{lem:kupperboundallNs} for (UB') we can now conclude
$\Gamma$-convergence in either of these two topologies.
\end{proof}

\subsection{Discussion of the range of
$\alpha$}\label{sec:discussionalpha}

The range of admissible $\alpha$ in the results in the previous
section is not only of theoretical interest, but is also important for
computations.  In simulations choosing $\e$ of the right order is a
hard problem.  If $\e$ is too small in gradient flow simulations this
leads to the phenomenon of `pinning', where the initial condition gets
pinned down into the wells of $W$ without changing its geometry.  On
the other hand, an $\e$ which is too large leads to immediate
diffusion of the initial condition and loss of some relevant features.
Our results do not directly address the gradient flow, but are in the
same spirit.

In the proof of compactness and the lower bound above we have assumed
$\alpha\in (0,\frac2{q+3})$ where $q$ is the degree of polynomial
growth of $W'$.  There are some reasons to believe this restriction
could possibly be relaxed to $\alpha\in (0, \frac2{q+2})$, but we have
not found a proof for this statement.

First note that the dependence on $q$ in the range of $\alpha$ comes
from the discrete nature of the problem.  Fundamentally it can be
traced back to the lack of a chain rule for discrete differentiation.
Trying to copy the classical Modica-Mortola result, we can define
$w_n:= \int_0^{u_n} \sqrt{W(s)}\,ds$ and estimate
\begin{align*}
&
\int_{\T^2} |\nabla w_n|  = N_n^{-1} \sum_{i,j=1}^{N_n}
|(w_n)_{i+1,j}-(w_n)_{i,j}| + |(w_n)_{i,j+1}-(w_n)_{i,j}|
\\
&\leq 2 N_n^{-1} \sum_{i,j=1}^{N_n}
\bigg ( {\left[((w_n)_{i+1,j}-(w_n)_{i,j})^2 +
((w_n)_{i,j+1}-(w_n)_{i,j})^2\right] } \bigg)^{\frac 12}
\\
&
= 2 N_n^{-1} \sum_{i,j=1}^{N_n}
 \bigg (  { \Big (\int_{(u_n)_{i,j}}^{(u_n)_{i+1,j}}
\sqrt{W(s)}\,ds \Big )^2+ \Big (\int_{(u_n)_{i,j}}^{(u_n)_{i,j+1}}
\sqrt{W(s)}\,ds \Big )^2} \bigg)^{\frac 12} 
\\
&
= 2 N_n^{-1} \sum_{i,j=1}^{N_n}
 \bigg(
  \Big [((u_n)_{i+1,j}-(u_n)_{i,j})^2 W((u_n^\ast)_{i+1,j}) 
\\
&
\qquad\qquad \qquad\quad
+
((u_n)_{i,j+1}-(u_n)_{i,j})^2 W((u_n^\ast)_{i,j+1}) \Big ]  
 \bigg)^{\frac 12},
\end{align*}
where $(u_n^\ast)_{i+1,j} \in [(u_n)_{i,j}, (u_n)_{i+1,j}]$ and
$(u_n^\ast)_{i,j+1} \in [(u_n)_{i,j}, (u_n)_{i,j+1}]$ come from the
mean value theorem.  If we had control over the behavior of $W$ in
between grid points, we could use Cauchy's inequality to bound the
expression above by $k_{N_n}^\alpha(u_n)$.  Without condition $(W_4)$
this control is lacking.  We could do without this condition if we
would somehow have an \textit{a priori} $L^\infty$ bound for the
sequence $\{u_n\}_{n=1}^\infty$.  This reflects a similar situation in
the continuum case \cite[Remark 1.35]{Sternberg88} where condition
$(W_3)$ with $p\geq 2$ can be dropped from the assumptions needed for
compactness if an \textit{a priori} $L^\infty$ bound is available.  By
construction a uniform bound on $\|u_n\|_{L^\infty(\T^2)}$ gives a
similar bound on $\|v_n\|_{L^\infty(\T^2)}$, hence under such an
\textit{a priori} bound we could drop conditions $(W_3)$ and $(W_4)$
from our assumptions and eleminate $q$ (\textit{i.e.,} $q=0$) from the
restriction on $\alpha$.  The difference with the continuum case is
that in our case (in the absence of an $L^\infty$ bound) we need
control over $W$ and its derivative $W'$.  The scale of the
discretization, $N^{-1}$, should be fine enough to resolve the
variations in $W'$.

Next note that in the proof of the upper bound we only use $\alpha\in
(0,1)$.  This restriction has a natural interpretation: If we
interpret $N^{-1}$ as the discretization spacing, then
$\e=N^{-\alpha}$ for $0<\alpha<1$ tells us that the discretization
should be fine enough to `resolve the diffuse interface' which, in the
continuum case, has width of order $\e$.  The following example shows
that for $\alpha>1$ the lower bound fails.  Let $u\in BV(\T^2,
\{0,1\})$ be equal to zero on half the torus (say on $[0,1/2] \times
[0,1)$) and equal to one on the other half and let
$\{u_n\}_{n=1}^\infty$ be a sequence converging to $u$ in $L^1(\T^2)$
obtained by simply discretizing $u$ on the grid $G_{N_n}$ for a
sequence $\{N_n\}_{n=1}^\infty$, $N_n\to \infty$ as $n\to\infty$.
Then $W(u_n)\equiv 0$ for all $n$.  The finite difference term in
$k_{N_n}^\alpha(u_n)$ only has nonzero contributions along the
boundary between the parts of the torus where $u=0$ and $u=1$.  Thus
there are $2N$ jumps of order 1 and hence $k_{N_n}^\alpha(u_n) = 2
N_n^{1-\alpha}$.  If $\alpha>1$ this converges to zero, but the limit
functional $k_{\infty,0}(u)>0$, which contradicts (LB).

Finally, note that in (\ref{eq:resttermRn}) our estimate is not sharp
since we use H\"older's inequality to go from
$\|u_n-v_n\|_{L^1(\T^2)}$ to $\|u_n-v_n\|_{L^2(\T^2)}$.  If instead a
bound $\|u_n-v_n\|_{L^1(\T^2)} = \mathcal{O}(N_n^{-1})$ could be
proved, possibly using the uniform bound on $k_{N_n}^\alpha(u_n)$,
then in (\ref{eq:resttermRn}) the condition on $\alpha$ relaxes to
$\alpha\in (0, \frac{2}{q+2})$, which in the absence of $q$ would
reduce to $\alpha\in (0,1)$.  We therefore conjecture that $\alpha\in
(0, \frac{2}{q+2})$ is in fact the natural restriction for $\alpha$
(on $\T^2$, see Remark~\ref{rem:alphadim} for a discussion about the
range of $\alpha$ in general dimensions), or, if an \textit{a priori}
$L^\infty$ bound is available, $\alpha\in (0,1)$.  However, it might
be the case that a bilinear interpolation is not the right
interpolation to attain this bound.

\subsection{Constraints}

In this section we show that addition of a fidelity term or imposing a
mass constraint are compatible with the three $\Gamma$-limits we
established, \textit{i.e.,} $N\to\infty$ for $k_{N,\e}$, $\e\to 0$ for
$k_{\infty,\e}$ and $N\to\infty$ for $k_N^\alpha$.

The Modica-Mortola limit $k_{\infty,\e} \overset{\Gamma}{\to}
k_{\infty,0}$ as $N\to \infty$ in the $L^p(\T^2)$ topology,
$p\in\{1,2\}$, is known to be compatible with a mass constraint,
\textit{e.g.} \cite[Proposition 2]{Modica87a}, \cite[Theorem
1]{Sternberg88}, \cite[Proposition 6.6]{Braides02}.  Furthermore,
since an $L^p(\T^2)$ fidelity term, $p\in\{1,2\}$, is clearly
continuous with respect to $L^p(\T^2)$ convergence, it is also
compatible with the $\Gamma$-limit.  The theorem below addresses the
other two $\Gamma$-limits for $k_{N,\e}$ and $k_N^\alpha$.

\begin{theorem}[Constraints] \label{thm:constraintsfork}
 \label{item:constraintska} 
$(1)$
$  k_{N,\e} + \lambda
N^{-2} |\cdot-f_N|_p^p \overset{\Gamma}{\to} k_{\infty,\e} + \lambda
\int_{\T^2} |\cdot-f|_p^p$ for $N\to \infty$ in the $L^p(\T^2)$
topology, where $p\in\{1,2\}$, $\lambda>0$, $f\in C^1(\T^2)$ and
$f_N\in \mathcal{A}_N$ is the sampling of $f$ on the grid $G_N$ ($f$,
$f_N$ and their norms can also be defined on subsets of $\T^2$ and
$G_N$ as in Theorem~$\ref{thm:constraintsforh},$
part~$\ref{item:hconstraintsa}).$  A compactness result for $k_{N,\e} +
\lambda N^{-2} |\cdot-f|_p^p$ as in
Theorem~$\ref{thm:compactnessforkNe}$ holds.

    If instead, for fixed $M\in [0,1]$, the domain of definition of
    $k_{N,\e}$ is restricted to $\mathcal{V}_N^M$ $($\textit{i.e.,}
    $\mathcal{V}^M$ from Theorem~$\ref{thm:constraintsforfe}$ on the
    grid $G_N ),$ then the $\Gamma$-convergence and compactness results
    for $N\to 0$ remain valid, with the domain of $k_{\infty,0}$
    restricted to $\mathcal{V}^M$.

$(2)$
$  k_N^\alpha + \lambda N^{-2} |\cdot-f_N|_p^p
\overset{\Gamma}{\to} k_{\infty,0} + \lambda \int_{\T^2}
|\cdot-f|_p^p$ for $N\to \infty$ in the $L^p(\T^2)$ topology, where
$p\in\{1,2\}$, $\lambda$, $f$, and $f_N$ are as in
part~$\ref{item:constraintska},$ and $\alpha\in (0,\frac2{q+3})$ as in
Theorem~$\ref{thm:GammaconvergenceforkNalpha}.$  A compactness result
for $k_N^\alpha + \lambda N^{-2} |\cdot-f|_p^p$ as in
Theorem~$\ref{thm:kNcompactness}$ and Remark~$\ref{rem:kNcompactnessL2}$
holds.

    If instead, for fixed $M\in [0,1]$, the domain of definition of
    $k_N^\alpha$ is restricted to $\mathcal{V}_N^M$, then the
    $\Gamma$-convergence and compactness results for $N\to 0$ remain
    valid, with the domain of $k_{\infty,0}$ restricted to
    $\mathcal{V}^M$.

\end{theorem}

We give a sketch of the proofs.

(1)
The compatibility of the fidelity term with the
$\Gamma$-convergence and compactness follows as in the proof of
Theorem~\ref{thm:constraintsforh}, part~\ref{item:hconstraintsb}.  As
in that theorem the mass constraint is preserved under $L^p(\T^2)$
convergence and so is compatible with both (LB) and
compactness\footnote{Both the fidelity term and the mass constraint
are not only compatible with the compactness result, but even help
with concluding uniform boundedness of either $L^1(\T^2)$ or
$L^2(\T^2)$ norm and hence can replace assumption $(W_2)$ in
Theorem~\ref{thm:compactnessforkNe} when compactness with respect to
the correct topology is considered.\label{foot:helpscompactness}}.  To
show that the mass constraint is compatible with (UB') as well we need
to check two conditions.  First, the recovery sequence which was
constructed (in the proof of Lemma~\ref{lem:kupperbound}) for $u\in
C^\infty(\T^2)$ should either satisfy or be able to be adapted to
satisfy the mass constraint.  Second, for each $u\in W^{1,2}(\T^2)$
there should be an approximating sequence $\{u_n\}_{n=1}^\infty
\subset C^\infty(\T^2)$ which has constant mass.  The latter follows
directly by the use of normalized mollifiers to construct the
approximating sequence.  For the former condition we follow an
argument reminiscent of the proof that a mass constraint is compatible
with the Modica-Mortola $\Gamma$-convergence result for the continuum
Ginzburg-Landau functional, see \textit{e.g.}
\cite{Modica87a,Sternberg88,Braides02}.  Assume that $\int_{\T^2} u =
M$ for some $M>0$.  Because $u$ is smooth it has bounded derivatives
on $\T^2$, hence (using the notation $S_{N_n}^{i,j}$ from
(\ref{eq:SNij}))
    \begin{align*}
    \int_{\T^2} u_n - u &= \sum_{i,j=1}^{N_n} \int_{S_{N_n}^{i,j}}
    \left[u(i/N_n, j/N_n) - u(x)\right]\,dx\\
    & \leq C_u \sum_{i,j=1}^{N_n} \int_{S_{N_n}^{i,j}} |(i/N_n,
    j/N_n)-x|_2\,dx \leq \tilde C_u N_n^{-1},
    \end{align*}
    for some constants $C_u$, $\tilde C_u$, depending only on $u$.
    Hence for each $n\in\N$ there is a $\delta_n =
    \mathcal{O}(N_n^{-1})$ such that $\tilde u_n := u_n + \delta_n$
    satisfies $\int_{\T^2} \tilde u_n = M$.  For this new proposed
    recovery sequence we compute
    \[
    k_{N_n,\e}(\tilde u_n) = k_{N_n,\e}(u_n) + \e^{-1} \int_{\T^2}
    \left[ W(\tilde u_n) - W(u_n)\right].  \]
    By Taylor's theorem, for $x\in\T^2$,
$$
    W(\tilde u_n(x)) - W(u_n(x)) = W'(c_n(x)) \delta_n, 
$$
    where $c_n(x) \in [u_n, u_n+\delta_n]$.  Since $u$ is continuous
    and hence bounded on $\T^2$, the sequence $\{u_n\}_{n=1}^\infty$
    is equibounded and hence, because $W\in C^2(\R)$, $W'(c_n)$ is
    equibounded.  Therefore, the sequence $\{\tilde u_n\}_{n=1}^\infty$
    is indeed a recovery sequence for (UB) with $u\in C^\infty(\T^2)$
    and satisfies the mass constraint.

(2)
As in part~\ref{item:constraintska} the addition of fidelity
terms is compatible with the $\Gamma$-limit and compactness and the
mass constraint is compatible with both (LB) and compactness.  For
compatibility with (UB') again we check two things: First that the
recovery sequence which was constructed (in the proof of
Lemma~\ref{lem:kupperboundallNs}) for $u\in BV(\T^2;\{0,1\})$ with
$\partial \supp u$ smooth either satisfies or can be adapted to
satisfy the mass constraint and second that for each $u\in
BV(\T^2;\{0,1\})$ an approximating sequence $\{u_n\}_{n=1}^\infty
\subset BV(\T^2;\{0,1\})$ can be chosen for which $\partial \supp u_n$
is smooth and which has constant mass.  The latter condition is
satisfied if we use the approximating sequence as in \cite[Theorem
3.42]{AmbrosioFuscoPallara00} and then introduce a small dilation,
diminishing along the sequence, of the support of each $u_n$ so that
the mass remains fixed (see \textit{e.g.} \cite[Proposition
7.1]{PeletierRoeger09}).  The former condition follows in a similar
way to the construction above in the case $N\to\infty$ for $k_{N,\e}$.

\subsection{Gradient flow for $k_{N,\e}$ with constraints}

To minimize $k_{N,\e}$ either under a mass constraint or with a
fidelity term we can use a gradient flow.  First consider the latter
case: $k_{N,\e,\lambda} := k_{N,\e} + \lambda |\cdot-f|_2^2$.  For $u,
v\in \mathcal{V}_N$ we compute $\text{grad}(k_{N,\e,\lambda})(u) \in
\mathcal{V}_N$ via 
$$
  \frac{d}{dt}
k_{N,\e,\lambda}(u+tv) \Big |_{t=0} = \langle
\text{grad}(k_{N,\e\,\lambda})(u), v\rangle_{\mathcal{V}}
$$
 and then
set for all $i,j\in I_N$
$$
  \frac{\partial
u_{i,j}}{\partial t} =
-\big(\text{grad}(k_{N,\e,\lambda})(u)\big)_{i,j}.
$$
 This leads to the
equation
\begin{equation}\label{eq:gradflowforkNel}
\frac{\partial u_{i,j}}{\partial t} = -4^{-r}
 \Big [2 \e 
\!\!\!\!\!
 \sum_{(k,l)\in \mathcal{N}(i,j)} 
\!\!\!\!\!
(u_{i,j}-u_{k,l}) +
\e^{-1} N^{-2} W'(u_{i,j}) + 2 \lambda (u_{i,j}-f_{i,j})  \Big ],
\end{equation}
where the set of indices of neighbors of $(i,j)$ is given by
$\mathcal{N}(i,j) = \{(i-1,j), (i+1,j), (i,j-1), (i,j+1)\}$.  The
overall prefactor $4^{-r}$ comes from the factor $d_{ij}^{-r}$ which is needed to 
cancel the factor $d_{i,j}^r$ in the $\mathcal{V}_N$ inner product. Here we assume 
the weights in this case to be equal to $1$ (on existing edges). Equation 
(\ref{eq:gradflowforkNel}) is the discretized analogue of the continuum 
Allen-Cahn equation with data fidelity
\[
\frac{\partial u}{\partial t} = 2 \e \Delta u - \e^{-1} W'(u) - 2
\lambda (u-f),
\]
which is the $L^2$ gradient flow of $F_\e^{GL}(u) + \lambda
\|u-f\|_{L^2(\T^2)}^2$.

If instead of the fidelity term a mass constraint is imposed the term
$2 \lambda (u_{i,j}-f_{i,j})$ gets replaced by a Lagrange multiplier
$$
\kappa = \e^{-1} N^{-4} \sum_{i,j=1}^N W'(u_{i,j}) .
$$
To illustrate we show simulation results using a fidelity term with
$f$ the characteristic function of a square.  We use a one-step
forward in time finite difference scheme to discretize the time
derivative, \textit{i.e.,}
\[
u_{i,j}^{n+1} \! = u_{i,j}^n - 4^{-r} dt 
 \Big [2 \e
\!\!\!\!\! 
\sum_{(k,l)\in \mathcal{N}(i,j)} 
\!\!\!\!\!
(u_{i,j}^n-u_{k,l}^n) + \e^{-1}
N^{-2} W'(u_{i,j}^n) + 2 \lambda (u_{i,j}^n-f_{i,j})  \Big ].
\]
Here $dt$ is the discrete time step and the superscript $n$ labels the
time step.  We start with a random initial condition $u^0$. We use the inner product
structure on $\mathcal{V}_N$ corresponding to the unnormalized
Laplacian ($r=0$).  Using the structure corresponding to the random
walk Laplacian ($r=1$) instead only gives an overall multiplicative
factor $\frac14$ in the right hand side of the gradient flow in
(\ref{eq:gradflowforkNel}) and hence is effectively just a time
rescaling leading to qualitatively the same behavior.

\begin{figure} 
{\includegraphics[width=0.32\textwidth]{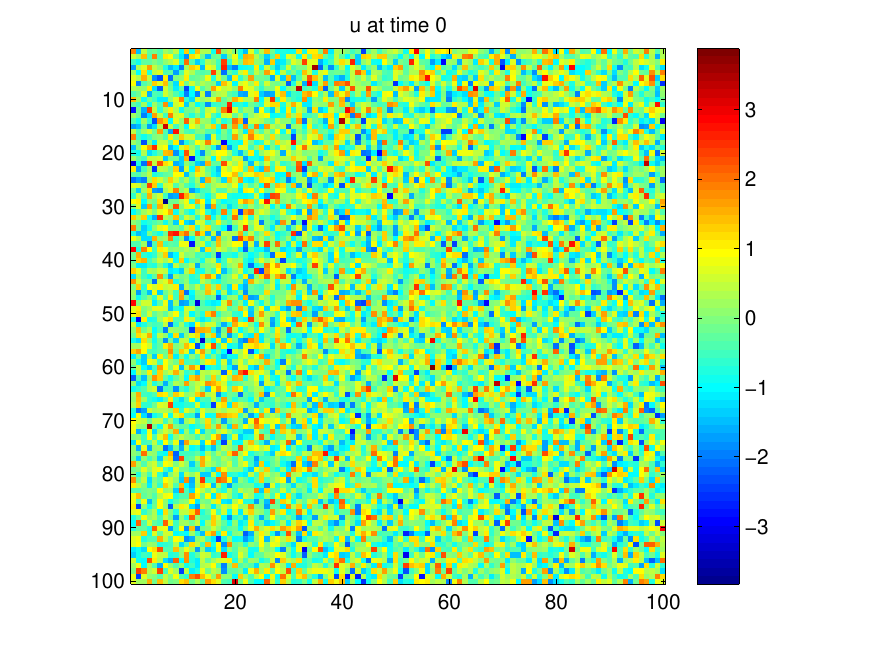}} 
{\includegraphics[width=0.32\textwidth]{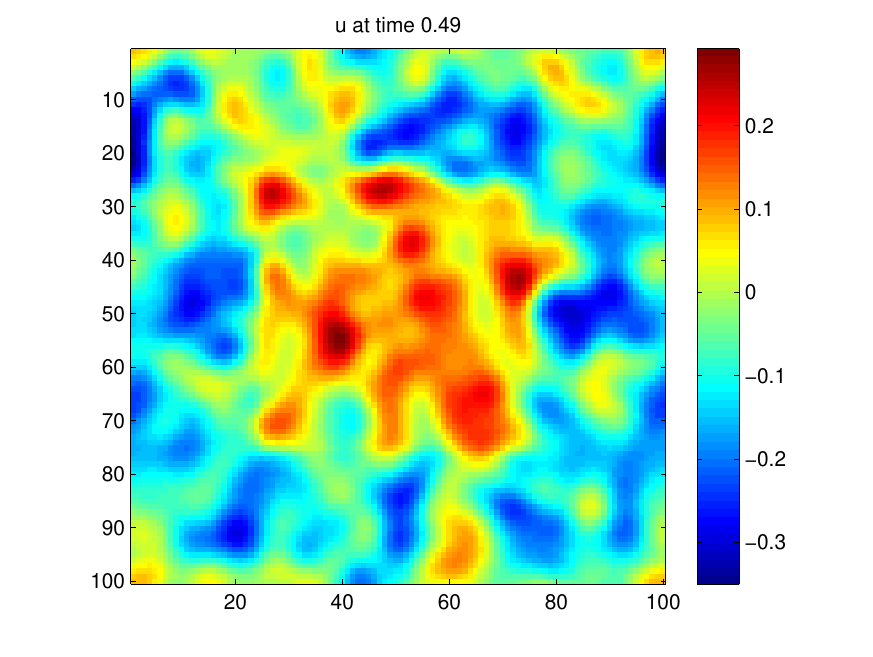}} 
{\includegraphics[width=0.32\textwidth]{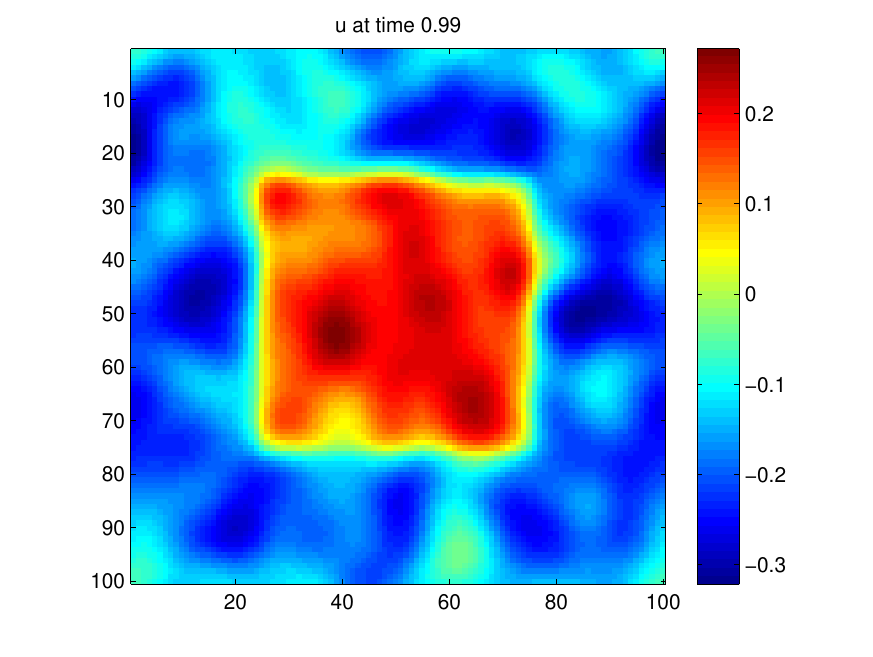}}
\\
(A) $t=0$
\hskip 70pt 
(B) $t=0.49$
\hskip 60pt
(C) $t=0.99$
\\ 
{\includegraphics[width=0.32\textwidth]{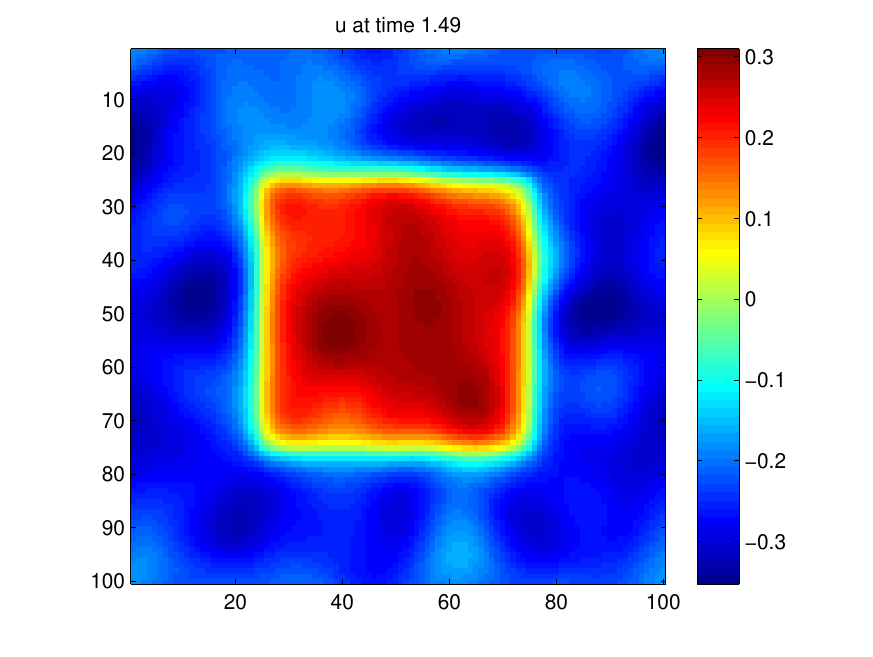}} 
{\includegraphics[width=0.32\textwidth]{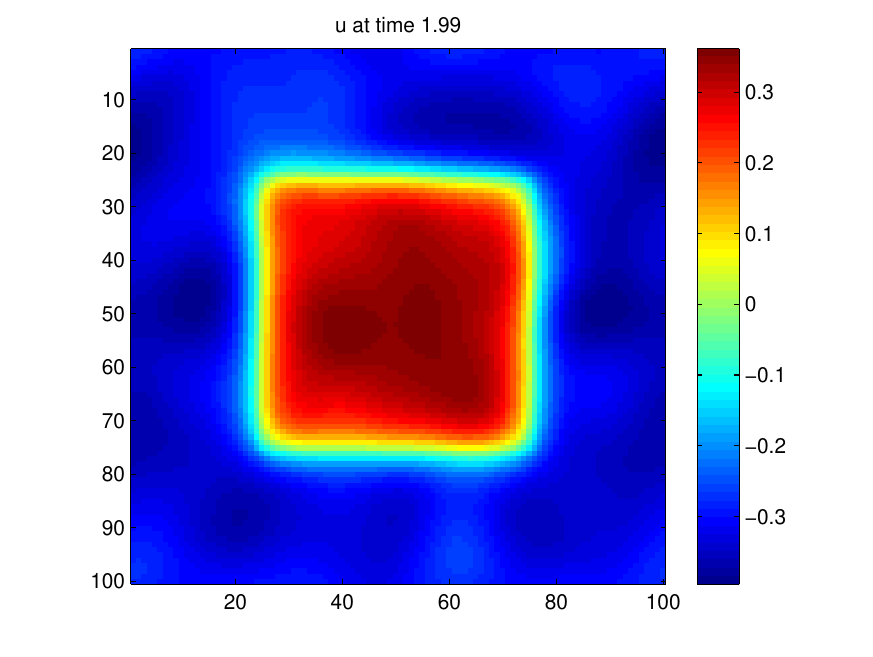}} 
{\includegraphics[width=0.32\textwidth]{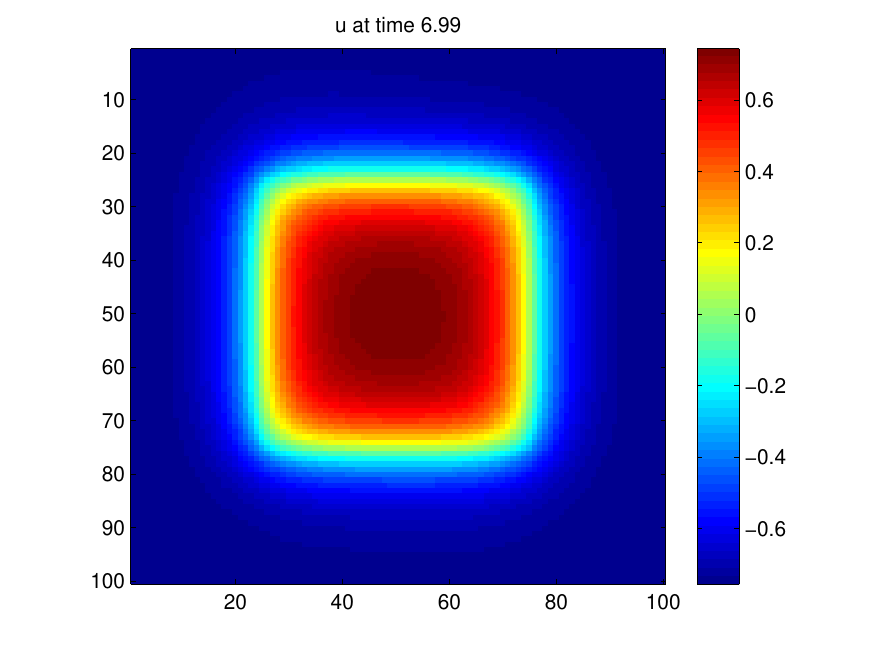}}
\\
(D) $t=1.49$
\hskip 60pt 
(E) $t=1.99$
\hskip 60pt
(F) $t=6.99$
\\ 
{\includegraphics[width=0.32\textwidth]{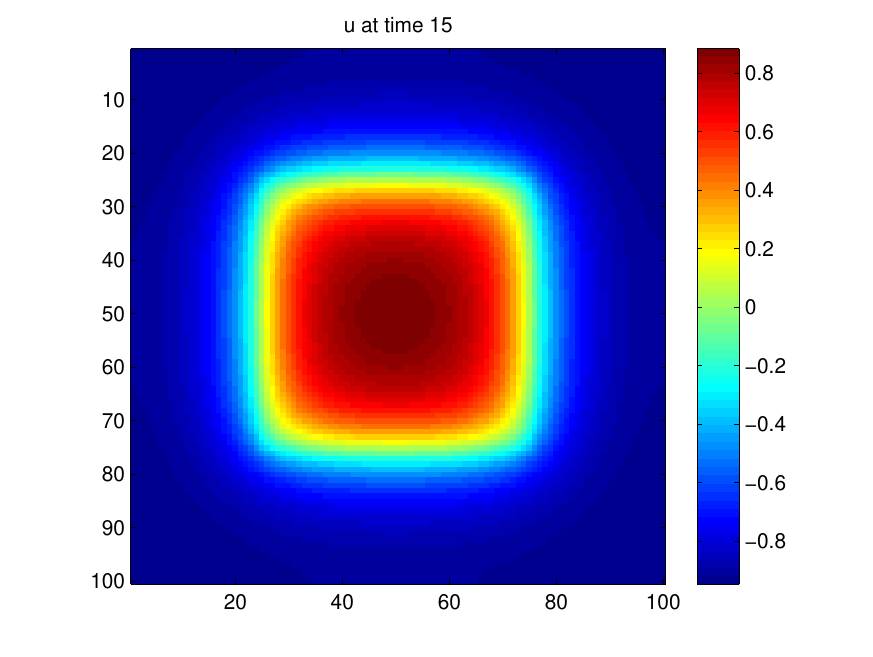}}
{\includegraphics[width=0.66\textwidth]{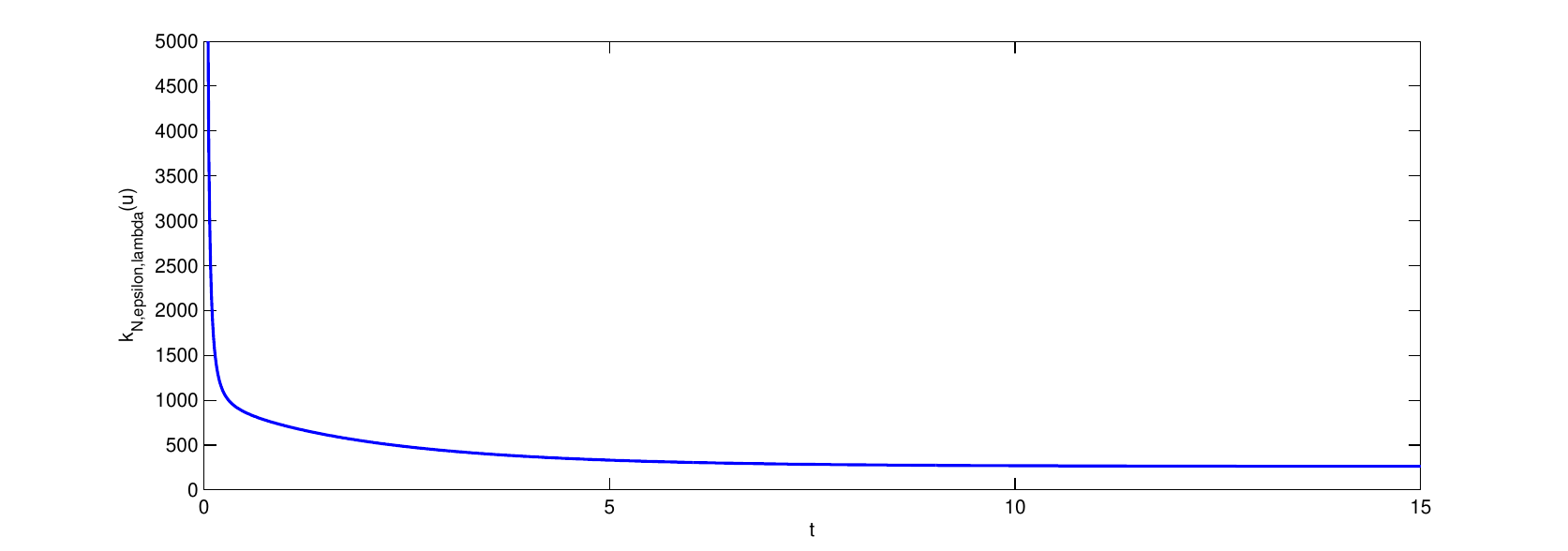}}
\\ 
(G)  $t=15$ 
\hskip 140pt
(H) \qquad\qquad\qquad\quad\quad
\caption{(A)--(G) show snap shots of the gradient flow
(\ref{eq:gradflowforkNel}) using the parameters in the text.  (h)
shows the corresponding time evolution of
$k_{N,\e,\lambda}(u)$.}\label{fig:kNe}
\end{figure}

In Figure~\ref{fig:kNe} we use the following parameter values:
$N=100$, $\e=5$, $\lambda = 0.1$ and $dt=0.01$.  We use $W(s) = s^2
(s-1)^2$ for the potential and $f$ is data prescribed to be 1 in a
square region and 0 outside that region.  

Note that $W'$ satisfies the growth condition $(W_4)$ with $q=3$,
hence according to Theorems~\ref{thm:kNcompactness}
and~\ref{thm:GammaconvergenceforkNalpha} $N$ and $\e$ should satisfy
the relation $N^{-\alpha} = \e$ for an $\alpha\in (0, 1/3)$.  The
combination $N=100$ and $\e=5$ used in Figure~\ref{fig:kNe} falls
outside this range ($\alpha \approx -0.35$), but the simulations still
give a good result.  Our theoretical results give a good guideline for
choosing $\alpha$ (especially the upper bound), but in practice
different values of $\alpha$, and hence $\e$, can produce good
gradient flow simulations.  We have chosen this larger value of $\e$
for our figures to show a stronger diffusion.

\section{The continuum limit of nonlocal means}\label{sec:NLM}

In this section we study the nonlocal means functional $g_N$ from
(\ref{eq:gN}) for a given fixed $\Phi\in C^\infty(\T^2)$.  Remember
that the weights are $\omega_{L,N} := e^{-d_{L,N}^2/\sigma^2}$ with
$\sigma, L >0$ constants (possibly depending on $N$) and $d_{L,N}$
defined in (\ref{eq:dLN}).  We define the limit weights
$\omega_{L,\sigma}, \omega_{\ell,c} \in L^\infty(\T^2\times\T^2)$ as
\[
\omega_{L,\sigma}(x,y) := e^{-\frac{4L^2}{\sigma^2}
\big(\Phi(x)-\Phi(y)\big)^2}, \qquad \omega_{\ell,c}(x,y) := e^{-c^2
\int_{S_\ell} \big(\Phi(x+z)-\Phi(y+z)\big)^2\,dz},
\]
where $\ell, c>0$ and $S_\ell := \{z\in \R^2: |z_1|+|z_2| \leq
\ell\}$.  The limit functionals $g_\infty^L: L^1(\T^2) \to \R$ and
$g_\infty^\ell: L^1(\T^2) \to \R$ are
\begin{align*}
g_\infty^L(u) 
&
:= \int_{\T^2} \int_{\T^2} \omega_{L,\sigma}(x,y)
|u(x)-u(y)|\,dx\,dy, \\
g_\infty^\ell(u) 
&
:= \int_{\T^2} \int_{\T^2}
\omega_{\ell,c}(x,y) |u(x)-u(y)|\,dx\,dy.
\end{align*}
We prove a $\Gamma$-convergence result.

\begin{theorem}[$\Gamma$-convergence]\label{thm:GammaconvergencegN} 
  $(1)$
If $\sigma, L>0$ are fixed, then $\displaystyle g_N
\overset{\Gamma}\to g_\infty^L$ as $N\to\infty$, in the $L^1(\T^2)$
topology.  

 $(2)$
 If $\sigma = \frac{N}c$ for some $c>0$ and $L$ is
such that $L/N \to \ell$ as $N\to \infty$, for some $\ell\in (0,1/2)$,
then $\displaystyle g_N \overset{\Gamma}\to g_\infty^\ell$ as
$N\to\infty$, in the $L^1(\T^2)$ topology. 
\end{theorem}

As explained in Remark~\ref{rem:nocompactnessforgN} we do \emph{not}
have compactness in this case.

Note that $g_N$ is the functional $f_0$ from Theorem~\ref{thm:Gammaf}
where the graph $G$ is the grid $G_N$ and the choices $\chi=N^{-4}$
and $\omega = \omega_{L,N}$ have been made.  Two main differences
between this functional and the previous functionals on the grid $G_N$
we considered are that the graph is now completely connected and the
weights are not uniform over the edges.  For the latter reason it is
useful to introduce a notation for the space of graph weights on
$G_N$.  Given a weight function $\omega$ and nodes $n_{i,j}, n_{k,l}
\in V_N$ we write $\omega_{i,j,k,l} := \omega(n_{i,j}, n_{k,l})$.
Define
\[
\mathcal{W}_N := \{ \omega: V_N \times V_N \to [0,\infty): \text{ for
all } i,j,k,l\in I_{N^2}\,\, \omega_{i,j,k,l} = \omega_{k,l,i,j}\}.
\]
Completely analogous to the identification between $\mathcal{V}_N$ and
$\mathcal{A}_N$ which was introduced in Section~\ref{sec:proofh}, we
can identify $\mathcal{W}_N$ with
\[
\Omega_N := \{ \omega: \T^2 \times \T^2 \to [0,\infty): \text{ for all
} x,y \in \T^2\,\, \omega(x,y) = \omega(y,x)\}.
\]
Identifying $\omega_{L,N}\in \mathcal{W}_N$ with the corresponding
$\omega_{L,N} \in \Omega_N$ and $u\in \mathcal{V}_N^b$ with the
corresponding $u\in \mathcal{A}_N^b$ we can write
\[
g_N(u) = \left\{\begin{array}{ll} \int_{\T^2} \int_{\T^2}
\omega_{L,N}(x,y) |u(x)-u(y)|\,dx\,dy & \text{if }
u\in\mathcal{A}_N^b,\\
+\infty & \text{if } u\in L^1(\T^2)\setminus\mathcal{A}_N^b.
\end{array}\right.
\]

We prove Theorem~\ref{thm:GammaconvergencegN} in two steps.  First we
show that uniform convergence of the weights suffices for
$\Gamma$-convergence of $g_N$ and then we show that uniform
convergence.

\begin{lemma}\label{lem:unifweightsgammafunc}
Let $\{\omega_N\}_{N=1}^\infty$ be such that $\omega_N \in \Omega_N$
and $\omega_N \to \omega$ uniformly as $N\to\infty$ for some
$\omega\in L^\infty(\T^2\times\T^2)$.  For $N\in\N$ define the
functional $g_\infty: L^1(\T^2) \to \R$ by
\[
g_\infty(u) := \int_{\T^2} \int_{\T^2} \omega(x,y)
|u(x)-u(y)|\,dx\,dy,
\]
then $g_N \overset{\Gamma}\to g_\infty$ as $N\to\infty$ in the
$L^1(\T^2)$ topology.
\end{lemma}

\begin{proof}[\bf Proof.]
Let $\{N_n\}_{n=1}^\infty \subset \N$ such that $N_n\to \infty$ as
$n\to\infty$, $u\in L^1(\T^2)$ and $\{u_n\}_{n=1}^\infty \subset
L^1(\T^2)$ such that $u_n \to u$ in $L^1(\T^2)$ as $n\to\infty$.
Assume that for each $n\in\N\,\,$ $u_n \in \mathcal{A}_{N_n}^b$.  
Then
\begin{align*}
&  \Big  | \int_{\T^2} \int_{\T^2} \omega_{N_n}(x,y) |
u_n(x) - u_n(y)|\,dx\,dy - \int_{\T^2} \int_{\T^2} \omega(x,y) | u(x)
- u(y)|\,dx\,dy \Big | \\
&=  \Big | \int_{\T^2} \int_{\T^2} (\omega_{N_n}(x,y) - \omega(x,y))
|u_n(x)-u_n(y)|\,dx\,dy   \\
 &  
\qquad + \int_{\T^2}\int_{\T^2} \omega(x,y) \left(
 |u_n(x)-u_n(y)|-|u(x)-u(y)|\right)\,dx\,dy \Big |\\
 &\leq I_1 + \int_{\T^2}\int_{\T^2} \omega(x,y)
 |u_n(x)-u_n(y)-u(x)+u(y)|\,dx\,dy \leq I_1 + 2 I_2,
\end{align*}
where for simplicity we have used the notation 
\begin{align*}
&
I_1(u) := \int_{\T^2}
\int_{\T^2} |\omega_{N_n}(x,y) - \omega(x,y)| |u_n(x)-u_n(y)|\,dx\,dy 
\\
&
I_2(u) := \int_{\T^2}\int_{\T^2} \omega(x,y) |u_n(x)-
u(x)|\,dx\,dy .
\end{align*}
Since $\omega_{N_n} \to \omega$ uniformly as $n\to\infty$, there is a
sequence of constants $C_n>0$ such that for $n$ large enough
$|\omega_{N_n}(x,y) - \omega(x,y)| \leq C_n$ and $C_n\to 0$ as
$n\to\infty$.  Hence,
$$
I_1 \leq C_n \int_{\T^2} \int_{T^2}
|u_n(x)-u_n(y)|\,dx\,dy . 
$$
 Because $u_n \in \mathcal{A}_{N_n}^b$ we
have $|u_n(x)-u_n(y)| \leq 2$ for almost all $(x,y)\in \T^2\times
\T^2$ and hence $I_1 \to 0$ as $n\to \infty$.  Furthermore
$
I_2 \leq \|\omega\|_{L^\infty(\T^2)} \|u_n-u\|_{L^1(\T^2)}
$
and thus $I_2 \to 0$ as $n\to\infty$.  We conclude that
$
\underset{n\to\infty}\lim\, g_{N_n}(u_n) = g_\infty(u)
$
for any sequence $\{u_n\}_{n=1}^\infty$ such that $u_n \in
\mathcal{A}_{N_n}^b$ and $u_n\to u$ in $L^1(\T^2)$ as $n\to\infty$.
In particular (LB) is proven.  To prove (UB') all that remains is to
show that there exists such a sequence.

First assume $u\in C^\infty(\T^2)$ and define $u_n(x) := u(i/{N_n},
j/{N_n})$ where $i, j \in I_N$ are such that $x\in S_{N_n}^{i,j}$ from
(\ref{eq:SNij}).  This sequence satisfies the required conditions,
hence (UB') is proved for $u\in C^\infty(\T^2)$.  We conclude the
argument by using the lower semicontinuity of the upper $\Gamma$-limit
and density of $C^\infty(\T^2)$ in $L^1(\T^2)$ as in the proof of
Theorems~\ref{lem:kupperbound} and ~\ref{lem:kupperboundallNs} to
deduce (UB') for $u\in L^1(\T^2)$.
\end{proof}

\begin{remark}
Note that $g_N$ does not converge uniformly to $g_\infty$, because if
$u\in L^1(\T^2)$ is such that for all $N\in\N\,\,$ $u\not\in
\mathcal{A}_N^b$, then $|g_N(u)-g_\infty(u)| = \infty$.  If we would
restrict the domains of $g_N$ and $g_\infty$ to continuous $u$ and
define $g_N$ to be
\[
g_N(u) := \int_{\T^2} \int_{\T^2} \omega_N(x,y)
|u_N(x)-u_N(y)|\,dx\,dy,
\]
where $u_N(x) = u(i/N, j/N)$ where $i$ and $j$ are such that $x\in
S_{N_n}^{i,j}$, then $g_N\to g_\infty$ uniformly by the estimates in
the proof of Lemma~\ref{lem:unifweightsgammafunc}.  By
\cite[Proposition 5.2]{DalMaso93} $g_N$ then $\Gamma$-converges to the
lower semicontinuous envelope of $g_\infty$.
\end{remark}

\begin{lemma}\label{lem:unifconvofweights} 
$(1)$
If $\sigma, L>0$ are fixed, then $\omega_{L,N} \to
\omega_{L,\sigma}$ uniformly as $N\to\infty$. 
  
$(2)$ If $\sigma =
\frac{N}c$ for some $c>0$ and $L$ is such that $L/N \to \ell$ as $N\to
\infty$, for some $\ell\in (0,1/2)$, then $\omega_{L,N} \to
\omega_{\ell,c}$ uniformly as $N\to\infty$. 
\end{lemma}

We defer the relatively straightforward proof to
Appendix~\ref{sec:deferredproofs}.

\smallskip

\noindent
{\bf Proof of Theorem~\ref{thm:GammaconvergencegN}.}
Combining Lemmas~\ref{lem:unifweightsgammafunc}
and~\ref{lem:unifconvofweights} the result follows directly.
\qed

\begin{remark}\label{rem:nocompactnessforgN}
It is important to note that for $g_N$ we do \emph{not} have a
compactness result in the $L^1(\T^2)$ topology.  If we have sequences
$\{N_n\}_{n=1}^\infty$ and $\{u_n\}_{n=1}^\infty$ such that
$N_n\to\infty$ as $n\to\infty$ and $u_n\in \mathcal{A}_{N_n}^b$, the
bound on $\|u_n\|_{L^\infty(\T^2)}$ allows us to conclude that $u_n
\overset{*}\rightharpoonup u$ for some subsequence (labelled again by
$n$) and some $u\in L^\infty(\T^2)$.  In order to deduce $L^1(\T^2)$
convergence we would need some information on the derivatives (or
finite differences) which we do not have when the weights $\omega$ are
nonsingular.  A uniform bound on $g_{N_n}(u_n)$ adds no useful
information since $g_{N_n}(u_n) \leq 1$ per definition, if $u_n\in
\mathcal{A}_{N_n}^b$.

A simple counterexample is the case where $\Phi$ is constant, hence
the graph weight function $\omega \equiv 1$.  Let $u_N \in
\mathcal{A}_N^b$ be a checker board pattern on $G_N$, \text{i.e.,} as
function in $\mathcal{V}_N^b\,\,$ $(u_N)_{0,0}=0$ and $(u_N)_{i,j}
\neq (u_N)_{i+1,j} = (u_N)_{i,j+1}$ for all $i, j$, then
\[
g_N(u_N) = \int_{\T^2} \int_{\T^2} |u_N(x)-u_N(y)|\,dx\,dy = 2 |\supp
u_N|\,\, |(\supp u_N)^c| \leq 2.
\]
However no subsequence of $\{u_N\}_N$ converges in $L^1(\T^2)$ as can
be seen as follows.  Let $M\gg N$ then the square $S_N^{i,j}$ contains
$\mathcal{O}\Big(\big(\frac{M}N\big)^2\Big)$ squares of size $M^{-1}$
by $M^{-1}$.  On approximately half (at least $\mathcal{O}(1)$) of
these $u_M \neq u_N$, so $\displaystyle \int_{\T^2} |u_N - u_{M}| =
N^2 M^{-2} \mathcal{O}\Big(\big(\frac{M}N\big)^2\Big) \mathcal{O}(1) =
\mathcal{O}(1)$.
\end{remark}

\section{Discussion and open questions}\label{sec:discopen}

In this paper we have shown various $\Gamma$-convergence results.  The
convergence of $f_\e$ in Section~\ref{sec:Gammagraph} shows that we
can extend the classical Modica-Mortola $\Gamma$-convergence result
for the Ginzburg-Landau functional to graphs, if we are careful about
the precise scaling.  The discrete nature of a graph forces us to not
include an $\e$ in the finite difference term, unlike the $\e$ in the
gradient term for the continuum Ginzburg-Landau functional.  As has
been shown on a specific regular square grid in
Section~\ref{sec:differentscalingshNe}, this has consequences for the
limit functional, which now behaves like an anisotropic instead of
isotropic total variation.  We recovered the isotropic total variation
for the regular grid case in Section~\ref{sec:differentscalingskNe} by
taking an approach reminiscent of classical numerical analytic
results, instead of graph based results.  Specifically, to do this we
need to choose a scaling in line with standard finite difference and
quadrature methods and make the limit $N\to \infty$ dominant over the
limit $\e\to0$ such that, in a sense, we first get back to the
continuum case, before passing to the total variation.  The lesson in
here is twofold.  On the one hand it shows that one has to be careful
when discretizing on a grid not to pick up grid direction, which has
been known to numerical analysts for a long time.  On the other hand
however, it entices us to look at graph based functionals and
nonlinear partial differential equations in their own right, because
they can behave in surprising ways if the topology of the graph is
allowed to interact with the functional or PDE. This conclusion is
reminiscent of the behavior which is found in
\cite{HeinAudibertvonLuxburg07}.  In that paper the authors study the
limit of the graph Laplacian on a graph which is constructed by
sampling points from a manifold.  They find the limit is independent
of the sampling distribution only for a specific scaling of the graph
Laplacian.

In Section~\ref{sec:NLM} we studied the limit of a functional of
nonlinear means type on graphs, showing that while the limit exists,
the nonlocal nature of the functional leads to loss of compactness.
This is not expected to be a specific problem of the graph based
nature of the functional, but of the nonlocality and as such is
expected to be present for a continuum version of $g_N$ as well.

One question raised in Section~\ref{sec:discussionalpha}, is whether
the range of $\alpha$ under which $\Gamma$-convergence and compactness
of $k_N^\alpha$ can be proven, can be extended to $(0,1)$.  This is an
important question in practice when running gradient flow simulations.
A choice of $\e$ which is too small or large can lead to either
pinning or too fast diffusion respectively.

The $\Gamma$-convergence results for $h_{N,\e}$ and $k_{N,\e}$
naturally lead to the question of the limit behavior for other graphs.
In order to have a good interpretation for that question it is in the
first place necessary to have a structured way in which to increase
$m$, the number of nodes for the graph.  A triangulation might be the
natural next step.  A random Erd{\H{o}}s-R{\'e}nyi graph
\cite{ErdosRenyi59} also carries a natural rule how to connect new
nodes to the graph and may be an interesting exploration into the
question whether the Ginzburg-Landau functional on a graph without
explicit spatial embedding can possibly have `continuum' limit.  For
arbitrary graphs it is less clear how to add new nodes in a structured
way.  One option could be to construct a sequence of graphs where in
each next step each existing edge is bisected by a new node.

A question that is very relevant for the applications of the
Ginzburg-Landau functional is that of stability of minimizers with
respect to perturbations of the graph (\textit{e.g.} perturb the
weights or add or delete nodes).  For example, in data analysis, if
the nodes represent data points and the edge weights measure
similarity, it is quite likely that noise is present in the weights.

\section{Acknowledgements}

We would like to thank Arjuna Flenner for many useful discussions. 
We also thank Dejan Slepcev and Thomas Laurent for recently drawing 
our attention to an independent and as yet unpublished manuscript 
\cite{BraidesYip12}, which overlaps with some of the results here. 

The research in this paper was made possible by funding from 
ONR grant N000141210040 , ONR grant N000141210838, 
ONR grant N000141010221 and AFOSR MURI grant FA9550-10-1-0569.

\appendix

\section{Choice of Hilbert space and difference
structure}\label{sec:structureexplained}

In this section we will give some background information and
derivations concerning the choices made in Section~\ref{sec:LaplDirTV}
that defined our graph operators and functionals.

We start by associating $\mathcal{V}$ and $\mathcal{E}$ with the
finite dimensional vector spaces $\R^m$ and $\R^{m(m-1)/2}$ respectively.
We will turn these vector spaces into Hilbert spaces by defining inner
products on them.  We follow the procedure described in \cite[Section
2]{HeinAudibertvonLuxburg07}\footnote{We slightly deviate from
\cite{HeinAudibertvonLuxburg07}.  Instead of sums $\sum_{i=1}^m$ they
use averages $\frac1m \sum_{i=1}^m$, which is a choice not unanimously
adopted in the literature, but leads to cleaner convergence statements
in \cite{HeinAudibertvonLuxburg07}.  We could adopt this convention in
this paper, but it would not significantly alter our results, mutatis
mutandis.}.  For $u, v\in \mathcal{V}$ and $\varphi, \phi \in
\mathcal{E}$ we define
\[
\langle u, v \rangle_{\mathcal{V}} := \sum_{i\in I_m} u_i v_i
\alpha(d_i), \qquad \langle \varphi, \phi \rangle_{\mathcal{E}} :=
\frac12 \sum_{i,j\in I_m} \varphi_{ij} \phi_{ij} \beta(\omega_{ij}),
\]
where $\alpha, \beta: [0, \infty) \to [0, \infty)$ are functions yet
to be determined.  Note that a priori we allow $\alpha$ and $\beta$ to
take the value zero, which means that the above `inner products' might
be positive semi-definite and not positive definite.  We will get back
to this issue after we have decided on our 
choices of $\alpha$~and~$\beta$\footnote{Note that if $\alpha$ and $\beta$ are such that
positive definiteness is satisfied, these inner products do indeed
turn $\mathcal{V}$ and $\mathcal{E}$ into Hilbert spaces, since
convergence with respect to the induced $\mathcal{E}$ norm preservers
skew-symmetry.}.

As in \cite{GilboaOsher08} we also define the dot product for
$\varphi, \phi \in \mathcal{E}$ as
\[
(\varphi\cdot\phi)_i := \frac12 \sum_{j\in I_m} \varphi_{ij} \phi_{ij}
\beta(\omega_{ij}).
\]
Note that $\varphi\cdot\phi \in \mathcal{V}$.
As explained in \cite{HeinAudibertvonLuxburg07} we can now define the
difference operator or gradient $\nabla: \mathcal{V} \to \mathcal{E}$
as
\[
(\nabla u)_{ij} := \gamma(\omega_{ij}) (u_j-u_i),
\]
where $\gamma: [0, \infty] \to [0, \infty]$ is a third yet to be
determined function.  With this choice for the gradient we find that
its adjoint, the divergence $\dvg: \mathcal{E}\to \mathcal{V}$, is
given by \cite[Lemma 3]{HeinAudibertvonLuxburg07}
\[
(\dvg \varphi)_i := \frac1{2\alpha(d_i)} \sum_{j\in I_m}
\beta(\omega_{ij}) \gamma(\omega_{ij}) (\varphi_{ji}-\varphi_{ij}).
\]
This expression follows from the defining property of the adjoint:
$\langle \nabla u, \varphi \rangle_{\mathcal{E}} = \langle u,
\dvg\varphi\rangle_{\mathcal{V}}$ for all $u\in \mathcal{V}$ and all
$\varphi\in \mathcal{E}$.

Now that we have inner products, a gradient operator, and a divergence
operator, we can define the following objects:

$\bullet$
Inner product norms $\|u\|_{\mathcal{V}} := \sqrt{\langle
u,u\rangle_{\mathcal{V}}}$ and $\|\varphi\|_{\mathcal{E}} :=
\sqrt{\langle \varphi, \varphi\rangle_{\mathcal{E}}}$.

$\bullet$
 Maximum
norms\footnote{To justify these definitions and 
convince ourselves that there should be no $\beta$ or $\gamma$ 
included in the maximum norms we define $\|\varphi\|_{\mathcal{E}, p}^p := 
\frac12 \sum_{i,j\in I_m} \varphi_{ij}^2 \beta(\omega_{ij})$. 
Adapting the proofs in the continuum case in \textit{e.g.} 
\cite[Theorems 2.3 and 2.8]{Adams75} to the graph situation 
we can prove a H\"older inequality $\displaystyle 
\|\varphi \phi\|_{\mathcal{E}, 1} \leq \|\varphi\|_{\mathcal{E},p} 
\|\phi\|_{\mathcal{E},q}$ for $1<p, q< \infty$ such that 
$\frac1p+\frac1q=1$, an embedding theorem of the form
$\displaystyle \|\varphi\|_{\mathcal{E},p} \leq \left( \frac12 \sum_{i,j\in I_m} 
\beta(\omega_{ij})\right)^{\frac1p-\frac1q} \|\varphi\|_{\mathcal{E},q}$ 
for $1 \leq p \leq q \leq \infty$ and the limit $\displaystyle 
\underset{p\to\infty}\lim\, \|\varphi\|_{\mathcal{E},p} = 
\|\varphi\|_{\mathcal{E}, \infty}$. A similar result holds for the norms 
on $\mathcal{V}$.}
$\|u\|_{\mathcal{V},\infty} := \max\{|u_i|: i\in I_m\}$ and \\ 
$\|\varphi\|_{\mathcal{E},\infty} := \max\{|\varphi_{ij}|: i,j \in
I_m\}$.

$\bullet$
  The norm corresponding to the dot product
$\|\varphi\|_i := \sqrt{( \varphi \cdot \varphi)_i}$. \\
 Note that
$\|\cdot\|_{\mathcal{E},\text{dot}} \in \mathcal{V}$.

$\bullet$
  The
Dirichlet energy $\frac12 \|\nabla u\|_{\mathcal{E}}^2$.

$\bullet$
  The
graph Laplacian $\Delta := \dvg\circ\nabla: \mathcal{V} \to
\mathcal{V}$.  So
\[
(\Delta u)_i := \frac1{\alpha(d_i)} \sum_{j\in I_m} \beta(\omega_{ij})
\gamma^2(\omega_{ij}) (u_i-u_j).
\]

$\bullet$
  The isotropic and anisotropic total variation $TV:
\mathcal{V}\to \R$ and $TV_a: \mathcal{V}\to \R$ respectively:
\begin{align*}
\TV(u) &:= \max\{ \langle \dvg \varphi, u\rangle_{\mathcal{V}} :
\varphi\in \mathcal{E},\,\, \underset{i\in
I_m}\max\,\,\|\varphi\|_i\leq 1\}.\\
\TVa(u) &:= \max\{ \langle \dvg \varphi, u\rangle_{\mathcal{V}} :
\varphi\in \mathcal{E},\,\, \|\varphi\|_{\mathcal{E},\infty}\leq 1\}.
\end{align*}
We note that by the property of the adjoint we can also use $\langle
\nabla u, \varphi\rangle_{\mathcal{E}}$ in the definitions above
instead of $\langle \dvg \varphi, u\rangle_{\mathcal{V}}$.  In this
finite dimensional setting these maxima over unit balls will be
achieved, hence we are justified in using $\max$ instead of $\sup$.

Before we start making specific choices for $\alpha$, $\beta$, and
$\gamma$ it is interesting to make some general observations which do
not depend on these choices.

$\bullet$
  We can alternatively derive the Laplacian via the variational
principle from the Dirichlet energy as follows.  Consider $u, v \in
\mathcal{V}$ and $t\in \R$, then
\begin{align*}
\left.\frac{d}{dt} \frac12 \|\nabla
u+tv\|_{\mathcal{E}}^2\right|_{t=0} &= \frac12 \sum_{i,j\in I_m}
\beta(\omega_{ij}) \gamma^2(\omega_{ij}) (u_i-u_j) (v_i-v_j)\\
&= \sum_{i,j\in I_m} \beta(\omega_{ij}) \gamma^2(\omega_{ij})
(u_i-u_j) v_i\\
&= \sum_{i,j\in I_m} \frac{\beta(\omega_{ij})
\gamma^2(\omega_{ij})}{\alpha(d_i)} (u_i-u_j) v_i \alpha(d_i) =
\langle \Delta u, v \rangle_{\mathcal{V}}.
\end{align*}
If we choose $v=u$ this also shows that an analogue of `integration by
parts' holds:
\[
\langle \Delta u, u \rangle_{\mathcal{V}} = \frac12 \sum_{i,j\in I_m}
\beta(\omega_{ij}) \gamma^2(\omega_{ij}) (u_i-u_j)^2 = \|\nabla
u\|_{\mathcal{E}}^2.
\]

$\bullet$
  We can also give a variational TV-type formulation of the
Dirichlet energy itself via:
\[
\|\nabla u\|_{\mathcal{E}} = \max\{ \langle \dvg \varphi,
u\rangle_{\mathcal{V}} : \varphi\in \mathcal{E},\,\,
\|\varphi\|_{\mathcal{E}}\leq 1\}.
\]
To see why this holds we first remember that $\langle \nabla u,
\varphi\rangle_{\mathcal{E}}=\langle \dvg \varphi,
u\rangle_{\mathcal{V}}$.  Then we see that by the Cauchy--Schwarz
inequality 
$$
\langle \nabla u, \varphi\rangle_{\mathcal{E}} \leq
\|\nabla u\|_{\mathcal{E}} \|\varphi\|_{\mathcal{E}} \leq \|\nabla
u\|_{\mathcal{E}} . 
$$
 Equality is achieved when 
$$
\varphi =
\varphi^{\mathcal{E}}(u) := \left\{\begin{array}{ll}\frac{\nabla
u}{\|\nabla u\|_{\mathcal{E}}} & \text{if } \|\nabla u\|_{\mathcal{E}}
\neq 0,\\ 0 & \text{if } \|\nabla u\|_{\mathcal{E}} =
0\end{array},\right.
$$
 which is permissible since
$\|\varphi^{\mathcal{E}}(u)\|_{\mathcal{E}}\leq 1$.

We will now make particular choices for $\alpha$, $\beta$, and
$\gamma$.  Our choices will be driven by the desire to satisfy the
following properties.

(1)
We will consider a family of graph Laplacians indexed by a
parameter $r\leq 1$ (not to be confused with the $p$-Laplacian from
the literature).  As it turns out only the choice of $\alpha$ is
influenced by the choice of $r$.  The Laplacians we consider are
    \[
    (\Delta_r u)_i := d_i^{1-r} u_i - \sum_{j\in I_m}
    \frac{\omega_{ij}}{d_i^r} u_j = \sum_{j\in I_m}
    \frac{\omega_{ij}}{d_i^r} (u_i-u_j) \]

    As explained in Section~\ref{sec:LaplDirTV} by choosing either
    $r=0$ or $r=1$ we recover the unnormalized or random walk
    Laplacian respectively.  To construct the symmetric normalized
    Laplacian as it appears in the literature requires a gradient of
    the form 
$$
(\nabla u)_{ij} = \gamma(\omega_{ij})
    \Big (\frac{u_j}{\sqrt{d_j}} - \frac{u_i}{\sqrt{d_i}} \Big )
$$
    (\textit{cf.} \cite{HeinAudibertvonLuxburg07}).  This falls
    outside our current framework and hence we will not consider it
    here.\footnote{As a word of caution we note that the use of the
    symmetric normalized Laplacian in combination with a double well
    potential $W$ with wells that are not symmetrically placed around
    0 (as in the case where the wells are at 0 and 1) causes an
    asymmetry between the two phases that is typically unwanted.} Some
    discussion of the pros and cons of different graph Laplacians can
    be found in \textit{e.g.}
    \cite{vonLuxburg07,WardetzkyMathurKaelbererGrinspun07}.

(2)
The Dirichlet energy is given by $\frac12 \|\nabla
u\|_{\mathcal{E}}^2 = \frac14 \sum_{i,j\in I_m} \omega_{ij}
(u_i-u_j)^2$, independently of the choice of $r$ in the Laplacian.

(3)
 The isotropic total variation is
\[
\TV(u) =\sum_{i\in I_m} \|\nabla u\|_i = \frac12 \sqrt2 \sum_{i\in
I_m} \sqrt{\sum_{j\in I_m} \omega_{ij} (u_i-u_j)^2}
\]
(\textit{cf.} \cite{GilboaOsher08} where $\TV$ is called nonlocal TV
because the graph is assumed to be embedded in an Euclidean space and
so what is local on the graph (neighboring vertices) might not be
local in the embedding space).

(4)
  We will consider a family of anisotropic total variations
parametrized by the parameter $q\in [1/2, 1]$\footnote{We can take
$q\in \R$ if we interpret $\omega_{ij}^q$ as zero whenever
$\omega_{ij}=0$.}:
\[
{\TVa}_q(u) = \langle \nabla u, \sgn(\nabla u)\rangle_{\mathcal{E}} =
\frac12 \sum_{i,j\in I_m} \omega_{ij}^q |u_i-u_j|.
\]
The parameter $q$ comes in via the definitions of $\beta$ and $\gamma$
and the signum function is understood to act element-wise on the
elements of $\nabla u$.

Let us consider all the points above to find out what conditions we
have to impose on $\alpha$, $\beta$, and $\gamma$ to satisfy this list
of requirements.

\begin{enumerate}

\item\label{item:Laplacians} As can be seen in the definition of the
Laplacian (\textit{cf.} also \cite[Definition
7]{HeinAudibertvonLuxburg07}) in order to get the desired Laplacians
we have to choose $\alpha$, $\beta$, and $\gamma$ such that for each
$\omega_{ij}$ and each $d_i$:
\[
\frac{\beta(\omega_{ij}) \gamma^2(\omega_{ij})}{\alpha(d_i)} =
\frac{\omega_{ij}}{d_i^r}.
\]
Specifically $\beta(\omega_{ij}) \gamma^2(\omega_{ij}) = \omega_{ij}$
for any choice of $r$ and $\alpha(d_i) = d_i^r$.  We will see below
that the choice of $\alpha$ is irrelevant for the
points~\ref{item:Dirichletenergy}--\ref{item:TVa} and hence all
choices of $r$ are compatible with what follows.

\item\label{item:Dirichletenergy} For the Dirichlet energy we compute
\[
\frac12 \|\nabla u\|_{\mathcal{E}}^2 = \frac14 \sum_{i,j\in I_m}
(u_i-u_j)^2 \beta(\omega_{ij}) \gamma^2(\omega_{ij}) = \frac14
\sum_{i,j\in I_m} \omega_{ij} (u_i-u_j)^2.
\]
Since the graph Laplacian appears as the natural operator in the
Euler-Lagrange equation associated with the Dirichlet energy it is not
surprising that we do not get any extra conditions on $\alpha$,
$\beta$ or $\gamma$ from the Dirichlet energy which we didn't already
get from the Laplacian.  It is interesting to note though that the
Dirichlet energy does not depend on the choice of $\alpha$ (and hence
$r$ in the Laplacian) at all.

\item For the isotropic total variation $\TV$ we use the
Cauchy-Schwarz inequality on the dot product norm to get $\langle
\nabla u, \varphi\rangle_{\mathcal{E}} = \sum_{i\in I_m} (\nabla u
\cdot \varphi)_i \leq \sum_{i\in I_m} \|\nabla u\|_i \|\varphi\|_i
\leq \sum_{i\in I_m} \|\nabla u\|_i$.  To achieve equality\footnote{Note 
that demanding $\varphi^{TV}$ to achieve equality does not determine 
it uniquely on the set of vertices for which $\|\nabla u\|_i = 0$.\label{foot:nonuniquevarphi}} let
$\varphi_{ij} = \varphi_{ij}^{\TV}(u) := \left\{
\begin{array}{ll}\frac{(\nabla u)_{ij}}{\|\nabla u\|_i} & \text{if }
\|\nabla u\|_i \neq 0,\\ 0 & \text{if } \|\nabla u\|_i=0\end{array}.\right.$ 
Again we do not require extra conditions
on $\beta$ and $\gamma$.  They will be determined by the last
requirement\footnote{If we would have defined the dot product
$(\varphi\cdot\phi)_i:= \frac12 \sum_{j\in I_m} \varphi_{ij} \phi_{ij}
\delta(\omega_{ij})$ for a function $\delta$ possibly different than
$\beta$, the requirement on $\TV$ would have led to the condition
$\frac{\beta^2(\omega_{ij})
\gamma^2(\omega_{ij})}{\delta(\omega_{ij})} = \omega_{ij}$.  Together
with $\beta(\omega_{ij}) \gamma^2(\omega_{ij}) = \omega_{ij}$ from
point~\ref{item:Laplacians} this gives $\delta=\beta$ as we have
assumed all along.}.

\item\label{item:TVa} To compute $\TVa$ we use the bound on the
maximum norm of $\varphi$ to find
\begin{align*}
\langle \nabla u, \varphi\rangle_{\mathcal{E}} &= \frac12 \sum_{i,j\in
I_m} \varphi_{ij} (u_j-u_i) \beta(\omega_{ij}) \gamma(\omega_{ij})
\\
&
\leq \frac12 \sum_{i,j \in I_m} |\varphi_{ij}| |u_i-u_j|
\beta(\omega_{ij})\gamma(\omega_{ij})\\
&\leq \frac12 \sum_{i,j \in I_m} |u_i-u_j|
\beta(\omega_{ij})\gamma(\omega_{ij}).
\end{align*}
To achieve equality we can choose $\varphi = \varphi^a:=\sgn(\nabla u)$,
\textit{i.e.,} $\varphi_{ij}^a = \sgn(u_j-u_i)$ for $i,j$ such that 
$\gamma(\omega_{ij})>0$ and $\varphi_{ij}^a=0$ otherwise\footnote{Note that we can change $\varphi^a$ on the set of vertices for which $\nabla u = 0$ without losing equality. See also footnote \ref{foot:nonuniquevarphi}.}. Hence
\[
\TVa(u) = \frac12 \sum_{i,j \in I_m} |u_i-u_j|
\beta(\omega_{ij})\gamma(\omega_{ij}).
\]
If we now choose $\beta(\omega_{ij})=\omega_{ij}^{2q-1}$ and
$\gamma(\omega_{ij}) = \omega_{ij}^{1-q}$ then $\TVa={\TVa}_q$ while
$\beta$ and $\gamma$ satisfy the necessary condition
$\beta(\omega_{ij}) \gamma^2(\omega_{ij}) = \omega_{ij}$ from the
previous points.
\end{enumerate}
These choices for $\alpha$, $\beta$, and $\gamma$ lead to the inner
products (or semi-definite sesquilinear forms), operators and
functions presented in Section~\ref{sec:LaplDirTV}.

It is interesting to consider the conditions under which $\|\nabla
u\|_{\mathcal{E}}=0$ and $\|\nabla u\|_i=0$.
\begin{align*}
\|\nabla u\|_{\mathcal{E}}=0 \Leftrightarrow \sum_{i,j\in I_m}
\omega_{ij} (u_i-u_j)^2 \Leftrightarrow \forall (i,j)\in I_m^2\,\,
[\omega_{ij}=0 \vee u_i=u_j].
\end{align*}
This means that $\|\nabla u\|_{\mathcal{E}}=0$ iff $u$ is constant on
connected components of the graph.  Similarly
\begin{align*}
\|\nabla u\|_i=0 \Leftrightarrow \sum_{j\in I_m} \omega_{ij}
(u_i-u_j)^2 \Leftrightarrow \forall j\in I_m\,\, [\omega_{ij}=0 \vee
u_i=u_j],
\end{align*}
hence $\|\nabla u\|_i=0$ iff $u$ is constant on the set $\{n_j\in V:
j=i \vee e_{ij}\in E\}$ consisting of neighboring vertices of the
$i^{\text{th}}$ vertex plus the $i^{\text{th}}$ vertex itself.

We see that these conditions do what we would hope and expect them to
do, even if the choice of $q$ has made the $\mathcal{E}$-sesquilinear
form semi-definite, \textit{i.e.,} $\|\nabla u\|_{\mathcal{E}}=0$ gives
global (per connected component) constants and $\|\nabla u\|_i=0$
gives local constancy.

\section{Deferred proofs}\label{sec:deferredproofs}

The next lemma was used in the proof of Lemma~\ref{lem:klowerbound}.

\begin{lemma}\label{lem:lemmaforklowerbound}
Let $\{u_n\}_{n=1}^\infty \subset L^1(\T)$ be a such that $u_n \to u$
in $L^1(\T)$ for a $u\in L^1(\T^2)$ and $u_n' \rightharpoonup v$ in
$L^2(\T)$ for a $v\in L^2(\T)$.  Then $v=u'$ (and thus $u\in
W^{1,2}(\T)$).
\end{lemma}
Note in the proof below that this result in fact does not depend on
the dimension and holds on $\T^n$.

\begin{proof}[\bf Proof of Lemma~\ref{lem:lemmaforklowerbound}]
Let $\varphi\in C_c^\infty(\T)$, then
\[
0 = \underset{n\to\infty}\lim\, \int_{\T} \varphi (v-u_n') =
\underset{n\to\infty}\lim\, \int_{\T} \left[ \varphi v - \varphi' u -
\varphi u_n' + \varphi' u\right] .
\]
There is a $C>0$ such that
\[
\underset{n\to\infty}\lim\, \int_{\T} \left[-\varphi' u - \varphi u_n'
\right] = \underset{n\to\infty}\lim\, \int_{\T} \varphi' \left[u_n -u
\right] \leq \underset{n\to\infty}\lim\, C \|u_n-u\|_{L^1(\T)} = 0,
\]
hence we conclude $\int_{\T} \varphi v = -\int_{\T} \varphi' u$.
\end{proof}

The next lemma is a discrete Rellich-Kondrachov type compactness
result used in proof of Theorem~\ref{thm:compactnessforkNe}.

\begin{lemma}[Discrete Rellich-Kondrachov compactness
result]\label{lem:discreteRK}
Let $\{u_n\}_{n\!=\!1}^\infty $ $ \subset L^2(\T^2)$ and $\{N_n\}_{n=1}^\infty
\subset (0,\infty)$ be sequences such that as $n\to \infty$ we have
$N_n \to \infty$, $u_n \rightharpoonup u$ in $L^2(\T^2)$ for some
$u\in L^2(\T^2)$, and the difference quotients $($see
$(\ref{eq:differencequotient})) $ $\|D_{N_n}^k u_n\|_{L^2(\T^2)}$
$( k\in\{1,2\} ) $ are uniformly bounded.  Then $u_n \to u$ in
$L^2(\T^2)$.
\end{lemma}

\begin{proof}[\bf Proof.]
For $\e > 0$ let $u_n^\e := J_\e u_n \in C^\infty(\T^2)$ be a
mollified function on the torus, defined to be the solution to the
heat equation after time $\e$ with initial condition $u_n$:
\[
J_\e u_n(x) = \sum_{k\in \Z^2} \widehat{u_n}(k) e^{-\e^2 |k|^2 + 2 \pi
i k\cdot x},
\]
where $\widehat{u_n}(k) = \int_{\T^2} u_n(x) e^{-2\pi i k\cdot x}
\,dx$.  We proceed in two steps.  First we need to prove some
properties of the mollifier, then we will prove the statement of the
lemma.

\smallskip

{\bf Step 1:} From \cite[Appendix B]{Greer03},
\cite[Lemma1]{GreerBertozzi04} we get that there is a $C>0$ such that
$\|u_n^\e\|_{L^2(\T^2)} \leq C \|u_n\|_{L^2(\T^2)}$.  Additionally,
$J_\e$ is a linear operator.  \cite{Greer03,GreerBertozzi04} also give
$\|J_\e f - f\|_{L^2(\T^2)} \leq \e \|f\|_{H^1(\T^2)}$ for $f\in
H^1(\T^2)$.  $u_n$ is not regular enough to use this estimate, so we
need a discrete version of this.

As in the references above, using that $1-e^{-\theta^2} \leq
|\theta|^2$ for $\theta \in \C$, we find that
\[
\Big  | \frac{ (1-e^{-\e^2 |k|^2} )^2}{1+|k|^2} \Big | <
\left\{ \begin{array}{ll} \e^2 \delta^2 & \text{ if } |k|<\delta,\\
\delta^{-2} & \text{ if } |k| \geq \delta.  \end{array}\right.
\]
Hence,
\begin{align*}
\| u_n^\e - u_n\|_{L^2(\T^2)}^2 &= \sum_{k\in \Z^2} \big  (1-e^{-\e^2
|k|^2}\big  )^2 |\widehat{u_n}(k)|^2  \\
&
\leq \Big  ( \underset{k\in
\Z^2}\sup\Big  | \frac{  (1-e^{-\e^2
|k|^2}  )^2}{1+|k|^2}\Big  | \Big  ) \sum_{k\in \Z^2}
\left(1+|k|^2\right) |\widehat{u_n}(k)|^2\\
&\leq \e^2 \sum_{k\in \Z^2} \left(1+|k|^2\right) |\widehat{u_n}(k)|^2.
\end{align*}
By Plancherel's/Parseval's identity we get immediately $\sum_{k\in
\Z^2} |\widehat{u_n}(k)|^2 = \|u_n\|_{L^2(\T^2)}$.  Furthermore,
\begin{align*}
\| D_{N_n} u_n\|_{L^2(\T^2)}^2 &= \int_{\T^2} \Big  [ \big  (D_{N_n}^1
u_n(x)\big  )^2 + \left(D_{N_n}^2 u_n(x)\right)^2 \Big  ] \,dx\\
 &= N_n^2 \sum_{k\in \Z^2} \Big  ( \Big  |\widehat{(u_n^{+1}-u_n)}
 (k)\Big  |^2 + \Big  |\widehat{(u_n^{+2}-u_n)} (k)\Big  |^2 \Big  ),
\end{align*}
where $u_n^{+j}(x) := u_n(x+N_n e_j)$ for standard basis vectors
$e_j$, $j\in \{1,2\}$.  It's easily computed that
$\widehat{u_n^{+j}}(k) = \widehat{u_n}(k) e^{2 \pi i k_j/N_n}$, hence
\[
\| D_{N_n} u_n\|_{L^2(\T^2)}^2 = \sum_{k\in \Z^2} |\widehat{u_n}(k)|^2
\Big  [ \Big  | N_n  ( e^{2 \pi i k_1/N_n} - 1 ) \Big  |^2 +
\Big | N_n  ( e^{2 \pi i k_2/N_n} - 1 )\Big |^2 \Big  ].
\]
Recognizing the difference quotient 
$$
N_n \big ( e^{2 \pi i k_1/N_n} -
1\big ) = 2 \pi i k_1 \frac{e^{2 \pi i k_1/N_n} - e^0}{2 \pi i
k_1/N_n} = 2\pi i k_1 + \mathcal{O}(N_n^{-1}) ,
$$
 we deduce
\[
\| D_{N_n} u_n\|_{L^2(\T^2)}^2 = 4 \pi^2 \sum_{k\in \Z^2} |k|^2
|\widehat{u_n}(k)|^2 + C_n \sum_{k\in \Z^2} |\widehat{u_n}(k)|^2,
\]
where $C_n = \mathcal{O}(N_n^{-1})$.  On the last term we can use
again Parseval's formula.  By the uniform bounds on
$\|u_n\|_{L^2(\T^2)}$ and $\|D_{N_n} u_n\|_{L^2(\T^2)}$ we then find
that
\begin{equation}\label{eq:uniforminnunune}
\| u_n^\e - u_n\|_{L^2(\T^2)}^2 \leq C \e^2, \quad \text{uniformly in
} n \text{ for } n \text{ large enough}.
\end{equation}
Next we compute an estimate on the derivatives of $u_n^\e$.
$$
\frac{\partial}{\partial x_1} u_n^\e(x) = 2 \pi i \sum_{k\in \Z^2}
k_1 \hat u_n(k) e^{-\e^2 |k|^2 + 2 \pi i k\cdot x}
$$
 and hence, since
$|\hat u_n(k)| \leq \|u_n\|_{L^1(\T^2)}$,
\[
\Big |\frac{\partial}{\partial x_1} J_\e u_n(x)\Big | \leq 2 \pi
\|v\|_{L^1(\T^2)} \sum_{k\in \Z^2} |k_1| e^{-\e^2 |k|^2}.
\]
We compute
\[
\sum_{k_2 \in \Z} e^{-\e^2 k_2^2} \leq 2 \sum_{k_2=0}^\infty e^{-\e^2
k_2^2} \leq 2 \int_0^\infty e^{-\e^2 k_2^2} \,dk_2 + 1= \sqrt{\pi}
\e^{-1} + 1
\]
and
\[
\sum_{k_1 \in \Z} |k_1| e^{-\e^2 k_1^2} \! = \! 2 \! \sum_{k_1=0}^\infty k_1
e^{-\e^2 k_2^2} \! \leq \! 2 \! \int_0^\infty \! \!  k_1 e^{-\e^2 k_1^2}  dk_1 
\! =\! 
\e^{-2} \int_0^\infty x^2 e^{-x^2} dx \! = \! \e^{-2}.
\]
Because $\|u_n\|_{L^1(\T^2)} \leq \|u_n\|_{L^2(\T^2)}$ is uniformly
bounded, we conclude (for $\e$ small enough) there is a $C>0$ such
that $\|\nabla J_\e u_n(x)\cdot e_k\|_{L^\infty(\T^2)} \leq C
\e^{-3}$, $k\in\{1,2\}$.

\smallskip

{\bf Step 2:} Let $\eta>0$ and let $n$ be large enough for the
bounds proved in Step 1 to hold.  Fix $\e>0$ small enough such that,
by (\ref{eq:uniforminnunune}), for each $n$ we have
$\|u_n-u_n^\e\|_{L^2(\T^2)} \leq \eta/3$.

By the bounds from Step 1 both $\|u_n^\e\|_{L^2(\T^2)}$ and $\|\nabla
J_\e u_n(x)\cdot e_k\|_{L^2(\T^2)}$ are uniformly (in $n$, for fixed
$\e$) bounded, and hence by the Rellich-Kondrachov compactness theorem
\cite[\S5.7 Theorem 1]{Evans02}, \cite[Theorem 6.2]{Adams75} the
sequence $\{u_n^\e\}_{n=1}^\infty$ converges strongly in $L^2(\T^2)$
as $n\to \infty$.  In particular it is a Cauchy sequence in
$L^2(\T^2)$, so choose $M_\e>0$ such that for all $n, m > M_\e$ we
have $\|u_n^\e - u_m^\e\|_{L^2(\T^2)} \leq \eta/3$.  Then for such
$n,m$
\[
\|u_n - u_m\|_{L^2(\T^2)} \leq \|u_n-u_n^\e\|_{L^2(\T^2)} +
\|u_m-u_m^\e\|_{L^2(\T^2)} + \|u_n^\e-u_m^\e\|_{L^2(\T^2)} \leq \eta.
\]
Thus, $\{u_n\}_{n=1}^\infty$ is a Cauchy sequence in $L^2(\T^2)$ and
therefore converges strongly in $L^2(\T^2)$.  By the uniqueness of the
limit it converges to $u$.
\end{proof}

\begin{proof}[\bf Proof of Lemma~\ref{lem:unifconvofweights}]
In both cases we assume without loss of generality that for $N$ is
large enough such that $\frac{L}N<\frac12$.
For fixed $N>0$ and $i,j,k,l\in I_{N^2}$, let $x=(i/N, j/N)$ and
$y=(k/N, j/N)$.  For $z\in \T^2$ write $f^{x,y}(z) =
\big(f(x-z)-f(y-z)\big)^2$, then from (\ref{eq:dLN})
$$
(d_{L,N}^2)_{i,j,k,l}= \sum_{r,s=-L}^L f^{x,y}(r/N, s/N) . 
$$
 Also
define $S_{L,N}:= \{z\in \T^2: |z_1|+|z_2| \leq L/N\}$.  By repeated
use of the trapezoidal rule for approximating integrals we then find
\begin{align}
N^{-2} (d_{L,N}^2)_{i,j,k,l} &= \int_{S_{L,N}} f^{x,y}(z)\,dz -
\frac{L}{6 N^3}\int_{-L/N}^{L/N} \frac{\partial^2 f^{x,y}}{\partial
z_2^2}(z_1, \zeta_2)\,dz_1  \notag
\\
&
\qquad - \frac{L}{6N^4} \sum_{s=-L}^L
\frac{\partial^2 f^{x,y}}{\partial z_1^2}(\zeta_1, s/N)\notag\\
&=: \int_{S_{L,N}} f^{x,y}(z)\,dz + R_{L,N},\label{eq:itsatrap}
\end{align}
where $(\zeta_1, \zeta_2) \in S_{L,N}$.  By smoothness of $f$ and
compactness of $\T^2$ we have $|R_{L,N}| \leq C_f \frac{L^2}{N^4}$ for
some constant $C_f>0$ depending on $f$.

For the first statement in the lemma we now find
\begin{align*}
(d_{L,N}^2)_{i,j,k,l} & =  \! N^2  \! \int_{S_{L,N}}  \!  \!  \!  \!  f^{x,y}(z) dz
 +  \!  N^2
R_{L,N}   \! 
=  \! \frac{4 L^2}{|S_{L,N}|} \!  \int_{S_{L,N}} 
 \!  \! 
f^{x,y}(z)\,dz + N^2
R_{L,N}\\
 &\to 4 L^2 f^{x,y}(0) = 4 L^2 \big(f(x)-f(y)\big)^2 \quad
 \text{uniformly as } N\to\infty.
\end{align*}
The uniformity of the convergence follows by the bound of the smooth
$f$ on the compact domain $\T^2$.  This proves the claim (since the
composition of a continuous function and a uniformly converging
sequence of functions is uniformly converging to the composition of
the continuous function and the limit of the sequence).

For the second statement the bound on $f$ allows us to conclude that
\[
\int_{S_{L,N}} f^{x,y}(z)\,dz \to \int_{S_\ell} f^{x,y}(z)\,dz \quad
\text{uniformly as } N\to\infty.
\]
The claim then follows by taking the limit $N\to\infty$ in
(\ref{eq:itsatrap}).
\end{proof}

\def\cprime{$'$} \def\cprime{$'$}


\begin{thebibliography}{99}


\bibitem{Adams75}
R.~A. Adams, ``Sobolev spaces'', Pure and applied mathematics; a
series of monographs and textbooks 65, first~ed., Academic Press, Inc,
New York, 1975.

\bibitem{Alberti01a}
G.~Alberti, {\it Un risultato di convergenza variazionale per
funzionale di tipo {G}inzburg-{L}andau in dimensione qualunque,}
Bollettino U. M. I., 8 (2001), 289--310.

\bibitem{Alberti01b}
G.~Alberti, {\it A variational convergence result for functionals of
{G}inzburg-{L}andau type in any dimension,} Note presented at the XVI
Congress of the Italian Mathematical Union.  English version of
\cite{Alberti01a}.

\bibitem{AlbertiBaldoOrlandi05}
G.~Alberti, S.~Baldo, and G.~Orlandi, {\it Variational convergence for
functionals of {G}inzburg-{L}andau type,} Indiana University
Mathematics Journal, 54 (2005), 1411--1472.

\bibitem{AlicandroCicalese04}
R.~Alicandro and M.~Cicalese, {\it A general integral representation
result for continuum limits of discrete energies with superlinear
growth,} SIAM J. Math.  Anal., 36 (2004), 1--37 (electronic).

\bibitem{AlicandroCicaleseGloria07}
R.~Alicandro, M.~Cicalese, and A.~Gloria, {\it Mathematical derivation
of a rubber-like stored energy functional,} C. R. Math.  Acad.  Sci.
Paris, 345 (2007), 479--482.

\bibitem{AlicandroBraidesCicalese08}
R.~Alicandro, A.~Braides, and M.~Cicalese, {\it Continuum limits of
discrete thin films with superlinear growth densities,} Calc.  Var.
Partial Differential Equations, 33 (2008), 267--297.

\bibitem{AlicandroCicaleseGloria10}
R.~Alicandro, M.~Cicalese, and A.~Gloria, {\it Integral representation
results for energies defined on stochastic lattices and application to
nonlinear elasticity,} Arch.  Rational Mech.  Anal., Published online:
05 October 2010.
  
\bibitem{AmbrosioFuscoPallara00}
L.~Ambrosio, N.~Fusco, and D.~Pallara, ``Functions of Bounded
Variation and Free Discontinuity Problems'', first~ed., Oxford
Mathematical Monographs, Oxford University Press, Oxford, 2000.
  
\bibitem{Baldo90}
S.~Baldo, {\it Minimal interface criterion for phase transitions in
mixtures of {C}ahn-{H}illiard fluids,} Ann.  Inst.  Henri Poincar\'e,
7 (1990), 67--90.

\bibitem{BarrosoFonseca94}
A.~C. Barroso and I.~Fonseca, {\it Anisotropic singular perturbations
- the vectorial case,} Proceedings of the Royal Society of Edinburgh,
124 (1994), 527--571.

\bibitem{BelkinNiyogi07}
M.~Belkin and P.~Niyogi, {\it Convergence of {L}aplacian eigenmaps,} 
Adv. Neural Inf. Process. Syst., 19 (2007), 129--137.

\bibitem{BelkinNiyogi08}
M.~Belkin and P.~Niyogi, {\it Towards a theoretical foundation for 
{L}aplacian-based manifold methods,} J. Comput. System Sci., 
74 (2008), 1289--1308.

\bibitem{BertozziFlenner12}
A.~L. Bertozzi and A.~Flenner, {\it Diffuse interface models on graphs
for analysis of high dimensional data,} Accepted in Multiscale
Modeling and Simulation, (2012).

\bibitem{Braides02}
A.~Braides, ``{$\Gamma$}-convergence for Beginners'', Oxford Lecture
Series in Mathematics and its Applications, 22, first~ed., Oxford
University Press, Oxford, 2002.

\bibitem{Braides06}
A.~Braides, ``A handbook of {$\Gamma$}-convergence'', Handbook of
differential equations.  Stationary partial differential equations,
III, first~ed., ch.~2, 101--214, Elsevier, 2006.

\bibitem{BraidesGelli02}
A.~Braides and M.~S. Gelli, {\it Continuum limits of discrete systems
without convexity hypotheses,} Math.  Mech.  Solids, 7 (2002), 41--66.

\bibitem{BraidesGelli06}
A.~Braides and M.~S. Gelli, {\it From discrete systems to continuous
variational problems: an introduction,} in ``Topics on concentration
phenomena and problems with multiple scales'', Lect.  Notes Unione
Mat.  Ital., 2 (2006), 3--77.

\bibitem{BraidesYip12}
A.~Braides and N.K.~Yip, {\it A quantitative description of mesh 
dependence for the discretization of singularly perturbed non-convex 
problems,} Accepted in SIAM J. Numer. Anal.

\bibitem{Bresson09}
X.~Bresson, {\it A short note for nonlocal {TV} minimization,} Note,
(2009).

\bibitem{BuadesCollMorel05}
A.~Buades, B.~Coll, and J.~M. Morel, {\it A review of image denoising
algorithms, with a new one,} Multiscale Model.  Simul., 4 (2005),
490--530.

\bibitem{CandoganMenacheOzdaglarParrilo11}
O.~Candogan, I.~Menache, A.~Ozdaglar, and P.~Parrilo, {\it Flow and
decompositions of games: harmonic and potential games,} Math.  Oper.
Res., 36 (2011), 474--503.
 
\bibitem{ChambolleGiacominiLussardi10}
A.~Chambolle, A.~Giacomini, and L.~Lussardi, {\it Continuous limits of 
discrete perimeters,} M2AN Math. Model. Numer. Anal., 44 (2010), 207--230.
 
\bibitem{ChoksivanGennipOberman11}
R.~Choksi, Y.~van Gennip, and A.~Oberman, {\it Anisotropic total
variation regularized {$L^1$} approximation and denoising/deblurring
of 2{D} bar codes,} Inverse Probl.  Imaging, 5 (2011), 591--617.
  
\bibitem{ChoksiRen05}
R.~Choksi and X.~Ren, {\it Diblock copolymer/homopolymer blends:
Derivation of a density functional theory,} Physica D, 203 (2005),
100--119.

\bibitem{Chung97}
F.~R.~K. Chung, ``Spectral graph theory'', CBMS Regional Conference
Series in Mathematics, 92, Published for the Conference Board of the
Mathematical Sciences, Washington, DC, 1997.

\bibitem{ContiFonsecaLeoni02}
S.~Conti, I.~Fonseca, and G.~Leoni, {\it A {$\Gamma$}-convergence
result for the two-gradient theory of phase transitions,}
Communications on Pure and Applied Mathematics, LV (2002), 857--936.

\bibitem{DalMaso93}
G.~Dal~Maso, ``An introduction to {$\Gamma$}-convergence'', Progress
in Nonlinear Differential Equations and Their Applications, 8,
first~ed., Birkh\"auser, Boston, 1993.

\bibitem{DeGiorgiFranzoni75}
E.~De~Giorgi and T.~Franzoni, {\it Su un tipo di convergenza
variazionale,} Atti Accad.  Naz.  Lincei Rend.  Cl.  Sci.  Mat.  Fis.
Natur., 58 (1975), 842--850.

\bibitem{ErdosRenyi59}
P.~Erd{\H{o}}s and A.~R{\'e}nyi, {\it On random graphs.  {I},} Publ.
Math.  Debrecen, 6 (1959), 290--297.

\bibitem{Evans02}
L.~C. Evans, ``Partial Differential Equations'', Graduate Studies in
Mathematics, 19, first~ed., American Mathematical Society, US, 2002.
  
\bibitem{EvansGariepy92}
L.~C. Evans and R.~F. Gariepy, ``Measure Theory and Fine Properties of
Functions'', first~ed., Studies in Advanced Mathematics, CRC Press
LLC, Boca Raton, Florida, 1992.

\bibitem{FonsecaTartar89}
I.~Fonseca and L.~Tartar, {\it The gradient theory of phase
transitions for systems with two potential wells,} Proceedings of the
Royal Society of Edinburgh, 111 (1989), 89--102.

\bibitem{GilboaOsher07}
G.~Gilboa and S.~Osher, {\it Nonlocal linear image regularization and
supervised segmentation,} Multiscale Model.  Simul., 6 (2007),
595--630.

\bibitem{GilboaOsher08}
G.~Gilboa and S.~Osher, {\it Nonlocal operators with applications to
image processing,} Multiscale Model.  Simul., 7 (2008), 1005--1028.
 
\bibitem{GineKoltchinskii06}
E.~Gin{\'e} and V.~Koltchinskii, {\it Empirical graph {L}aplacian 
approximation of {L}aplace-{B}eltrami operators: large sample results,} 
IMS Lecture Notes Monogr. Ser.: High dimensional probability, 51 (2006), 
238--259.
  
\bibitem{Giusti84}
E.~Giusti, ``Minimal Surfaces and Functions of Bounded Variation'',
Monographs in Mathematics, 80, first~ed., Birkh{\"a}user, Boston,
1984.

\bibitem{Greer03}
J.~B. Greer, {\it Fourth order diffusions for image processing,}
Thesis, Duke University, (2003).

\bibitem{GreerBertozzi04}
J.~B. Greer and A.~L. Bertozzi, {\it {$H^1$} solutions of a class of
fourth order nonlinear equations for image processing,} Discrete
Contin.  Dyn.  Syst., 10 (2004), 349--366.

\bibitem{HeinAudibertvonLuxburg05}
M.~Hein, J.-Y. Audibert, and U.~von Luxburg, {\it From graphs to 
manifolds---weak and strong pointwise consistency of graph 
{L}aplacians,}, Lecture Notes in Comput. Sci.: Learning Theory, 
3559 (2005), 470--485.

\bibitem{HeinAudibertvonLuxburg07}
M.~Hein, J.-Y. Audibert, and U.~von Luxburg, {\it Graph {L}aplacians
and their convergence on random neighborhood graphs,} J. Mach.  Learn.
Res., 8 (2007), 1325--1368 (electronic).

\bibitem{KohnSternberg89}
R.~V. Kohn and P.~Sternberg, {\it Local minimisers and singular
perturbations,} Proc.  Roy.  Soc.  Edinburgh Sect.  A, 111 (1989),
69--84.

\bibitem{MaiervonLuxburgHein11}
M.~Maier, U.~von Luxburg, and M.~Hein, {\it How the result of graph 
clustering methods depends on the construction of the graph,} 
Arxiv preprint arXiv:1102.2075, (2011).

\bibitem{ManfrediObermanSviridov12}
J.~J. Manfredi, A.~M. Oberman, A.~P. Sviridov, {\it Nonlinear elliptic 
partial differential equations and $p$-harmonic functions on graphs,} 
preprint (2012)

\bibitem{Modica87a}
L.~Modica, {\it The gradient theory of phase transitions and the
minimal interface criterion,} Arch.  Rational Mech.  Anal., 98 (1987),
123--142.

\bibitem{Modica87b}
L.~Modica, {\it Gradient theory of phase transitions with boundary
contact energy,} Ann.  Inst.  Henri Poincar{\'e}, 4 (1987), 487--512.

\bibitem{ModicaMortola77}
L.~Modica and S.~Mortola, {\it Un esempio di {$\Gamma$}-convergenza,}
Bollettino U.M.I., 5 (1977), 285--299.

\bibitem{Neuberger06}
J.~M. Neuberger, {\it Nonlinear elliptic partial difference equations
on graphs,} Experiment.  Math., 15 (2006), 91--107.

\bibitem{NgJordanWeiss02}
A.~Ng, M.~Jordan, and Y.~Weiss, {\it On spectral clustering:
{A}nalysis and an algorithm,} in ``{A}dvances in Neural Information
Processing Systems'', 14 (2002), 849--856.

\bibitem{PeletierRoeger09}
M.~A. Peletier and M.~R{\"o}ger, {\it Partial localization, lipid
bilayers, and the elastica functional,} Arch.  Ration.  Mech.  Anal.,
193 (2009), 475--537.

\bibitem{RenWei00}
X.~Ren and J.~Wei, {\it On the multiplicity of solutions of two
nonlocal variational problems,} SIAM J. Math.  Anal., 31 (2000),
909--924.

\bibitem{RudinOsherFatemi92}
L.~I. Rudin, S.~Osher, and E.~Fatemi, {\it Nonlinear total variation
based noise removal algorithms,} in ``Proceedings of the eleventh
annual international conference of the Center for Nonlinear Studies on
Experimental mathematics: computational issues in nonlinear science'',
(1992), 259--268.
    
\bibitem{ShiMalik00}
J.~Shi and J.~Malik, {\it Normalized cuts and image segmentation,}
IEEE Transactions on Patternal Analysis and Machine Intelligence, 22
(2000), 888--905.

\bibitem{Sternberg88}
P.~Sternberg, {\it The effect of a singular perturbation on nonconvex
variational problems,} Arch.  Rational Mech.  Anal., 101 (1988),
209--260.

\bibitem{SzlamBresson10}
A.~Szlam and X.~Bresson, {Total variation and {C}heeger cuts,} in
``Proceedings of the 27th International Conference on Machine Learning
(ICML-10)'', (2010), 1039--1046.

\bibitem{vonLuxburg07}
U.~von Luxburg, {\it A tutorial on spectral clustering,} Stat.
Comput., 17 (2007), 395--416.

\bibitem{vonLuxburgBelkinBousquet08}
U.~von Luxburg, M.~Belkin, and O.~Bousquet, {\it Consistency of 
spectral clustering,} Ann. Statist., 36 (2008), 555--586.

\bibitem{WardetzkyMathurKaelbererGrinspun07}
M.~Wardetzky, S.~Mathur, F.~K{\"a}elberer, and E.~Grinspun, {\it
Discrete {L}aplace operators: No free lunch,} Eurographics Symposium
on Geometry Processing (ed.  A. Belyaev, M. Garland), (2007).

\end{thebibliography}
\end{document}